\documentclass[11pt, letterpaper,english]{amsart}
\usepackage[left=1in,right=1in,bottom=1in,top=1in]{geometry}
\usepackage{amsmath,amsthm,amssymb}
\usepackage{mathrsfs}

\usepackage[all,cmtip]{xy}
\usepackage{cite}
\usepackage{hyperref}
\usepackage{enumerate}
\usepackage{graphicx}
\usepackage{booktabs}
\usepackage[usenames,dvipsnames]{color}
\usepackage{colonequals}   

\usepackage{subcaption}

\usepackage[T1]{fontenc}
\usepackage{array}
\usepackage{makecell}
\newcolumntype{x}[1]{>{\centering\arraybackslash}p{#1}}

\usepackage{tikz}
\newcommand\diag[4]{%
  \multicolumn{1}{p{#2}|}{\hskip-\tabcolsep
  $\vcenter{\begin{tikzpicture}[baseline=0,anchor=south west,inner sep=#1]
  \path[use as bounding box] (0,0) rectangle (#2+2\tabcolsep,\baselineskip);
  \node[minimum width={#2+2\tabcolsep},minimum height=\baselineskip+\extrarowheight] (box) {};
  \draw (box.north west) -- (box.south east);
  \node[anchor=south west] at (box.south west) {#3};
  \node[anchor=north east] at (box.north east) {#4};
 \end{tikzpicture}}$\hskip-\tabcolsep}}

\theoremstyle{plain}
\newtheorem{lem}{Lemma}
\newtheorem{lemma}[lem]{Lemma}

\newtheorem{theorem}[lem]{Theorem}
\newtheorem{prop}[lem]{Proposition}

\newtheorem{cor}[lem]{Corollary}
\newtheorem{corollary}[lem]{Corollary}
\newtheorem{conjecture}[lem]{Conjecture}
\newtheorem{fact}[lem]{Fact}

\newtheorem{philosophy}[lem]{Philosophy}

\theoremstyle{definition}
\newtheorem{defn}[lem]{Definition}
\newtheorem{definition}[lem]{Definition}
\newtheorem{example}[lem]{Example}
\newtheorem{remark}[lem]{Remark}

\newtheorem{notation}[lem]{Notation}

\numberwithin{equation}{section}
\numberwithin{lem}{section}

\newcommand{\mathfont}{\mathbf}

\newcommand{\C}{\mathfont C}
\newcommand{\Z}{\mathfont Z}
\newcommand{\ZZ}{\mathfont Z}

\newcommand{\Q}{\mathfont Q}
\newcommand{\QQ}{\mathfont Q}

\newcommand{\F}{\mathfont F}
\newcommand{\FF}{\mathfont F}

\newcommand{\bQ}{\mathfont{Q}}

\newcommand{\cT}{\mathcal{T}}

\DeclareFontFamily{OT1}{rsfs}{}
\DeclareFontShape{OT1}{rsfs}{n}{it}{<-> rsfs10}{}
\DeclareMathAlphabet{\mathscr}{OT1}{rsfs}{n}{it}

\newcommand{\Gscr}{\mathscr{G}}

\newcommand{\sS}{\mathscr{S}}
\newcommand{\Pscr}{\mathscr{P}}
\newcommand{\Oscr}{\mathscr{O}}

\DeclareMathOperator{\Hom}{Hom}

\DeclareMathOperator{\coker} {coker}

\DeclareMathOperator{\Aut}{Aut}
\DeclareMathOperator{\Gal}{Gal}
\DeclareMathOperator{\tr}{tr}

\DeclareMathOperator{\ord}{ord}

\newcommand{\kbar}{\bar{k}}

\DeclareMathOperator{\Ht}{ht}

\renewcommand{\ll}{\mathrm{ll}}
\newcommand{\mult}{\mathrm{mult}}

\newcommand{\et}{\text{\'et}}

\newcommand{\PP}{\mathfont{P}}

\DeclareMathOperator{\Pic}{Pic}

\newcommand{\dk}[2]{a^{#1}(#2)}
\newcommand{\D}{\mathbf{D}}

\DeclareMathOperator{\Ig}{Ig}
\newcommand{\moduli}{\Pscr}
\newcommand{\timing}{\textrm{time}}
\newcommand{\basic}{\textrm{basic}}
\newcommand{\modded}{\textrm{mod}}
\newcommand{\nilp}{\text{-}\mathrm{nil}}
\newcommand{\bij}{\text{-}\mathrm{bij}}

\DeclareMathOperator{\Cl}{Cl}

\newcommand{\smallmod}[1]{(\operatorname{mod} #1)}

\title{Iwasawa Theory for $p$-torsion Class Group Schemes in Characteristic $p$}

\author{Jeremy Booher}
\address{School of Mathematics and Statistics, University of Canterbury, Private Bag 4800, Christchurch 8140, New Zealand}
\email{jeremy.booher@canterbury.ac.nz}

\author{Bryden Cais}
\address{Department of Mathematics,
The University of Arizona,
617 N. Santa Rita Ave.
Tucson, AZ. 85721 USA}
\email{cais@math.arizona.edu}

\date{October 1, 2022}

\begin{document} 

\begin{abstract}
  We investigate a novel geometric Iwasawa theory for $\Z_p$-extensions of function fields over a perfect field $k$
  of characteristic $p>0$ by replacing the usual study of $p$-torsion in class groups with the study of 
  $p$-torsion class group {\em schemes}.  
  That is, if $\cdots \to X_2 \to X_1 \to X_0$ is the tower of curves over $k$ associated to a $\Z_p$-extension of function fields totally ramified over a finite non-empty set of places,
  we investigate the growth of the $p$-torsion group scheme in the Jacobian of $X_n$ as $n\rightarrow \infty$.
  By Dieudonn\'e theory, this amounts to studying the first de Rham cohomology groups of $X_n$ equipped
  with natural actions of Frobenius and of the Cartier operator $V$.  We formulate and test a number of conjectures
  which predict striking regularity in the $k[V]$-module structure of the space $M_n:=H^0(X_n, \Omega^1_{X_n/k})$ of global regular differential forms as $n\rightarrow \infty.$
  For example,  for each tower in a basic class of $\Z_p$-towers we conjecture that the dimension of the kernel of $V^r$ on $M_n$ is 
  given by $a_r p^{2n} + \lambda_r n + c_r(n)$ for all $n$ sufficiently large, where $a_r, \lambda_r$ are rational constants and $c_r : \Z/m_r \Z \to \Q$ 
  is a periodic function, depending on $r$ and the tower.
   To provide evidence for these conjectures, we collect extensive experimental data based on new and more efficient algorithms for working with differentials on $\Z_p$-towers of curves, and we prove our conjectures in the case $p=2$
  and $r=1$.
\end{abstract}

\maketitle

\section{Introduction}

\subsection{Geometric Iwasawa Theory} 
Fix a perfect field $k$ of characteristic $p>0$, and an algebraic function field $K$ in one variable over $k$.
Let $L/K$ be a Galois extension with $\Gamma:=\Gal(L/K)\simeq \Z_p$, the group of $p$-adic integers. 
We suppose that $L/K$ is unramified outside a finite
set of places $S$ of $K$ (which are trivial on $k$) and totally ramified at every place in $S$.\footnote{This last hypothesis may be achieved
by passing to a finite extension of $K$.  If $k$ is finite, all places of $K$ are automatically trivial on $K$.}  
Let $\Gamma_n:=p^n\Z_p$, and write $K_n = L^{\Gamma_n}$ for the fixed field of $\Gamma_n$.

In the spirit of classical Iwasawa theory, we seek to understand the growth of the $p$-primary part
of the class group of $K_n$ as $n$ grows.  When $L$ is the {\em constant} $\Z_p$-extension of $K$,
the regular growth of the class groups of $K_n$ was indeed Iwasawa's primary motivation for the eponymous theory
he initiated for number fields \cite{Iwasawa}.  
When $k$ is algebraically closed in $L$---which we assume henceforth---the growth of the class groups of
$K_n$ with $k$ finite has been studied by Mazur--Wiles \cite{mw83} and Crew \cite[\S3]{CrewIwasawa2} (for $S$ nonempty)
and by Gold-Kisilevsky \cite{gk88}. 
These works analyze the {\em physical} class group $\Cl_{K_n}$ of degree zero divisor classes {\em defined over $k$}
modulo linear equivalence, and prove---in perfect analogy with the number field setting---that 
when $k$ is a finite field, 
the Iwasawa module $\varprojlim_n \Cl_{K_n}[p^{\infty}]$
is finitely generated and torsion over $\Lambda:=\Z_p[\![\Gamma]\!]$, with no finite submodules.
The celebrated growth formula $\# \Cl_{K_n}[p^{\infty}] = p^{n\lambda+p^n\mu+\nu}$ for $n\gg 0$ follows.

In this function field setting, however, there is another, far more interesting {\em motivic} interpretation of ``class group''
provided by the Jacobian of the associated algebraic curve.  
Writing $X_n$ for the 
unique smooth, projective and geometrically connected curve over $k$ with function field $K_n$ (with $K_0 = K$ corresponding to $X_0$),
we obtain a {\em $\Z_p$-tower of curves}
\[
\cT : \cdots \to X_n \to \cdots \to X_2 \to X_1 \to X_0
\]
with $X_n\rightarrow X_0$ a branched $\Z/p^n\Z$-cover, unramified outside a finite set of points $S$ of $X_0$
and totally ramified over every point of $S$.
For each $n$, 
the Jacobian $J_{X_n}:=\Pic^0_{X_n/k}$ represents the functor of equivalence classes of degree zero divisors
on $X_n$, and is a rich algebro-geometric object with no analogue in the number field setting.
From this point of view, the $p$-primary part of the motivic class group is the full $p$-divisible (Barsotti--Tate) group 
$J_{X_n}[p^{\infty}]$, which is an inductive system of $p$-power group {\em schemes}. 
The $p$-primary part $\Cl_{K_n}[p^{\infty}]$ of the ``physical'' class group is none other than the
group of $k$-rational points of $J_{X_n}[p^{\infty}]$, which is only a very small piece of $J_{X_n}[p^{\infty}]$;
for example, when $X_0=\PP^1_k$ and $S=\{\infty\}$ (which is a prototypical case), the abelian group
$\Cl_{K_n}[p^{\infty}]$ is trivial, while the $p$-divisible group $J_{X_n}[p^{\infty}]$ has height ${2g_n}$
with $g_n$ the genus of $X_n$.

Our aim is to understand the structure---broadly construed---of 
the full $p$-divisible group $J_{X_n}[p^{\infty}]$ as $n\rightarrow \infty$.
Recent work provides some evidence
that there should be an Iwasawa theory for these
objects.  
By analyzing $L$-functions,
Davis, Wan and Xiao \cite{DWX16} prove that, for a certain class
of $\Z_p$-towers $\{X_n\}_{n\ge 0}$ with $X_0=\PP^1$ and $S=\{\infty\}$ (a class which we call ``basic''
in what follows; see Section~\ref{defn:basictower}), the isogeny type of $J_{X_n}[p^{\infty}]$ over $\kbar$
behaves in a remarkably regular way as $n$ grows (c.f. \cite{KostersZhu,rwxy,Xiang,KMU}).  However,
{\em isogeny type} is a somewhat coarse invariant, as it loses all touch with {\em torsion} phenomena.
As a first and critical step towards understanding this more subtle torsion in the full $p$-divisible group, we will investigate the
$p$-torsion group schemes $J_{X_n}[p]$
which are polarized ``1-truncated Barsotti--Tate groups''. 
These objects have a rich and extensive history, yet despite being the focus
of much research (e.g.~\cite{Oort99,PriesUlmer}) remain rather mysterious.
The goal of this paper is to provide 
evidence---both theoretical and computational---for the following
Iwasawa-theoretic principle:

\begin{philosophy} \label{philosophy}
    For any $\Z_p$-tower of curves $\{X_n\}_{n\ge 0}$,
    the $p$-torsion group schemes $J_{X_n}[p]$ behave in a ``regular'' way as $n\rightarrow \infty$.
\end{philosophy}

As a first approximation to $J_{X_n}[p]$, we will study the 
kernel of Frobenius $J_{X_n}[F]$.  
Note the quotient of $J_{X_n}[p]$ by $J_{X_n}[F]$ is canonically isomorphic to the Cartier dual of $J_{X_n}[F]$.
In this way, knowledge of $J_{X_n}[F]$ determines $J_{X_n}[p]$ up to a single extension. 
The virtue of focusing attention on $J_{X_n}[F]$ is that it can be understood {\em explicitly} via differentials on the
curve $X_n$.  
Indeed, 
the group scheme $J_{X_n}[F]$ functorially determines and is determined by its contravariant Dieudonn\'e module,
which by a theorem of Oda \cite{Oda}, is naturally identified with
the $k[V]$-module $M_n:=H^0(X_n,\Omega^1_{X_n/k})$
of global regular 1-forms on $X_n$, with $V$ acting as the Cartier operator.
Thus, to analyze the growth of the group schemes $J_{X_n}[F]$,
we will study the $k[V]$-module structure of $M_n$ as $n$ grows.  
In this paper, we develop efficient algorithms to compute with differentials on $\Z_p$-towers in order to provide computational evidence for Philosophy~\ref{philosophy} and we prove instances of the philosophy when $p=2$.

Let us describe our contributions in more detail.
For each $n$, Fitting's Lemma gives a natural direct sum decomposition of $k[V]$-modules
\begin{equation}
    M_n=H^0(X_n,\Omega^1_{X_n/k}) = M_n^{V\nilp}\oplus M_n^{V\bij},\label{eq:fitting}
\end{equation}
where $M_n^{V\nilp}$ (respectively $M_n^{V\bij}$) is the maximal $k[V]$-submodule
on which $V$ is nilpotent (respectively bijective).  As the $\Z_p$-tower is totally ramified over the set $S$, the {\em Deuring--Shafarevich formula} \cite{Subrao}
provides a dimension formula for the $p$-rank
\begin{equation}
    d_n:=\dim_k M_n^{V\bij} = p^n (d_0 + |S| - 1) - (|S| -1)
    \label{eq:DSformula}
\end{equation}
which is an instance of Philosophy \ref{philosophy}.
Moreover, one has an isomorphism of $\kbar[V]$-modules 
\[
    M_n^{V\bij}\otimes_k \kbar \simeq (\kbar[V]/(V-1))^{d_n},
\]
which with \eqref{eq:DSformula} provides a nearly complete understanding
of the behavior of $M_n^{V\bij}$ as $n$ grows.

As for the $V$-nilpotent part, taken together the Riemann--Hurwitz and Deuring--Shafarevich
formulae yield the dimension formula
\begin{equation} \label{eq:nilformula}
    \dim_k M_n^{V\nilp} = (g_n - d_n)  = 
    p^n (g_0-d_0) + \frac{1}{2}(p-1) \sum_{Q \in S} \sum_{i=1}^n p^{i-1} (s_Q(i) - 1) 
\end{equation}
where $s_Q(i)$ is the $i$-th break in the upper ramification filtration of $\Gamma$ at $Q\in S$
and $g_n$ is the genus of $X_n$.  
As every point in $S$ must be wildly ramified and the very nature of wild ramification forces
$s_Q(i+1) \ge p s_Q(i) $ for all $Q$ and $i$,
if $S$ is nonempty there is a lower bound of the form $g_n \ge c p^{2n}$ with $c > 0$;
see \cite[Theorem 1]{gk88} and {\em cf.} \cite[Theorem 1.1]{kw18} and \cite{kw19}.
In fact, it follows from class field theory (see \cite[Remark 3]{gk88})
that, for {\em any} sequence $\{s_i\}$ of positive integers satisfying $s_{i+1}\ge p s_i$,
there exists a $\Gamma$-tower $\{X_n\}$ with $X_0=\PP_k^1$ and $S=\{\infty\}$ in which $s_Q(i) \ge s_i$.
In other words, the dimension of $M_n^{V\nilp}$ can grow {\em arbitrarily} fast!

In order to have any hope of identifying regular structure in $M_n^{V\nilp}$ as $n\rightarrow\infty$,
we will therefore restrict our attention to towers in which the upper ramification breaks behave in a regular way.
For the purposes of this introduction --- and in much of this paper --- we will focus on the class of {\em basic}
$\Z_p$-towers over $k=\FF_p$ with ramification invariant $d$, given by the Artin--Schreier--Witt equation
\[
 Fy - y = \sum_{\substack{i=1\\ p\nmid i}}^d [c_i x^i]
\]
for $c_i \in \F_p$ and $c_d \neq 0$; see Section~\ref{ss:ASW} and Definition~\ref{defn:basictower}.  
Each such tower has base curve $X_0=\PP^1$ and $S=\{\infty\}$, with $s_{\infty}(n)=dp^{n-1}$ for $n\ge 1$,
so repeated applications of Riemann--Hurwitz 
shows that such towers are {\em genus stable} \cite{kw18}, in the sense that the
genus of the $n$th curve $X_n$ is given by a quadratic polynomial in $p^n$ with rational coefficients for $n\gg 0$.
Explicitly:
\begin{equation} \label{eq:genus}
g_n = \frac{d}{2(p+1)} p^{2n}  -\frac{1}{2}p^n + \frac{p+1-d}{2(p+1)} \quad \text{for} \quad n \ge 0
\end{equation}
which is very much in the spirit of \eqref{eq:DSformula}
and provides another validation of Philosophy \ref{philosophy}.
Note that any basic $\Z_p$-tower has $M_n^{V\bij}=0$ so $\dim_k M_n^{V\nilp}=\dim_k H^0(X_n,\Omega^1_{X_n/k}) = g_n$.

As $V$ is nilpotent on $M_n^{V\nilp}$, 
for each $n$ the $k[V]$-module structure of $M_n^{V\nilp}$
is completely determined by the sequence of positive integers 
\[
    a^{(r)}_n:=\dim_k \ker \left( V^r : H^0(X_n,\Omega^1_{X_n/k}) \rightarrow H^0(X_n,\Omega^1_{X_n/k})\right).
\]
The integer $a_n:=a_n^{(1)}$ is the {\em $a$-number} of the curve $X_n$,
and has been studied extensively
\cite{CMHyper,Fermat,ReBound,ElkinPriesanum1,Johnston,ElkinCyclic,Suzuki,Dummigan,FermatHurwitz,Frei,zhou19,bc20,minimal}.
For any {\em fixed} $n$ and $r$ sufficiently large, $V^r$ is zero on $M_n^{V\nilp}$
so \eqref{eq:nilformula} gives a formula for $a_n^{(r)}$
in such cases.  This relies on the Riemann--Hurwitz and 
Deuring--Shafarevich formulae; there is no analogous 
formula for the $a$-number.
Indeed, as $p$-groups are solvable the essential instances of the Riemann--Hurwitz and Deuring--Shafarevich 
formulae are for a branched $\ZZ/p\ZZ$-cover $Y\rightarrow X$ of smooth projective curves over $k$,
and in general the $a$-number of $Y$
can not be determined by the $a$-number of $X$ and the ramification data of the covering.
While \cite{bc20} does provide {\em bounds} on the $a$-number of $Y$ that depend only on the $a$-number of $X$
and the ramification data, 
these bounds allow for considerable variation. 
For a basic $\Z_p$-tower $\cT$ with ramification invariant $d$, the bounds imply
\begin{equation}
\frac{1}{2}\left(1-\frac{1}{p}\right)\left(1 - \frac{1}{p^2}\right) +O(p^{-n}) \le \frac{a_n}{g_n} \le \frac{2}{3}\left(1-\frac{1}{2p}\right) + O(p^{-n})
 \label{naiveguess}
\end{equation}
as $n\rightarrow\infty$, with implicit constants depending only on $d$ and $p$.
If a basic $\Z_p$-tower behaves like a ``random'' sequence of $\Z/p\Z$-covers, we might guess that $a_n$ is asymptotically $\frac{1}{2}(1-p^{-1})(1-p^{-2})\cdot g_n$,  since $a$-numbers of random $\Z/p\Z$-covers 
experimentally seem to be close to the lower bound with high probability \cite[Remark 1.5(3)]{minimal}.

For any fixed basic $\Z_p$-tower $\{X_n\}_{n\ge 0}$ and integer $r$, 
to compute $a_n^{(r)}$ we must determine the matrix of $V^r$
and its kernel on the $g_n$-dimensional space
of holomorphic differentials of $X_n$.
As $g_n$ grows like $c p^{2n}$ with $c>0$ by \eqref{eq:genus},
such computations rapidly become intractable, even for small values of $p$.
A key contribution of the present paper is the development of much more efficient algorithms (implemented in {\sc Magma} \cite{bcgit}) for computing with differentials on a $\Z_p$-tower of curves in order to investigate the behavior of $a_n^{(r)}$.  After computing numerous examples, 
we are led to:

\begin{conjecture}\label{MC}
    Let $\{X_n\}_{n\ge 0}$ be a basic $\ZZ_p$-tower with ramification invariant $d$.
    For each positive integer $r$, there exists an integer $m > 0$,
    a rational number $\lambda$, and a periodic function $c: \ZZ/m\ZZ \rightarrow \QQ$ such that
    \[
        a_n^{(r)} = \alpha(r,p) dp^{2n} + \lambda n + c(n)\quad\text{with}\quad \alpha(r,p) := \frac{r}{2(p+1)\left(r+\frac{p+1}{p-1}\right)}
    \]
    for all $n$ sufficiently large.  If $D$ is the prime-to-$p$ part of the denominator of $\alpha(r,p)$ in lowest terms
    and $D>1$, then $m$ may be taken to be the multiplicative order 
    of $p^2$ modulo $D$. 
    When in addition $m=1$, we may take $\lambda=0$ and $c$ constant.
\end{conjecture}

\begin{remark}
We compute that $\alpha(1,p) = \frac{p-1}{4p(p+1)}$, so we may take $m=1$ and $\lambda =0$ when $r=1$.  In other words, we predict that the $a$-number of the $n$th level of a basic $\Z_p$-tower with ramification invariant $d$ is $\frac{p-1}{4 p (p+1)} d p^{2n} + c$ (with the constant $c$ depending on the tower) for $n \gg 0$.
\end{remark}

We are able to prove Conjecture \ref{MC} when $p=2$ and $r=1$:

\begin{theorem} [see Corollary~\ref{cor:anumberbasic}]
    Let $\{X_n\}_{n\ge 0}$ be a basic $\ZZ_2$-tower with ramification invariant $d$.
    Then for $n>1$
    \[
        a_n=a_n^{(1)} = \frac{d}{24}\cdot 2^{2n} + \frac{d + (-1)^{(d-1)/2}\cdot 3}{12}
        = \frac{d}{4}\frac{(2^{2n-1} +1)}{3} + \frac{(-1)^{(d-1)/2}}{4}.
    \]
\end{theorem}

\begin{example}
Igusa curves in characteristic two (rigidified using $\Gamma_1(5)$) form a basic $\Z_2$-tower $\{X_n\}_{n\geq 0}$.  We have
\[
g(X_n) = 2^{2n-2}-2^{n}+1 \quad \text{and} \quad a(X_n) = 2^{2n-4} \text{ for } n >1.
\]
See Example~\ref{ex:igusaproperties} and Example~\ref{ex:igusap2}.
\end{example}

\begin{remark}
 Conjecture \ref{MC} indicates that
 the na\"ive guess that a $\ZZ_p$-tower 
 behaves like a sequence of ``random'' $\ZZ/p\ZZ$-covers is {\em wrong},
 as  together with \eqref{eq:genus} it implies that $a_n/g_n$ approaches $\frac{1}{2}(1-p^{-1})$
 and {\em not} $\frac{1}{2}(1-p^{-1})(1-p^{-2})$ as the guess would predict.
 In other words, a basic $\ZZ_p$-tower has more structure than a ``random'' sequence of Artin-Schreier covers
 which force the $a$-numbers to be larger. 
\end{remark}

\begin{remark}
    For a basic $\Z_p$-tower $\{X_n\}_{n\ge 0}$ with ramification invariant $d$,
    and each $n\ge 1$, we have an isomorphism of $k[V]$-modules
    \[
        M_n^{V\nilp} = H^0(X_n,\Omega^1_{X_n/k}) \simeq \bigoplus_{i\ge 1} \left(\frac{k[V]}{V^i}\right)^{m_n(i)}
    \]
    for uniquely determined nonnegative integers $m_n(i)$.  Conjecture \ref{MC}
    implies that for each $i$, there exists an integer $\ell>0$, a rational number $\mu$, 
    and a periodic function $\gamma: \ZZ/\ell\ZZ\rightarrow \QQ$ such that
    \[
        m_n(i) =\beta(i,p)\frac{dp^{2n}}{p-1} + \mu n + \gamma(n)\quad\text{with}\quad 
         \beta(i,p)=\frac{1}{(i+\frac{p+1}{p-1})^3 - (i+\frac{p+1}{p-1})} 
    \]
    for all $n$ sufficiently large, which shows that the $k[V]$-module $H^0(X_n,\Omega^1_{X_n/k})$---and therefore
    the $F$-torsion in the motivic class group $J_{X_n}$---behaves in an astonishingly regular manner as $n\rightarrow\infty$.
\end{remark}

To simplify this introduction, we have focused on basic $\Z_p$-towers.  Later, we will consider some other classes of towers and see that some form of Philosophy~\ref{philosophy} continues to hold.  Monodromy stable towers behave like basic towers, while in other examples $a_n^{(r)}$ still appears regular but does not behave exactly as in Conjecture~\ref{MC} (see Section~\ref{sec:conjectures} and Section~\ref{sec:beyondbasic}). 

\begin{remark}
Writing $\Gscr_{X_n}:=J_{X_n}[p^{\infty}]$ for the $p$-divisible group of the Jacobian of $X_n$,
there is a canonical decomposition of $p$-divisible groups 
\[
\Gscr_{X_n}=\Gscr_{X_n}^{\et}\times \Gscr_{X_n}^{\mult}\times \Gscr_{X_n}^{\ll}
\]
into  \'etale, multiplicative, and local-local components.  As $\Cl_{K_n}[p^{\infty}] = \Gscr_{X_n}(k) = \Gscr_{X_n}^{\et}(k)$, the results of Mazur--Wiles, Crew, 
and Gold--Kiselevsky can be understood as theorems about the structure of $\Gscr_{X_n}^{\et}$.
Indeed, generalizing \cite[Proposition 2]{mw83}, Crew \cite[\S3]{CrewIwasawa2} proves
that for $S$ nonempty
and $k$ algebraically closed the projective limit $\varprojlim_n \Hom_{k}(\Gscr_{X_n}^{\et},\Q_p/\Z_p)$
is free of finite rank over $\Lambda$, and deduces the structure of $\varprojlim_n \Cl_{K_n}[p^{\infty}]$
for finite $k$ from this result.
The analogue of this result for the multiplicative part is provided by
\cite{CaisIwasawa}, which treats arbitrary pro-$p$ extensions of function fields,
and allows $S$ to be empty.
The local-local components $\Gscr_{X_n}^{\ll}$ are far more mysterious, and incorporate information about the structure of $M_n^{V\nilp}$.  
\end{remark}

\begin{remark}
Continuing the notation of the previous remark, when $S$ is {\em empty}, equation \eqref{eq:nilformula} reads $\Ht(\Gscr_{X_n}^{\ll}) = 2p^n(g_0-h_0)$.  
As in the cases of the \'etale and multiplicative components, this is a numerical shadow of a much deeper fact: the ``limit'' Dieudonn\'e module
$\D_{\infty}^{\ll} \colonequals \varprojlim_n \D(\Gscr_{X_n}^{\ll})$ is free of rank $2(g_0-h_0)$ over $\Lambda_W:=W(k)[\![\Gamma]\!]$
(see \cite{CaisIwasawa}).  Using familiar arguments from Iwasawa theory,  this structural result gives
complete control over $\Gscr_{X_n}^{\ll}$ as $n$ grows.  In particular, for each \'{e}tale $\Z_p$-tower and positive integer $r$, there exist $b_r , c_r \in \Q$ such that $a^{(r)}_n = b_r p^n + c_r$ for $n \gg 0$. 

This is very different than the behavior for ramified $\Z_p$-towers.  When $S$ is \emph{non-empty}, the $\Lambda_W$-module $\D_{\infty}^{\ll}=\varprojlim_n \D(\Gscr_{X_n}^{\ll})$ is {\em never} finitely generated \cite{CaisIwasawa}.  
 One might hope to tame such wild behavior by suitably enlarging 
the Iwasawa algebra, and indeed the canonical Frobenius and Verschiebung morphisms give $\D_{\infty}^{\ll}$
the structure of a (left) module over the ``Iwasawa Dieudonn\'e''-ring $\Lambda_W[\![ F,V]\!]$.
But it follows from \eqref{naiveguess} that $\D_{\infty}^{\ll}$
is not finitely generated over $\Lambda_W[\![F,V]\!]$ either!  
Indeed, writing $M_n:=\D(\Gscr_n^{\ll})$ and $M_{\infty}:=\varprojlim M_n$, 
the canonical projections $M_\infty \rightarrow M_n$ are all surjective, 
so if $M_{\infty}$ were generated by $\delta$ generators over $\Lambda_W[\![F,V]\!]$,
then the same would be true of
$M_n/(F,V)M_n$ as a module over $k[\Gamma/\Gamma_n]$;
in particular, the $k$-dimension of $M_n/(F,V)M_n$ would be bounded above
by $\delta |\Gamma/\Gamma_n|=\delta p^n$.
However, 
we have a natural identification
\begin{equation*}
    M_n/(F,V)M_n = 
    \coker\left( V: H^0(X_n,\Omega^1_{X_n/k}  )^{V\nilp} \rightarrow  H^0(X_n,\Omega^1_{X_n/k}  )^{V\nilp}\right) 
\end{equation*}
and the dimension of this cokernel is none other than the $a$-number $a_n$ of $X_n$.
As $a_n$ is bounded below by $cp^{2n}$ with $c>0$ thanks to \eqref{naiveguess}, the putative upper bound of $\delta p^n$
is violated for $n \gg 0$.
\end{remark}

\begin{remark}
Iwasawa theory usually considers the $p$-part of the class group, not the $p$-torsion, while in this paper we mainly look at the $p$-torsion in the motivic class group $J_{X_n}$.  However, the usual Iwasawa-theoretic arguments give similar results about the $p$-torsion in class groups of number fields; see \cite{monsky83} for an example where this is spelled out (in a more general setting).  
\end{remark}

\subsection{Overview of the Paper} 

As previously discussed, the goal of this paper is to provide computational and theoretic evidence of Philosophy~\ref{philosophy}.  Section~\ref{sec:towers} reviews information about $\Z_p$-towers of curves, Artin-Schreier-Witt theory, and invariants of towers.  Section~\ref{sec:conjectures} formulates a more general version of the conjecture in the introduction for monodromy stable towers, 
which are one natural class of towers to consider.

Section~\ref{sec:basicanumber} and Section~\ref{sec:basicgeneral} are the computational heart of the paper, providing an extensive set of examples\footnote{As these computations take significant amounts of time, we include a large collection of examples as part of \cite{bcgit}.} which support the conjecture for basic towers.  Section~\ref{sec:basicanumber} focuses on the $a$-number, while Section~\ref{sec:basicgeneral} addresses higher powers of the Cartier operator.  Section~\ref{sec:beyondbasic} presents some examples that support our conjectures for monodromy stable towers which are not basic and that suggest Philosophy~\ref{philosophy} continues to hold for non-monodromy stable towers.

In Section~\ref{sec:computing}, we describe an algorithm which we have implemented in the {\sc Magma} computer algebra system \cite{magma} that lets us produce these examples.  
Computer algebra systems like {\sc Magma} have the ability to compute a matrix representing the Cartier operator on the space of regular differentials on any smooth projective curve over a finite field.  We work in the special setting that the tower is based over the projective line and is totally ramified over the point at infinity and unramified elsewhere.  Our algorithm is much faster as it takes advantage of the structure of a $\Z_p$-tower and incorporates as much theoretical information as possible.  In particular, when a $\Z_p$-tower is presented in a standard form we are able to use results of Madden~\cite{madden} to obtain a simple basis for the space of regular differentials on each curve in the tower which greatly accelerates the computations.  This efficiency is crucial, as the genus of the curves in a $\Z_p$-tower very quickly become too large for the generic methods provided by {\sc Magma} to handle.  Our 
algorithm is efficient enough that we are able to compute sufficiently many levels of $\Z_p$-towers with small $p$
to provide convincing evidence for our conjectures.

Section~\ref{sec:theoretical} is the theoretical heart of the paper, where we prove special cases of our conjectures when $p=2$.  We do so by proving a general result (valid in any characteristic) about the trace of differentials on an Artin-Schreier cover that are killed by the Cartier operator.  When $p=2$, this is enough to gain control over the $a$-number.  These ideas give only very limited information about higher powers of the Cartier operator, even in characteristic two (Section~\ref{ss:otherproof}).

\begin{remark}
Computations in this paper were done using {\sc Magma} 2.25-6 and 2.25-8 \cite{magma} running on several different personal computers\footnote{The largest examples were done on a 2020 iMac with 3.8 GHz 8-Core Intel Core i7 and 128 GB 2667 MHz DDR4 Ram.}
and a server at the University of Canterbury.  Thus running times for different examples are not directly comparable as they may have been run on different machines, although they are of a similar magnitude.  When directly comparing running times, the same computer was used.
\end{remark}

\begin{notation}
In the rest of the paper, we will often want to compare multiple $\Z_p$-towers simultaneously while also avoiding excessive subscripts.  To do so, we adopt the following notation.  
\begin{itemize}
    \item  For a tower of curves $\cT$, we let $\cT(n)$ denote the $n$th level of the tower.
    
    \item  For a curve $X$, we use the notation $g(X)$, $a(X)$, and $a^r(X)$ for the genus, $a$-number, and dimension of the kernel of the $r$th power of $V_X$ on the space of regular differentials.
    
    \item      We let $J_{X}$ denote the Jacobian of $X$. 
    
    \item  Given a tower $\cT$ and point $Q$ in the base curve, Notation~\ref{not:ramcond} introduces invariants $s_Q(\cT(n))$, $u_Q(\cT(n))$, and $d_Q(\cT(n))$ which reflect the ramification of $\cT(n)$ over $Q$.
\end{itemize}
Notation~\ref{not:alpha} and Notation~\ref{notation:mr} give constants $\alpha(r,p)$ and $m(r,p)$ appearing in our conjectures.  
\end{notation}

\subsection{Acknowledgments} We thank Joe Kramer-Miller and James Upton for helpful conversations about $\Z_p$-towers and Daniel Delbourgo for helpful conversations about Iwasawa theory.  We thank Paul Brouwers for help with the mathmagma server at the University of Canterbury, Lu\'{i}s Finotti for helpful conversations about computing with Witt vectors, and Maher Hasan for helpful advice about the practicalities of software engineering.  We thank the referee for helpful suggestions.  Booher was partially supported by the Marsden Fund Council administered by the Royal Society of New Zealand.  Cais is supported by NSF grant number DMS-1902005.

\section{Towers of Curves} \label{sec:towers}

Fix a perfect field $k$ of characteristic $p>0$.  By a {\em curve} over $k$, we mean a smooth, projective, geometrically connected, $k$-scheme of dimension one.  We refer to a branched cover $\pi : Y \to X$ simply as a {\em cover}.  We view the branch locus as a set of $\kbar$-points of $X$.   We say the cover is {\em Galois} (resp.~{\em has Galois group} $G$) if the corresponding extension of function fields is Galois (resp. has Galois group $G$).

\subsection{Artin-Schreier-Witt Theory and \texorpdfstring{$\Z_p$}{Zp}-towers} \label{ss:ASW}

\begin{defn}
A $\Z_p$-tower of curves $\cT$ is a sequence of curves over $k$
\[
 \cT: \ldots \to \cT(3) \to  \cT(2) \to \cT(1) \to \cT(0)
\]
such that $\cT(n)$ is a Galois (branched) cover of $\cT(0)$ with $\Gal(\cT(n)/\cT(0)) \simeq \Z_p /p^n \Z_p \simeq \Z/p^n\Z$ for $n \geq 1$.  We assume that there is a finite non-empty set $S$ of $\kbar$-points of $\cT(0)$
such that $\cT(n)\rightarrow \cT(0)$ \'etale outside of $S$ and totally ramified over every point of $S$, for all $n$.
We refer to $\cT(n)$ as the \emph{$n$-th level} (or \emph{$n$-th layer}) of the tower, and to $\cT(0)$ as the {\em base} of the tower. 
\end{defn}

As we define curves to be geometrically connected, our $\Z_p$-towers are automatically geometric towers in the sense that all $\cT(n)$ have constant field $k$.

We can equivalently describe a tower of curves as a $\Z_p$-tower of function fields $k(\cT(n))$.  All $\Z_p$-towers of curves (equivalently function fields) can be described by Artin-Schreier-Witt theory.  This goes back to \cite{witt}: an accessible recent reference is \cite[\S3]{kwarxiv}, which builds on the theory of Witt vectors which are briefly reviewed in \cite[\S2]{kwarxiv} and more extensively reviewed in \cite{rabinoff}.  We mainly need the following special cases, which describe $\Z_p$-extensions of $k(\!(t)\!)$ (which are local) and $\Z_p$-towers over the projective line.

Let $W(K)$ denote the Witt vectors of the characteristic $p$ field $K$ with Frobenius $F$, and let $\wp : W(K) \to W(K)$ be given by $\wp(y) := Fy -y$.  We write $[\cdot]:K\rightarrow  W(K)$ for the Teichm\"{u}ller map,
which is the unique multiplicative section to the canonical projection $W(K)\rightarrow K$ onto the first Witt component.
Let $v$ be the $p$-adic valuation on $W(k)$ normalized so $v(p)=1$.  

\begin{fact} \label{fact:localasw}
Let $k$ be a finite field of characteristic $p$, and fix an element $\alpha$ of $k$ such that $\tr_{k/\F_p}(\alpha) \neq 0$. 
All $\Z_p$-extensions of $K = k(\!(t)\!)$ may be obtained by adjoining a solution $y_1,y_2,\ldots$ of the equation
\begin{equation} \label{eq:localasw}
\wp((y_1,y_2,\ldots)) = F(y_1,y_2,\ldots) - (y_1,y_2,\ldots) = c [\alpha] + \sum_{\gcd(i,p) =1} c_i [t^{-i}]
\end{equation}
in $W(k(\!(t)\!) )$, where $c_i \in W(k)$ and $c_i \to 0$ as $i \to \infty$.  The unique $\Z/p^n\Z$-subextension $K_n$ arises from adjoining $y_1,y_2,\ldots,y_n$ to $K$, and depends only on the right side modulo $p^n$.

The conductor of $K_n$ over $K$ is $(t^{u_n})$, where 
\[
u_n = \begin{cases}
1 + \max \{ i p^{n-1-v(c_i)} : \gcd(i,p)=1, \, v(c_i) < n \} & \textrm{if there exists } i \textrm{ such that } v(c_i) <n\\
0 & \textrm{otherwise}.
\end{cases}
\]
\end{fact}

This is \cite[Example 2.4, Proposition 3.1, Proposition 3.3]{kw18}.  Note this is a local statement, while the next fact is a global statement.

\begin{fact} \label{fact:p1asw}
Let $k$ be a finite field of characteristic $p$, and fix an element $\alpha$ of $k$ such that $\tr_{k/\F_p}(\alpha) \neq 0$. 
All $\Z_p$-extensions of $K = k(x)$ ramified over a set $S \subset \PP^1_k(\kbar)$ may be obtained by adjoining a solution $y_1,y_2,\ldots$ of the equation
\[
\wp((y_1,y_2,\ldots)) = F(y_1,y_2,\ldots) - (y_1,y_2,\ldots) = c [\alpha] + \sum_{Q \in S} \sum_{\gcd(i,p) =1} c_{Q,i} [\pi_Q^{-i}],
\]
with $c \in W(k)$, $c_{Q,i} \in W(\kbar)$, and with $\pi_Q = x - Q$ if $Q \in \kbar$ and $\pi_Q = 1/x$ if $Q = \infty$, such that
\begin{enumerate}
    \item for $\sigma \in \Gal(\kbar/k)$ and $Q \in \PP^1_k(\kbar)$, we have $\sigma c_{Q,i} = c_{\sigma Q,i}$;
    \item  for every integer $n \geq 1$, there exists finitely many $c_{Q,i}$ with $v(c_{Q,i}) < n$.  
\end{enumerate}
The unique $\Z/p^n\Z$-subextension $K_n$ arises from adjoining $y_1,y_2,\ldots,y_n$ to $K$, and depends only on the right side modulo $p^n$.
The tower is geometric
if there exists a $c_{Q,i}$ with valuation $0$.  
\end{fact}

Again see \cite{kw18}, especially Proposition 4.9.   

\begin{remark}
The first level of these extensions (given by adjoining $y_1$, or equivalently working in the truncated Witt-vectors $W_1(K)$ and with the right side modulo $p$)  are Artin-Schreier extensions.  For example, \eqref{eq:localasw} becomes $$y_1^p - y_1 = c \alpha + \sum_{\substack { \gcd(i,p)=1 \\ v(c_i) =0 }} c_i t^{-i}.$$
Similarly, the unique $\Z/p^n\Z$-extension of $L=K(\{y_i\})$ can be described
using the truncated Witt vectors $W_n(K)$.  Recall that arithmetic with Witt vectors is not done component-wise, and is highly non-trivial.  In particular, while $[c x^i] = (c x^i,0,0,\ldots)$, the sum $[c_i x^i] + [c_j x^j]$ is not $(c_i x^i + c_j x^j,0,0,\ldots)$.
\end{remark}

\subsection{Ramification and Conductors in Towers}

\begin{notation} \label{not:ramcond}
Let $\cT$ be a $\Z_p$-tower of curves over $k$ and let $Q\in S$.  

\begin{enumerate}
    \item Let $d_{Q}(\cT(n))$ be the unique break in the lower ramification filtration of the cover $\cT(n) \to \cT(n-1)$ at the point above $Q$ (the \emph{ramification invariant} above $Q$).
    
\item  Let $s_Q(\cT(n))$ be the largest break in the upper ramification filtration for the cover $\cT(n) \to \cT(0)$ above $Q$.   

\item  When $k$ is finite\footnote{A finite residue field is necessary to define the conductor using class field theory.}, let $u_Q(\cT(n))$ be the exponent of the conductor for the extension of local fields coming from $\cT(n)_Q \to \cT(0)_Q$.  
\end{enumerate}
\end{notation}

Recall that the upper numbering is compatible with quotients, 
so we can give $\cT$ an upper ramification filtration making $s_Q(\cT(n))$ the $n$th (upper) break above $Q$.  The lower numbering is compatible with subgroups, and hence the largest break in the lower ramification filtration of $\cT(n) \to \cT(0)$ above $Q$ is $d_Q(\cT(n))$. 

\begin{lem} \label{lem:breaks}
Let $\cT$ be a $\Z_p$-tower of curves over $k$ and $Q\in S$. For each positive integer $n$:
\begin{enumerate}
\item \label{break1} $\displaystyle d_Q(\cT(n)) = p^{n-1} s_Q(\cT(n)) - \sum_{j=1}^{n-1} \varphi(p^j) s_Q(\cT(j))$;
\item \label{break3} $d_Q(\cT(n+1)) - d_Q(\cT(n)) = \left(s_Q(\cT(n+1)) - s_Q(\cT(n)) \right) p^n $; 
\item \label{break2} if $k$ is finite, $u_Q(\cT(n)) = s_Q(\cT(n))+1$.
\end{enumerate}
\end{lem}

\begin{proof}
This result is standard, although we do not know a good reference for this exact statement.
The relationship between the breaks in the upper and lower ramification filtrations in a $\Z/p^n\Z$-extension of local fields is spelled out in \cite[IV.3 Example]{serrelocal}.   There exist positive integers $i_0, i_1, \ldots, i_{n-1}$ such that the breaks in the upper numbering filtration are $i_0, i_0 + i_1, \ldots , i_0 + i_1 + \ldots+ i_n$ and the breaks in the lower numbering filtration are $i_0, i_0 + p i_1, \ldots, i_0 + p i_1 + \ldots p^{n-1} i_{n-1}$.  In particular, $s_Q(\cT(j)) = i_0 + i_1 + \ldots + i_{j-1}$ and $d_Q(\cT(n)) = i_0 + p i_1 + \ldots p^{n-1} i_{n-1}$, and (\ref{break1}) follows. Statement (\ref{break3}) is a formal consequence of the previous part.  When $k$ is finite, the relationship between the conductor and the upper ramification breaks in (\ref{break2}) follows from \cite[\S XV.2 Corollary 2 to Theorem 1]{serrelocal}.
\end{proof}

\begin{lemma} \label{lem:genuslower}
Let $\cT$ be a $\Z_p$-tower totally ramified above a finite set $S$ of $\kbar$-points of $\cT(0)$.  Then
\begin{align*}
2 g(\cT(n)) -2  &= p^n ( 2 g(\cT(0))-2) + \sum_{Q \in S} \sum_{i=1}^n \varphi(p^{n+1-i}) (d_Q(\cT(i)) +1 ) \\
&=  p^n ( 2 g(\cT(0))-2)) + \sum_{Q \in S} \sum_{i=1}^n \varphi(p^i) (s_Q(\cT(i))+1) .
\end{align*}
\end{lemma}

\begin{proof}
Apply the Riemann-Hurwitz formula. 
\end{proof}

\begin{remark} \label{remark:lowergrowth}
As remarked in the introduction, for $\Z_p$-towers in characteristic $p$ there is always ``a lot'' of ramification.  In particular, if $\cT$ is totally ramified above $Q$ then $s_Q(\cT(n)) \geq p s_Q(\cT(n-1))$.
Using Lemma~\ref{lem:breaks}(\ref{break1}) to convert to the lower ramification filtration, it follows that $d_Q(\cT(n)) \geq (p^2-p+1) d_Q(\cT(n-1))$.  Using Lemma~\ref{lem:genuslower}, for any ramified $\Z_p$-tower there is a constant $c>0$ such that $g(\cT(n)) \geq c p^{2n}$.
\end{remark}

\subsection{Types of Towers}

We next identify several nice kinds of $\Z_p$-towers which we will focus on.
\begin{definition} \label{defn:types}
Let $\cT$ be a $\Z_p$-tower of curves over $k$ with branch locus $S$.
\begin{enumerate}

    \item  We say $\cT$ is \emph{monodromy stable}, or has stable monodromy, if for every $Q \in S$ there exists $c_Q, d_Q \in \Q$ such that for $n \gg 0$ 
    \[
    s_Q(\cT(n)) = c_Q + d_Q p^{n-1}.
    \]
    
    \item  We say that $\cT$ has \emph{periodically stable monodromy}, or is periodically monodromy stable, if for every $Q \in S$ there exists an integer $m_Q$, a $d_Q \in \Q$, and a function $c_Q : \Z/m_Q \Z \to \Q$ such that for $n \gg 0$ 
    \[
        s_Q(\cT(n)) = c_Q(n) + d_Q p^{n-1}.
    \]
 \end{enumerate}
\end{definition}

As the Riemann-Hurwitz formula determines the genus of a cover in terms of the genus of the base curve and the ramification, the genus of a monodromy stable (resp. periodically monodromy stable) $\Z_p$-tower is of the form $a p^{2n} + b p^n + c$ for $n \gg 0$ (is of the form $a(n) p^{2n} + b(n) p^n + c(n)$ where $a,b,c$ are eventually periodic functions).  This behavior is referred to as being \emph{genus stable} (resp. \emph{periodically genus stable}).  For later use, we record:

\begin{lem} \label{lem:dsinstable}
If $\cT$ is monodromy stable and $Q\in S$ with $s_Q(\cT(n)) = c_Q + d_Q p^{n-1}$ for $n \gg 0$, then there exists $c'_Q \in \Q$ such that
\[
d_Q(\cT(n)) = d_Q \frac{p^{2n-1}}{p+1} + c'_Q \text{ for } n \gg 0.
\]
Furthermore, $g(\cT(n))$ is asymptotically $ \displaystyle \left( \sum_{Q \in S} d_Q \right) \frac{ p^{2n}}{2 (p+1)} $.
\end{lem}

\begin{proof}
For the first, use the definition of monodromy stability plus Lemma~\ref{lem:basicinvariants}(\ref{break3}).  Then the second statement follows using Lemma~\ref{lem:genuslower}.
\end{proof}

\begin{remark}
Monodromy stable towers are a very natural class of towers to consider as all $\Z_p$-towers of ``geometric origin'' are monodromy stable \cite{kmmonodromy}.
\end{remark}

Many of our computations will deal with a particularly simple class of $\Z_p$-towers over $\PP^1_k$ where $k$ is finite which we refer to as {\em basic} $\Z_p$-towers.
Fix a coordinate $x$ on the projective line $\PP^1_k$.

\begin{definition} \label{defn:basictower}
Let $d$ be a positive integer that is prime to $p$ and let $k$ be a finite field of characteristic $p$.  A \emph{basic $\Z_p$-tower $\cT$ with ramification invariant $d$} is the $\Z_p$-tower over $\PP^1_k$ given by the Artin-Schreier-Witt equation
\[
Fy - y = \sum_{\substack{i=1 \\ (i,p)=1}}^d [c_i x^i]
\]
with $c_i \in k$ and $c_d \neq 0$.  (It is convenient to then define $c_i =0$ when $p | i$.)
\end{definition}

These are also called \emph{unit root $\Z_p$-extensions} \cite[Example 4.10]{kw18}.  By Fact~\ref{fact:p1asw}, the function field of $\cT(n)$ is the $\Z/ p^n \Z$-extension of $k(x)$ given by adjoining $y_1,y_2,\ldots, y_n$ where $(y_1,y_2,\ldots,y_n) \in W_n(k(x))$ is a solution of the Witt vector equation
\begin{equation}
F(y_1,y_2,\ldots,y_n) - (y_1,y_2,\ldots,y_n) = \sum_{i=1}^d  (c_i x^i,0,\ldots, 0).
\end{equation}
In particular, $\cT(1)$ is the Artin-Schreier curve given by $\displaystyle y_1^p - y_1 = \sum_{i=1}^d c_i x^i$.

\begin{lemma} \label{lem:basicinvariants}
A basic $\Z_p$-tower $\cT$ with ramification invariant $d$ is totally ramified over $\infty$ and unramified elsewhere.  Recalling Notation~\ref{not:ramcond}, we have that 
\[
u_\infty(\cT(n)) -1 = s_\infty(\cT(n))  = d p^{n-1} \quad \text{and that} \quad d_\infty(\cT(n)) = d  \cdot \frac{p^{2n-1} +1}{p+1}.
\]
In particular, basic $\Z_p$-towers are monodromy stable (recall Definition~\ref{defn:types}) and the genus satisfies
\[
2 g(\cT(n)) -2 = \frac{d}{p+1} p^{2n} - p^n  - \frac{p+1+d}{p+1}.
\]
\end{lemma}

\begin{proof}
We see that $u_\infty(\cT(n)) = 1 + d p^{n-1}$ using Fact~\ref{fact:localasw}.  By Lemma~\ref{lem:basicinvariants}, we obtain the formulas for $s_\infty(\cT(n))$, $d_\infty(\cT(n))$, and $g(\cT(n))$.  See also \cite[Example 4.10]{kw18}
\end{proof}

\begin{remark}\label{basicnotation}
When working with a basic $\Z_p$-tower $\cT$, each layer $\cT(n)\rightarrow \cT(0)=\PP^1_k$ 
is totally ramified over the point at infinity, and unramified elsewhere.
We will therefore often write $u(\cT(n))$ instead of $u_\infty(\cT(n))$ (and similarly for $s(\cT(n))$ and $d(\cT(n))$).  
\end{remark}

Another nice example of a monodromy stable tower is the Igusa tower.  

\begin{example} \label{ex:igusaproperties}
We work over $k=\overline{\F}_p$, and let $\Ig(n)$ denote the curve representing the moduli problem of elliptic curves over $k$ with an Igusa level structure of level $p^n$ and a $\moduli$-level structure for a suitable auxiliary moduli problem $\moduli$ (see \cite[Chapter 12]{km85}). 

For example, when $p \neq 5$ we could choose to use $\moduli = \Gamma_1(5)$, which satisfies the hypotheses of \cite[Theorem 12.9.1]{km85}, as the auxiliary moduli problem.  (It is standard to compute that the moduli problem $\Gamma_1(5)$ has degree $24$ and has $4$ cusps.)   Note that $\Ig(n)$ is a smooth proper curve over $k$, and it is connected (it suffices to check this for $X_1(5)$ over $\C$).  Then $\Ig(1)$ is a $(\Z/p\Z)^\times / \{\pm 1\}$-cover of $X_1(5)_k \simeq \PP^1_k$ totally ramified over the supersingular points and \cite[Corollary 12.9.4]{km85} gives a $\Z_p$-tower
\[
\Ig : \cdots \to \Ig(3) \to \Ig(2) \to  \Ig(1) ,
\]
totally ramified over the $p-1$ points $S$ of $\Ig(1)$ which lie over the supersingular points of $X_1(5)_k$, and unramified elsewhere.  
We know $d_Q(\Ig(n)) = p^{2(n-1)}-1$ for each $Q \in S$ by \cite[Lemma 12.9.3]{km85}, which implies $g(\Ig(n)) = p^{2n-1} (p-1)/2 - 2 p^{n-1}(p-1) + 1$ as in \cite[Corollary 12.9.4]{km85}.  
\end{example}

\section{Conjectures for Monodromy Stable Towers} \label{sec:conjectures}

For a $\Z_p$-tower of curves $\cT$ over a perfect field of characteristic $p$, Philosophy~\ref{philosophy}
predicts that the invariants of $J_{\cT(n)}[p]$ should be ``regular'' for $n \gg 0$.  This regularity should furthermore depend only the local information given by the ramification filtration at each ramified point.

\begin{philosophy} \label{philosophy:sum}
For a $\Z_p$-tower of curves $\cT$ over a perfect field of characteristic $p$ ramified over $S$, invariants of $J_{\cT(n)}[p]$ should be a sum of ``local contributions'' depending only on the ramification of $\cT$ at each branch point $Q\in S$.
\end{philosophy}

\begin{remark}
 The genus and $p$-rank of a monodromy stable tower illustrate this philosophy as they include a contribution from each point of ramification.  See for example the asymptotic for the genus in Lemma~\ref{lem:dsinstable} and equation \eqref{eq:DSformula}  
\end{remark}

In this section, we will formulate precise conjectures for monodromy stable $\Z_p$-towers that exemplify these philosophies.  We make these conjectures only for $\dk{r}{\cT(n)}$ with $r \geq 1$, which are a partial list of invariants for $J_{\cT(n)}[p]$.  We restrict ourselves in this manner as:
\begin{itemize}
    \item the ramification is simple in monodromy stable towers, so it is much easier to see how $\dk{r}{\cT(n)}$ is ``regular'' for $n \gg 0$;
    
    \item  it is feasible to compute with them: $\dk{r}{\cT(n)}$ can be computed using the action of the Cartier operator on the space of regular differentials, and in monodromy stable towers the dimension of this vector space (the genus) is ``only'' asymptotic to $c p^{2n}$ with $c>0$.   Other $\Z_p$-towers with more complicated ramification will usually have even faster genus growth. 
\end{itemize}

We begin by considering the asymptotic growth of $\dk{r}{\cT(n)}$ in monodromy stable $\Z_p$-towers. 

\begin{notation} \label{not:alpha}
For a prime $p$ and positive integers $d$ and $r$, define 
\begin{equation}
    \alpha(r,p) \colonequals \frac{r (p-1) }{2(p+1)((p-1)r + (p+1)) }.
\end{equation}
We will also use the shorthand $\alpha(p) \colonequals \alpha(1,p)$.
\end{notation}

\begin{conjecture} \label{conj:stableasymptotic}
Let $\cT$ be a monodromy stable $\Z_p$-tower totally ramified over $S$ and unramified elsewhere.  For $Q \in S$, let $c_Q, d_Q \in \Q$ with $s_Q(\cT(n)) = d_Q p^{n-1} + c_Q$ for $n \gg 0$ and set $\displaystyle D \colonequals \sum_{Q \in S} d_Q$.  Then $\dk{r}{\cT(n)}$ is asymptotically $\alpha(r,p) D p^{2n}$ for large $n$; in other words
\[
\lim_{n \to \infty} \frac{\dk{r}{\cT(n)}}{\alpha(r,p) D p^{2n}} =1.
\]
\end{conjecture}

\begin{cor} \label{cor:ratio}
Conjecture~\ref{conj:stableasymptotic} implies that for a totally ramified monodromy stable $\Z_p$-tower $\cT$
\begin{equation}
\lim_{n \to \infty} \frac{\dk{r}{\cT(n)}}{g(\cT(n))} = \frac{r (p-1)}{(p-1)r + (p+1)} = \frac{1}{1 + \frac{p+1}{(p-1) r}} .
\end{equation}
\end{cor}

\begin{proof}
For $Q \in S$, as before let $s_Q(\cT(n)) = d_Q p^{n-1} + c_Q$ for $n \gg 0$ with $c_Q, d_Q \in \Q$ and set $\displaystyle D \colonequals \sum_{Q \in S} d_Q$.
From Lemma~\ref{lem:dsinstable}, we know that $g(\cT(n)$ is asymptotic to $D/ (2(p+1)) p^{2n}$.  Then compare with the asymptotic for $\dk{r}{\cT(n)}$ from Conjecture~\ref{conj:stableasymptotic}.  
\end{proof}

For example, in monodromy stable towers we predict that 
\begin{equation} \label{eq:asymptoticanumber}
    \displaystyle \lim_{n \to \infty} \frac{a(\cT(n))}{g(\cT(n)} = \frac{p-1}{2p}.
\end{equation}

\begin{remark}
The limit in Corollary~\ref{cor:ratio} approaches $1$ as $r$ becomes large.  Thus Conjecture~\ref{conj:stableasymptotic} predicts that the Cartier operator is essentially nilpotent on $H^0(\Omega^1_{\cT(n)})$.  This is as expected: the $k[V_{\cT(n)}]$-module
$H^0(\Omega^1_{\cT(n)})$ decomposes as a direct sum of its $V_{\cT(n)}$-nilpotent and $V_{\cT(n)}$-bijective submodules as in \eqref{eq:fitting},
and the Deuring-Shafarevich formula \eqref{eq:DSformula} shows that the $k$-dimension of the $V_{\cT(n)}$-bijective component
 is bounded by a constant times $p^n$, whereas the genus (and hence the $k$-dimension of the $V_{\cT(n)}$-nilpotent component) is on the order of $p^{2n}$.  In other words, the Cartier operator acts nilpotently on essentially all of $H^0(\Omega^1_{\cT(n)})$ as $n\rightarrow \infty$.  
\end{remark}

We also formulate more precise conjectures about the exact values of $a(\cT(n))$ and $\dk{r}{\cT(n)}$ in monodromy stable towers.  We begin with the $a$-number, whose behavior seems simplest.

\begin{conjecture} \label{conj:stableanumber}
For every monodromy stable $\Z_p$-tower of curves $\cT$ over a perfect field of characteristic $p$, 
there exist $a,b,c \in \Q$ such that
\[
a(\cT(n)) = \dk{1}{\cT(n)} = a p^{2n} + b p^n + c \text{ for } n \gg 0.
\]
\end{conjecture}

Note that Conjecture~\ref{conj:stableasymptotic} predicts the value of $a$ in Conjecture~\ref{conj:stableanumber}.

\begin{conjecture} \label{conj:stablespecific}
Fix $r \geq 1$.  For every monodromy stable $\Z_p$-tower of curves $\cT$ over a perfect field of characteristic $p$, 
there exists a positive integer $m$ and functions $a,b,c,\lambda : \Z/m\Z \to \Q$ such that
\[
\dk{r}{\cT(n)} = a(n) p^{2n} + b(n) p^n + c(n) + \lambda(n) n \text{ for } n \gg 0.
\]
\end{conjecture}

Again, Conjecture~\ref{conj:stableasymptotic} predicts that the function $a(n)$ in Conjecture~\ref{conj:stablespecific} is a constant function taking on a specific value.   
Writing $s_Q(\cT(n)) = d_Q p^{n-1} + c_Q$ for $Q\in S$ and $n \gg 0$ with $c_Q, d_Q \in \Q$, it predicts that 
$$a(n) = \alpha(r,p) \left( \sum_{Q \in S} d_Q \right). $$ 

In Sections~\ref{sec:basicanumber}-\ref{sec:beyondbasic}, we provide evidence for these conjectures. 
We mainly focus on basic $\Z_p$-towers as they are easiest to compute with; note that Conjecture~\ref{MC}, which addressed basic towers, is compatible with these more general conjectures.  
We then give some additional examples of other monodromy stable towers as well as a few examples featuring non-monodromy stable towers that support Philosophy~\ref{philosophy} while exhibiting more complicated behavior.

\begin{remark}
Towers with periodic, non-stable monodromy do not seem to satisfy Conjecture~\ref{conj:stableanumber}.  There does appear to be similar formula for the $a$-number, but unsurprisingly the constants depend on the parity of $n$.   However, limited investigations suggest that towers with periodic monodromy may satisfy Conjecture~\ref{conj:stablespecific} as well; see Section~\ref{ss:periodic}.
\end{remark}

\begin{remark}  We are not completely confident that monodromy stable $\Z_p$-towers are the correct class of $\Z_p$-towers to consider.  After this paper first appeared as a preprint, Joe Kramer-Miller and James Upton suggested that these conjectures might only hold for overconvergent $\Z_p$-towers.  Basic tower are both monodromy stable and overconvergent, so since most of our evidence comes from computing with basic towers it is difficult to investigate the difference. 
\end{remark}

\section{\texorpdfstring{$a$}{a}-numbers for Basic Towers} \label{sec:basicanumber}
We first focus on the $a$-number of curves in basic $\Z_p$-towers $\cT$ (Definition~\ref{defn:basictower}) with ramification invariant $d$.    
By Lemma~\ref{lem:basicinvariants} (and noting Remark \ref{basicnotation}), we have $s(\cT(n)) = d p^{n-1}$.  Unwinding Notation~\ref{not:alpha}, we see that
\[
\alpha(p) = \alpha(1,p) = \frac{(p-1)}{4 (p+1) p}.
\]
In this case Conjecture~\ref{conj:stableasymptotic} predicts that 
\begin{equation}
\lim_{n \to \infty}  \frac{a(\cT(n))}{\alpha(p) d p^{2n}}  = 1.
\end{equation}
We now present a refinement of Conjecture~\ref{conj:stableanumber} and the $r=1$ case of Conjecture~\ref{MC}.

\begin{conjecture} \label{conj:basicanumber}
For every basic $\Z_p$-tower $\cT$ with ramification invariant $d$, there exists a positive integer $N_d$ (depending only on $d$ and $p$) and $c \in \bQ$ (depending on $\cT$) such that
\[
a(\cT(n)) = \dk{1}{\cT(n)} = \alpha(p) d p^{2n} + c \text{ for } n \geq N_d.
\]
\end{conjecture}

Note that $ \alpha(p) d p^{2n} $ need not be an integer, but it is straightforward to check that $\alpha(p) d (p^{2n}-p^2)$ is always an integer when $p > 2$. 
Thus for convenience we define
\begin{equation} 
    \delta_d (\cT(n)) \colonequals a(\cT(n)) - \alpha(p) d (p^{2n}-p^2).\label{eq:deltadef}
\end{equation}
Conjecture~\ref{conj:basicanumber} for a basic tower $\cT$ with ramification invariant $d$ is equivalent to $\delta_d(\cT(n))$ being constant for sufficiently large $n$.

\subsection{Examples in Characteristic Three} We begin by focusing on $\Z_3$-towers in characteristic $3$, which we analyzed using the methods of Section~\ref{sec:computing}.  

\begin{example} \label{example:p3d7}
Let $p=3$ and $d=7$.  Consider the basic towers 
\[
\cT_1 : Fy - y = [x^7], \quad  \cT_2 : Fy -y = [x^{7}] -[x^{5}] -[x^{2}],  \quad \cT_3 : Fy-y = [x^{7}] - [x^{5}].
\]
These towers have ramification invariant $7$, and the corresponding levels of each tower have the same genus.  Table~\ref{table:p3d7} shows they do not have identical $a$-numbers, although the $a$-numbers are highly constrained.  

\begin{table}[ht]
\centering
\begin{tabular}{c|c c c c c c }
Level: & 1 & 2 & 3 & 4 & 5 \\ 
\toprule
$g(\cT_1(n))$ & 6 & 66 & 624 & 5700 & 51546 \\ 
$a(\cT_1(n))$ & 4 & 25 & 214 & 1915 & 17224 \\ 
$a(\cT_2(n))$ & 3 & 24 & 213 & 1914 & 17223 \\ 
$a(\cT_3(n))$ & 3 & 24 & 213 & 1914 & 17223 \\ 
$\delta_7(\cT_1(n))$ & 4 & 4 & 4 & 4 & 4 \\ 
$\delta_7(\cT_2(n))$ & 3 & 3 & 3 & 3 & 3 \\ 
$\delta_7(\cT_3(n))$ & 3 & 3 & 3 & 3 & 3 \\ 
\bottomrule
\end{tabular}
\caption{Basic towers with $p=3$ and $d=7$, five levels}
\label{table:p3d7}
\end{table}

In particular, letting $\cT$ be any of these three towers, we observe that for $1 \leq n \leq 5$,
\begin{equation} \label{eq:p3d7}
a(\cT(n)) = 7 \alpha(3) (3^{2n}- 9) + a(\cT(1)) = \frac{7}{24} (3^{2n}- 9)  + a(\cT(1)).
\end{equation}
This holds for all of levels of all basic towers in characteristic $3$ with ramification invariant $7$ that we have computed.  (Including the previous examples, this is $4$ towers for which 5 levels were analyzed and $16$ for which $4$ levels where analyzed.) 
Note that by \cite[Theorem 6.26]{bc20}, $3 \leq a(\cT(1)) \leq 4$ for any $\Z_3$-tower with ramification invariant $7$.
\end{example}

\begin{example} \label{example:p3d5}
Let $p=3$ and $d=5$.  Consider the basic towers
\[
\cT_1 : Fy - y = [x^{5}] - [x^{2}] \quad \text{and} \quad \cT_2 : Fy - y = [x^{5}] - [x^{4}]  - [x] .
\]
Table~\ref{table:p3d5} shows that unlike for towers with ramification invariant $7$, the $a$-number of the first level does not determine the $a$-number of higher levels for $\cT_1$ and $\cT_2$.  

\begin{table}[ht]
\centering
\begin{tabular}{c|c c c c c c }
Level: & 1 & 2 & 3 & 4 & 5 \\ 
\toprule
$g(\cT_1(n))$ & 4 & 46 & 442 & 4060 & 36784 \\ 
$a(\cT_1(n))$ & 2 & 19 & 154 & 1369 & 12304 \\ 
$a(\cT_2(n))$ & 2 & 18 & 153 & 1368 & 12303 \\ 
$\delta_5(\cT_1(n))$ & 2 & 4 & 4 & 4 & 4 \\ 
$\delta_5(\cT_2(n))$ & 2 & 3 & 3 & 3 & 3 \\ 
\bottomrule
\end{tabular}
\caption{Basic towers with $p=3$ and $d=5$, five levels}
\label{table:p3d5}
\end{table}
For $n \geq 2$, it appears that
\[
a(\cT_1(n)) = \frac{5}{24} (3^{2n} - 9) + 4 \quad \text{and} \quad a(\cT_2(n)) =  \frac{5}{24} (3^{2n} - 9)  + 3.
\]
These formulae are not valid for $n=1$.  In particular, this illustrates that the restriction that $n$ is sufficiently large in Conjecture~\ref{conj:basicanumber} is necessary.  Based on our computations with $13$ towers (some with only four levels computed), it appears we may take $N_5 = 2$.  Also, note that by \cite[Theorem 6.26]{bc20}, $a(\cT(1))=2$ for any basic $\Z_3$-tower $\cT$ with ramification invariant $5$.
\end{example}

\begin{example} \label{example:p3d23}
Table~\ref{table:p3d23} shows the $a$-numbers of five selected basic $\Z_3$-towers with ramification invariant $23$.  The tower $\cT_1$ is $Fy - y = [x^{23}]$, while the other four towers are more complicated.\footnote{Computing $a(\cT_1(5))$ took around $5$ hours because of the tower's simple description, but it took over a month to compute $a(\cT_2(5))$.  This is why we have declined to compute the $a$-numbers of the $5$th levels for the remaining towers.}  For example, $\cT_2$ is 
\[
Fy - y = [x^{23}] +[x^{20}] +[x^{17}] +[x^{16}] +[x^{14}] -[x^{13}] -[x^{10}] -[x^{8}] -[x^{7}] -[x^{5}] +[x^{2}] +[x].
\]

\begin{table}[ht]
\centering
\begin{tabular}{c|c c c c c c }
n= & 1 & 2 & 3 & 4 & 5 \\ 
\toprule
$g(\cT_1(n))$ & 22 & 226 & 2080 & 18820 & 169642 \\ 
$a(\cT_1(n))$ & 12 & 83 & 706 & 6295 & 56596 \\ 
$a(\cT_2(n))$ & 10 & 80 & 702 & 6291 & 56592\\ 
$a(\cT_3(n))$ & 11 & 81 & 702 & 6291 \\ 
$a(\cT_4(n))$ & 12 & 81 & 702 & 6291 \\ 
$a(\cT_5(n))$ & 11 & 80 & 703 & 6292 \\ 
\\
$\delta_{23}(\cT_1(n))$ & 12 & 14 & 16 & 16 & 16 \\ 
$\delta_{23}(\cT_2(n))$ & 10 & 11 & 12 & 12 & 12\\ 
$\delta_{23}(\cT_3(n))$ & 11 & 12 & 12 & 12 \\ 
$\delta_{23}(\cT_4(n))$ & 12 & 12 & 12 & 12 \\ 
$\delta_{23}(\cT_5(n))$ & 11 & 11 & 13 & 13 \\ 
\bottomrule
\end{tabular}
\caption{Basic Towers with $p=3$ and $d=23$}
\label{table:p3d23}
\end{table}



We see the same basic phenomena as in Examples~\ref{example:p3d7} and \ref{example:p3d5}, although the stabilization is now more complicated.  It appears $\delta_{23}(\cT(n))$ may not stabilize until the third level, there are multiple choices for the $a$-number of level one, and $\delta_{23}(\cT(n))$ may jump multiple times.  Still, all of our examples are consistent with Conjecture~\ref{conj:basicanumber} holding with $N_{23}= 3$.
\end{example}

\begin{remark}
For basic $\Z_3$-towers, computing the $a$-number of the $5$th level is pushing the limit of what is feasible to compute as  illustrated by Example~\ref{example:p3d23}.  As the genus is growing exponentially with $n$, computing with the sixth level would require more time and memory than is reasonable.\footnote{This is not just a problem of limited resources.  {\sc Magma} imposes a limit on the number of monomials allowed in a multivariable polynomial expression.  Our program would run into this limit while attempting to construct an explicit representation of the sixth level as an Artin-Schreier extension of the fifth.}
\end{remark}

\subsection{Evidence in Characteristic Three} \label{ss:evidencep3}
In total, we have computed the $a$-number for the first four or five levels of at least $243$ basic $\Z_3$-towers.\footnote{As these computations are time intensive, we have made the results 
publicly available \cite[data\_storage]{bcgit}.}  The largest ramification invariant $d$ with which we have computed is $d=49$, and most of the computations of the fifth level of $\Z_3$-towers have taken place either for the tower $F y = y = [x^d]$ or with $d$ relatively small.  The computations take increasing amounts of time for larger $d$ as the genus of the $n$th level depends linearly on $d$ and the running time is polynomial in the genus.  For larger $d$, we analyzed five levels for the tower $Fy - y = [x^d]$ for $d$ up to $49$; 
as discussed in Remark~\ref{rmk:xdfast} this tower is quicker to compute with.\footnote{Despite being ``quicker'', computing  $a(\cT(5))$ for the tower $\cT:Fy- y = [x^{49}]$ took around $40$ hours.}

We also computed the first three levels of $510$ towers with ramification invariant up to thirty,\footnote{The results of these computations are stored in \cite[data\_storage\_small]{bcgit}.} carefully chosen so as to have diversity of $a$-numbers for the first level.  For each $d$, we searched through a large number of polynomials $f \in \F_3[x]$ of degree $d$ and computed the $a$-number of the Artin-Schreier curve 
\[
C_f : y^3 -y = f(x).
\]
For each value $\alpha$ of the $a$-number appearing frequently, we picked $10$ polynomials $f = \sum_{i=0}^d c_i x^i$ (with $c_d \neq 0$ and $c_i=0$ when $p \mid i$) such that $a(C_f) = \alpha$ and computed the $a$-numbers for the first three levels of the Artin-Schreier-Witt tower
\[
\cT_f: Fy - y = \sum_{i=0}^d [c_i x^i]
\]
whose first level is $C_f$. 

\begin{defn}
An integer $n>1$ is a {\em discrepancy} of a basic tower $\cT$ with ramification invariant $d$
if $\delta_d(\cT(n)) \neq \delta_d(\cT(n-1))$, where $\delta_d$ is as in \eqref{eq:deltadef}. 
\end{defn}

Conjecture~\ref{conj:basicanumber} is equivalent to the assertion that for each $d$, the largest discrepancy for a basic tower with ramification invariant $d$ is bounded independently of the tower.  If the conjecture holds, for $n$ sufficiently large $\delta_d(\cT(n)) $ would be the constant term $c$.

\begin{table}
\centering
\begin{tabular}{c|c c c c c c c c c c c c c c c c}
\toprule
$d$ & 2 & 4 & 5 & 7 & 8 & 10  & 11 & 13 & 14 & 16 & 17 & 19 \\ 
Discrepancies: & $\emptyset$ & $\emptyset$ & \{2\} & $\emptyset$  & $\emptyset$  & \{2\} & \{3\} & $\emptyset$ & \{2\} & \{2\} & \{3\} & \{2\}  \\
Towers:  & 4 & 25 & 13 & 40 & 25 & 25 & 25 & 36 & 25 & 36 & 36 &  36\\
\hline
$d$ & 20 & 22 & 23 & 25 & 26 & 28 & 29 & 31 & 32 & 34 & 35 & 37 \\
Discrepancies: & \{2\} & \{3\} & \{2,3\} &  \{2\} &  \{2\}  &  \{2,3\} &  \{2,4\} &  \{2\} &  \{2\} &  \{2,3\} &  \{2,4\} &  \{2\}           \\
Towers: & 36 & 47 & 37 & 47 & 46 & 48 & 47 & 10 & 9 & 9 & 10  & 9 \\
\hline
$d$ & 38 & 40 & 41 & 43 & 44 & 46 & 47 & 49 \\
Discrepancies: &  \{2,3\} &  \{2,3\} &  \{2,4\} &  \{2\}   &  \{2,3\} &  \{2,3\} &  \{2,4\} &  \{2,3\}  \\
Towers:  & 10 & 10 & 10 & 9 & 9 & 9 & 9 & 9  \\
\bottomrule
\end{tabular}
\caption{Observed Discrepancies for Basic Towers with Ramification Invariant $d<50$}
\label{table:discrepancies3}
\end{table}

Table~\ref{table:discrepancies3} shows the discrepancies for all of the towers we have collected data on with $d<50$ as well as the number of towers we analyzed for each $d$.  
(For small values of $d$, it is essential to work over extensions of $\FF_3$ as there are not that many basic towers defined over $\F_3$.)
This table supports Conjecture~\ref{conj:basicanumber} as it suggests that the discrepancies for towers with a given ramification invariant are bounded; 
the first time $n=2$ is a discrepancy is for $d=5$, the first time $n=3$ is a discrepancy is for $d=11$, and the first time $n=4$ is a discrepancy is for $d=29$.  In particular, we expect that for each basic $\Z_3$-towers with ramification invariant $d$ there exists $c \in \Z$ such that
\[
a(\cT(n)) = \alpha(3) d (3^{2n} -9) + c \quad \text{for} \quad n \gg 0,
\]
with the threshold for ``$n \gg 0$'' growing slowly with $d$.



\begin{remark}
    \begin{enumerate}
        \item As described above we have looked at fewer examples with $30 < d < 50$, so are less confident that we have identified all of the discrepancies possible for basic towers with that ramification invariant.  
        
        \item  In all of the examples we have looked at, $| \delta_d(\cT(n)) - \delta_d(\cT(n+1))| \leq 4 $.
        
        \item  As we have very few examples of computations with five levels and large $d$, and no computations in level six, it is difficult to be confident that the $a$-numbers for towers that have a discrepancy at level $n=4$
        actually stabilize.  For example, while the data in Table~\ref{table:p3d35} suggests that $\delta_{35}(\cT'(n))$ might stabilize for $n\geq 4$, we have no direct evidence that $\delta_{35}(\cT'(n))=18$ for $n \geq 4$.  
        However, we do see that for small $d$ (where the computations are fastest), the discrepancies (when there are any)
        are all very small, and only gradually increase as $d$ increases, which we find to be convincing evidence that all basic towers satisfy Conjecture~\ref{conj:basicanumber} for sufficiently large $N_d$.
            \end{enumerate}
\end{remark}

\begin{table}[ht]
\centering
\begin{tabular}{c|c c c c c c }
Level: & 1 & 2 & 3 & 4 & 5 \\ 
\toprule
$g(\cT(n))$ & 34 & 346 & 3172 & 28660 & 258214 \\ 
$a(\cT(n))$ & 20 & 127 & 1072 & 9579 & 86124 \\ 
$a(\cT'(n))$ & 17 & 122 & 1067 & 9573 \\ 
$\delta_{35}(\cT(n))$ & 20 & 22 & 22 & 24 & 24 \\ 
$\delta_{35}(\cT'(n))$ & 17 & 17 & 17 & 18 \\ 
\bottomrule
\end{tabular}
\caption{$\cT: Fy-y=[x^{35}]$, $\cT'$ also has ramification invariant $35$, $p=3$} \label{table:p3d35}
\end{table}

\subsection{Characteristic Two}  We now briefly discuss the $a$-numbers of $\Z_2$-towers in characteristic two.  Note that $\alpha(2) d = d/24$.  Table~\ref{table:p2d7} gives a representative example; it shows the $a$-numbers for \emph{any} basic $\Z_2$-tower with ramification invariant $7$.  
\begin{table}[ht]
\centering
\begin{tabular}{c|c c c c c c c c }
Level: & 1 & 2 & 3 & 4 & 5 & 6 & 7 \\ 
\toprule
$g(\cT(n))$ & 3 & 16 & 70 & 290 & 1178 & 4746 & 19050 \\ 
$a(\cT(n))$ & 2 & 5 & 19 & 75 & 299 & 1195 & 4779 \\ 
$a(\cT(n)) - 7 (2^{2n}-4)/24 + 1/2$  & 5/2 & 2 & 2 & 2 & 2 & 2 & 2 \\ 
\bottomrule
\end{tabular}
\caption{$\cT$ is any basic $\Z_2$-tower with ramification invariant $7$, seven levels}
\label{table:p2d7}
\end{table}

This is compatible with Conjecture~\ref{conj:basicanumber}.  In fact, for every positive odd integer $d$, the $a$-numbers of all basic towers with ramification invariant $d$ appear to be the same, and to support Conjecture~\ref{conj:basicanumber}.  
We are able to {\em prove} this: Corollary~\ref{cor:anumberbasic} will show that for any odd $d$ and all $n>1$
$$a(\cT(n)) = \frac{d}{24}( 2^{2n} - 4) + a(\cT(1)) - \frac{1}{2} = \frac{d}{6} (2^{2(n-1)}-1) + a(\cT(1)) - \frac{1}{2}.$$

\subsection{Other Characteristics}  Basic $\Z_p$-towers for $p>3$ are more difficult to compute with as the curves involved are of even higher genus.  (Recall the genus of the $n$-th level of a $\Z_p$-tower with ramification invariant $d$ is on the order of $d p^{2n}$ by Lemma~\ref{lem:basicinvariants}.)  We have only done substantial computations with a few simple towers in characteristic $5$. 

\begin{table}[ht]
\centering
\begin{tabular}{c|c c c c  || c | c  c c c }
Level: & 1 & 2 & 3 & 4 & Level: & 1 & 2 & 3 & 4\\ 
\toprule
$a(\cT_3(n))$ & 4 & 64 & 1564 & 39064 & $a(\cT_8(n))$ & 10 & 170 & 4170 & 104170 \\ 
$\delta_{3}(\cT_3(n))$ & 4 & 4 & 4 & 4 & $\delta_{8}(\cT_8(n))$ & 10 & 10 & 10 & 10 \\ 
$a(\cT_4(n))$ & 4 & 84 & 2084 & 52084 & $a(\cT_9(n))$ & 10 & 192 & 4692 & 117192 \\ 
$\delta_{4}(\cT_4(n))$ & 4 & 4 & 4 & 4 & $\delta_{9}(\cT_9(n))$ & 10 & 12 & 12 & 12 \\ 
$a(\cT_6(n))$ & 10 & 130 & 3130 & 78130 & $a(\cT_{11}(n))$ & 14 & 234 & 5734 & 143234 \\ 
$\delta_{6}(\cT_6(n))$ & 10 & 10 & 10 & 10 & $\delta_{11}(\cT_{11}(n))$ & 14 & 14 & 14 & 14 \\ 
$a(\cT_7(n))$ & 8 & 148 & 3650 & 91150 & $a(\cT_{12}(n))$ & 16 & 256 & 6256 & 156256 \\ 
$\delta_{7}(\cT_7(n))$ & 8 & 8 & 10 & 10 & $\delta_{12}(\cT_{12}(n))$ & 16 & 16 & 16 & 16 \\ 
\bottomrule
\end{tabular}
\caption{$\cT_d : Fy-y=[x^d]$ with $3 \leq d \leq 12$, four levels, $p=5$}
\label{table:p5}
\end{table}


\begin{example} \label{ex:p5}
 Table~\ref{table:p5} shows the $a$-numbers of the first four levels of the $\Z_5$-towers $\cT_d : Fy - y = [x^d]$ for small $d$.  All of these towers support Conjecture~\ref{conj:basicanumber} as $\delta_d(\cT_d(n))$ appears to be eventually constant.  To give context, $g(\cT_{12}(4)) = 390312$ and computing the $a$-number of $\cT_{12}(4)$ using the methods of Section~\ref{sec:computing} took around 253 hours, while $g(\cT_3(4)) = 97344$ and computing the $a$-number of $\cT_3(4)$ ``only'' took eight and a half hours.
 \end{example}
 
\begin{example}
As $\Z_p$-towers with $p>5$ are much slower to compute with, we have only been able to compute with the first two levels.  This is not enough to address Conjecture~\ref{conj:basicanumber}, but is enough to provide evidence for the leading term by computing 
\begin{equation} \label{eq:towererror}
\left| \frac{a(\cT(2))} { \alpha(p) d p^{4} } -1 \right|.
\end{equation}
We expect it to be close to zero.
\begin{itemize}
\item  When $p=7$, we computed the $a$-number for the second level of slightly over a thousand $\Z_7$-towers;  this quantity was less than $.015$ for all of them.

\item When $p=11$, we computed the $a$-number for the second level of around 650 $\Z_{11}$-towers; this quantity was less than $.0053$ for all of them.

\item  When $p=13$, we computed the $a$-number for the second level of eleven $\Z_{13}$-towers; this quantity was less than $.0051$ for all of them.
\end{itemize}
Approximating the leading term using just the second level in fact works better for larger $p$.  For example, when just looking at the second level there are $\Z_3$-towers with \eqref{eq:towererror} larger than $0.12$.  Of course, for those $\Z_3$-towers we have computed many more levels which support the conjectured leading term much better.
\end{example}
 
\section{Further Invariants for Basic Towers} \label{sec:basicgeneral}

We now investigate $\dk{r}{\cT(n)}$ for basic $\Z_p$-towers when $r >1$.  
We begin with a more precise version of Conjecture~\ref{conj:stablespecific} for basic $\Z_p$-towers which is a refinement of Conjecture~\ref{MC}.

\begin{notation} \label{notation:mr}
For fixed $r$ and $p$, write the rational number $\alpha(r,p) = \frac{r (p-1)}{2 (p+1) ( (p-1) r + (p+1))}$ from Notation~\ref{not:alpha} in lowest terms, and let $D$ be its denominator. Let $D'$ be the maximal divisor of $D$ which is prime to $p$.  
When $D' >1$ ({\em i.e.}~$D$ is not a power of $p$), we define
$m(r,p)$ to be the multiplicative order of $p^2$ modulo $D'$.
In the edge case that $D'=1$, we set $m(r,p)=0$.
\end{notation}

\begin{conjecture} \label{conj:basicpowers}
Fix a prime $p$ and positive integers $d$ and $r$.  If $m(r,p) =1$, then there exists a positive integer $N_{d,r}$ such that for any basic $\Z_p$-tower $\cT$ with ramification invariant $d$ there exists a rational number $c \in \bQ$ such that
\[
\dk{r}{\cT(n)} = \alpha(r,p) d p^{2n} + c \quad \text{ for } n \geq N_{d,r}.
\]
If $m(r,p)> 1$, then there exists a positive integer $N_{d,r}$ and $\lambda_{d,r}\in \bQ$ such that for any basic $\Z_p$-tower $\cT$ with ramification invariant $d$ there exists a function $c : \Z/m(r,p) \Z \to \Q$ such that
\[
\dk{r}{\cT(n)} = \alpha(r,p)d p^{2n} + c(n) + \lambda_{d,r} \cdot n \quad \text{ for } n \geq N_{d,r}.
\]
\end{conjecture}

Note that the denominator of $\alpha(1,p)$ is $4 (p+1) p$, and hence $m(1,p)=1$ for any prime $p$.  Thus this conjecture is compatible with Conjecture~\ref{conj:basicanumber}.

\begin{remark}
The definition of $m(r,p)$ is natural as $\dk{r}{\cT(n)}$ must be an integer while $\alpha(r,p)d p^{2n}$ is often not an integer.  With $D$ as in Notation \ref{notation:mr}, for $n$ sufficiently large the congruence class of $p^{2n}$ modulo $D$ depends only on $n$ modulo $m(r,p)$.  
To obtain an integer prediction for $\dk{r}{\cT(n)}$ in the conjecture, it is therefore natural to expect a formula depending on $n$ modulo $m(r,p)$.

When $m(r,p)=0$ (i.e. $D'=1$) we still expect $\dk{r}{\cT(n)}$ to be of the form $\alpha(r,p)d p^{2n} + c(n) + \lambda \cdot n $ with $c(n)$ a function with period $m \geq 1$.  However, we do not make any prediction for the period.  In these cases, 
it seems that $\lambda$ and $m$ may depend more subtly on the tower, rather than just on $d,p$ and $r$;
see Example~\ref{ex:p2m=0}.
\end{remark}

For convenience while testing this conjecture, for a $\Z_p$-tower $\cT$ and rational number $\lambda$ we define
\begin{equation}\label{eq:disrcdef}
    \delta_{d,r}(\cT(n) , \lambda) := \dk{r}{\cT(n)} - \left( \alpha(r,p) d p^{2n} + \lambda n \right);
\end{equation}
c.f. equation \eqref{eq:deltadef}.  Analogously, we define:

\begin{definition} \label{defn:discrepancy2}
An integer $n>m(r,p)$ is a \emph{discrepancy} of a basic tower $\cT$ with ramification invariant $d$
for the $r$th power of the Cartier operator if $$\delta_{d,r}(\cT(n),\lambda_{d,r}) \neq \delta_{d,r}(\cT(n-m(r,p)),\lambda_{d,r}).$$
\end{definition} 

Conjecture~\ref{conj:basicpowers} is equivalent to $\delta_{d,r}(\cT(n), \lambda_{d,r})$ being eventually periodic with period $m(r,p)$ (for an appropriate choice of $\lambda_{d,r}$).  Equivalently, the largest discrepancy with respect to the $r$th power of the Cartier operator for towers with ramification invariant $d$ should be bounded independently of the tower.

\subsection{Characteristic Two Examples}
Because the constant term in our conjectured formula will often depend on the congruence class of $n$ modulo $m(r,p)$, the best evidence for this conjecture comes from characteristic two, where it is feasible to compute with more levels of the towers.  
While we have been able to prove an exact formula for $\dk{r}{\cT(n)}$ in characteristic $p=2$ when $r=1$ (for all $n$) in Corollary \ref{cor:anumberbasic}, we have been unable to generalize this result to larger values of $r$.  As such,
the evidence we collect below is 
necessarily computational in nature.

\begin{example} \label{ex:p2d21}
We begin by considering two basic $\Z_2$-towers with ramification invariant $21$.  Tables~\ref{table:p2d21a} and \ref{table:p2d21b} show the genus and the dimension of the kernel for the first ten powers of the Cartier operator for the first seven levels of these two towers. 
We see that $a(\cT(n)) = a(\cT'(n))$ for all $1 \leq n \leq 7$ and we see that $a^r(\cT(1)) = a^r(\cT'(1))$ for $1 \leq r \leq 10$, which we prove in Corollary~\ref{cor:anumberbasic} and Lemma~\ref{lem:p2higher}(\ref{p2higher:1}).  Beyond that, $a^r(\cT(n))$ will depend on the tower $\cT$.  For example, $\dk{2}{\cT(3)} = 94 \neq \dk{2}{\cT'(3)} = 95$ and $\dk{3}{\cT(2)} = 31 \neq \dk{3}{\cT'(2)} = 33$.

\begin{table}[ht]
\centering
\begin{tabular}{c|c c c c c c c c }
n= & 1 & 2 & 3 & 4 & 5 & 6 & 7 \\ 
\toprule
$g(\cT(n))$  & 10 & 51 & 217 & 885 & 3565 & 14301 & 57277 \\ 
$\dk{1}{\cT(n)}$ & 5 & 16 & 58 & 226 & 898 & 3586 & 14338 \\ 
$\dk{2}{\cT(n)}$ & 8 & 25 & 94 & 363 & 1440 & 5741 & 22946 \\ 
$\dk{3}{\cT(n)}$ & 9 & 31 & 116 & 452 & 1796 & 7172 & 28676 \\ 
$\dk{4}{\cT(n)}$ & 10 & 36 & 131 & 517 & 2055 & 8198 & 32776 \\ 
$\dk{5}{\cT(n)}$ & 10 & 40 & 142 & 562 & 2242 & 8962 & 35842 \\ 
$\dk{6}{\cT(n)}$ & 10 & 43 & 152 & 603 & 2399 & 9563 & 38238 \\ 
$\dk{7}{\cT(n)}$ & 10 & 45 & 162 & 635 & 2515 & 10045 & 40150 \\ 
$\dk{8}{\cT(n)}$ & 10 & 47 & 169 & 660 & 2610 & 10432 & 41715 \\ 
$\dk{9}{\cT(n)}$ & 10 & 48 & 175 & 680 & 2696 & 10760 & 43016 \\ 
$\dk{10}{\cT(n)}$ & 10 & 49 & 180 & 696 & 2768 & 11031 & 44116 \\ 
\bottomrule
\end{tabular}
\caption{$\cT: Fy-y=[x^{21}] +[x^{19}] +[x^{15}] +[x^{13}] +[x^{9}]$ with $(p,d) = (2,21)$} \label{table:p2d21a}
\end{table}

\begin{table}[ht]
\centering
\begin{tabular}{c|c c c c c c c c }
n= & 1 & 2 & 3 & 4 & 5 & 6 & 7 \\ 
\toprule
$g(\cT'(n))$ & 10 & 51 & 217 & 885 & 3565 & 14301 & 57277 \\ 
$\dk{1}{\cT'(n)}$ & 5 & 16 & 58 & 226 & 898 & 3586 & 14338 \\ 
$\dk{2}{\cT'(n)}$ & 8 & 25 & 95 & 363 & 1441 & 5741 & 22947 \\ 
$\dk{3}{\cT'(n)}$ & 9 & 33 & 117 & 453 & 1797 & 7173 & 28677 \\ 
$\dk{4}{\cT'(n)}$ & 10 & 39 & 131 & 519 & 2057 & 8198 & 32778 \\ 
$\dk{5}{\cT'(n)}$ & 10 & 42 & 142 & 562 & 2242 & 8962 & 35842 \\ 
$\dk{6}{\cT'(n)}$ & 10 & 45 & 152 & 603 & 2400 & 9563 & 38238 \\ 
$\dk{7}{\cT'(n)}$ & 10 & 47 & 162 & 637 & 2515 & 10047 & 40150 \\ 
$\dk{8}{\cT'(n)}$ & 10 & 49 & 171 & 662 & 2610 & 10432 & 41718 \\ 
$\dk{9}{\cT'(n)}$ & 10 & 50 & 179 & 683 & 2699 & 10763 & 43019 \\ 
$\dk{10}{\cT'(n)}$ & 10 & 51 & 185 & 697 & 2769 & 11031 & 44116 \\ 
\bottomrule
\end{tabular}
\caption{$\cT' : Fy-y=[x^{21}] +[x^{13}] +[x^{9}] +[x^{5}] +[x^{3}]$ with $(p,d) = (2,21)$} \label{table:p2d21b}
\end{table}

\begin{table}
    \centering
    \begin{tabular}{c|c c c c c c c c c c}
     r    & 1 & 2 & 3 & 4 & 5 & 6 & 7 & 8 & 9 & 10 \\
         \toprule
     $21 \alpha(r,2)$ & 7/8 & 7/5 & 7/4 & 2 & 35/16 & 7/3 & 49/20 & 28/11 & 21/8 & 35/13 \\
     $m(r,2)$ & 1 & 2 & 1 & 3 & 1 & 3 & 2 & 5 & 0  & 6\\
    \bottomrule
    \end{tabular}
    \caption{Constants for $(p,d) = (2,21)$}
    \label{table:constants212}
\end{table}

Table~\ref{table:constants212} shows $21 \alpha(r,2)$ and $m(r,2)$ for $1 \leq r \leq 10$.
Our computations with the first seven levels of $\cT$ and $\cT'$ support Conjecture~\ref{conj:basicpowers}.
For example, we see that for $1 < n \leq 7$
\[
\dk{2}{\cT(n)} = \begin{cases}
\frac{7}{5}( 2^{2n}+1) + n & n \text{ odd} \\
\frac{7}{5} ( 2^{2n}-1) +n  +2 & n \text{ even}
\end{cases}
\quad
\dk{2}{\cT'(n)} = \begin{cases}
\frac{7}{5}( 2^{2n}+1) + n +1 & n \text{ odd} \\
\frac{7}{5} ( 2^{2n}-1) +n  +2 & n \text{ even.}
\end{cases}
\]
Note that $m(2,2)=2$ as expected.
Likewise for $2<  n \leq 7$ 
\[
\dk{3}{\cT(n)} = \frac{7}{4} \cdot 2^{2n} + 4, \quad \dk{3}{\cT'(n)} = \frac{7}{4} 2^{2n} + 5.  
\]
Furthermore,
\[
\dk{4}{\cT(n)} = \begin{cases}
2 \cdot 2^{2n} + n,  & {n \equiv 0 \pmod{3}} \\
2 \cdot 2^{2n} + n  + 1, & n \equiv 1 \pmod{3} \\
2 \cdot 2^{2n} + n  + 2, & n \equiv 2 \pmod{3} \\
\end{cases}
\quad
\dk{4}{\cT'(n)} = \begin{cases}
2 \cdot 2^{2n} + n,  & {n \equiv 0 \pmod{3}} \\
2 \cdot 2^{2n} + n  + 3, & n \equiv 1 \pmod{3} \\
2 \cdot 2^{2n} + n  + 4, & n \equiv 2 \pmod{3} \\
\end{cases}
\]
with $1 <n \leq 7$ for $\cT$ and $2 < n \leq 7$ for $\cT'$.  Considering the fifth power, for $2 < n \leq 7$
\[
\dk{5}{\cT(n)} = \dk{5}{\cT'(n)} = \frac{35}{16} 2^{2n} + 2.
\]
There appear to be similar formulas with $\lambda=1$ for $\dk{r}{\cT(n)}$ depending on $n$ modulo $3$ for $r = 6$ and depending on $n$ modulo $2$ for $r=7$.  These are all compatible with Conjecture~\ref{conj:basicpowers}.    There are not obvious formulas of a similar nature when $r=8$ or $r=10$, but our conjecture predicts that the formulas would depend on $n$ modulo $5$ or $6$.  With only seven levels of the tower and with the invariants taking a couple of levels to stabilize, we would not expect to see periodic behavior.  When $r=8$ and $r=10$, the dimensions are quite close to $\alpha(r,p) d p^{2n}$ as expected.  When $r=9$, as the denominator of $\alpha(9,2) = 21/8$ is a power of two we don't make a prediction for the period.  It appears that the period is one, as for $4 \leq n \leq 7$
\[
\dk{9}{\cT(n)} = \frac{21}{8} 2^{2n} + 8, \quad \dk{9}{\cT'(n)} = \frac{21}{8} 2^{2n} + 11.
\]
\end{example}

\begin{example}
Consider the $\Z_2$-towers $\cT: Fy - y = [x^9] + [x^3] + [x]$ and $\cT': Fy - y = [x^9] + [x]$.  It appears that $\lambda_{9,2} = 1/2 $, $\lambda_{9,4} = 1/3$, and $\lambda_{9,7} = 1/2$.  Table~\ref{table:p2d9} shows some selected values of $\delta_{d,r}(\cT(n),\lambda_{d,r})$ and $\delta_{d,r}(\cT'(n),\lambda_{d,r})$.  These all support Conjecture~\ref{conj:basicpowers}, which predicts that the tower will have period $2$ (resp. $3$, $2$) when $r=2$ (resp. $4$, $7$).  However, there are now larger discrepancies.
For example, it looks as if $a^2(\cT(n))  = a^2(\cT'(n)) = 3\cdot 2^{2n}/5 + n/2 + c(n)$ where $c(n) = 1/10$ if $n$ is odd and $c(n) = -3/5$ if $n$ is even, \emph{except} for $n=2,3$.  

\begin{table}[ht]
\centering
\begin{tabular}{c|c c c c c c c c }
Level: & 1 & 2 & 3 & 4 & 5 & 6 & 7 \\ 
\toprule
$\delta_{9,2}(\cT(n),1/2)$ & 1/10 & 2/5 & 11/10 & -3/5 & 1/10 & -3/5 & 1/10 \\ 
$\delta_{9,2}(\cT'(n),1/2)$ & 1/10 & 2/5 & 11/10 & -3/5 & 1/10 & -3/5 & 1/10 \\ 
$\delta_{9,4}(\cT(n),1/3)$ & 5/21 & 97/21 & 1/7 & 26/21 & 76/21 & 1/7 & 26/21 \\ 
$\delta_{9,4}(\cT'(n),1/3)$ & 5/21 & 55/21 & 1/7 & 26/21 & 76/21 & 1/7 & 26/21 \\ 
$\delta_{9,7}(\cT(n),1/2)$ & -7/10 & 16/5 & 3/10 & 21/5 & -7/10 & 16/5 & -7/10 \\ 
$\delta_{9,7}(\cT'(n),1/2)$ & -7/10 & 16/5 & 3/10 & 21/5 & -7/10 & 16/5 & -7/10 \\
\bottomrule
\end{tabular}
\caption{``Constant Terms'' for $\cT$ and $\cT'$, $(p,d) = (2,9)$} \label{table:p2d9}
\end{table}
\end{example}

\begin{example} \label{ex:p2longperiod}
The tower $Fy - y = [x^3]$ is simple enough that we have been able to compute with the eighth level, allowing us to see some slightly longer periods.  When $r=8$ (resp. $r=10$) observe that $m(r,2) = 5$ (resp. $m(r,2)=6$).  Table~\ref{table:p2d3} shows the beginnings of periodic behavior of the expected period.  This example is quite simple as the ramification invariant is so small; we estimate that $\lambda_{3,r}=0$ and the low levels of the tower do not appear to have any irregularities relative to the rest of the tower.
\begin{table}[ht]
\centering
\begin{tabular}{c|c c c c c c c c c }
Level: & 1 & 2 & 3 & 4 & 5 & 6 & 7 & 8 \\ 
\toprule
$\delta_{3,8}(\cT(n),0)$ & -5/11 & 2/11 & -3/11 & 10/11 & -4/11 & -5/11 & 2/11 & -3/11 \\ 
$\delta_{3,10}(\cT'(n),0)$ & -7/13 & -2/13 & 5/13 & -6/13 & 15/13 & -5/13 & -7/13 & -2/13 \\ 
\bottomrule
\end{tabular}
\caption{``Constant Terms'' for $Fy - y = [x^3]$, $p=2$} \label{table:p2d3}
\end{table}
\end{example}

\begin{example} \label{ex:p2highpowers}
When $r$ is large, it is difficult to test Conjecture~\ref{conj:basicpowers} as $m(r,p)$ is often too big to see periodic behavior given the number of levels we are able to compute.  Furthermore, as $V_{\cT(n)}$
is nilpotent, for any fixed $n$ the genus of $\cT(n)$ is equal to $\dk{r}{\cT(n)}$ for $r$ sufficiently large, which means we would need additional levels to see the behavior for large powers of the Cartier operator.

Consider the $\Z_2$-towers 
\begin{align*}
    \cT &: Fy - y = [x^{19}] + [x^{17}] + [x^{13}] + [x^5] + [x^3] \\
    \cT' &: Fy - y = [x^{19}] + [x^{17}] + [x^{15}] + [x^{11}] + [x^9] + [x^7] + [x^5] + [x].
\end{align*}We computed $\dk{r}{\cT(n)}$ for $r \leq 200$ and $n \leq 7$.  For $r=13$ and $r=17$, we see the expected behavior with periods $1$ and $2$ as predicted.  On the other hand, for $r=125$ our conjecture predicts the period to be one but we cannot see this; $\delta_{19,125}(\cT(n),0)$ and $\delta_{19,125}(\cT'(n),0)$ do not appear to be constant.  However, this is not so surprising as for $n \leq 5$ we have that $\dk{125}{\cT(n)} = g(\cT(n))$ and likewise for $\cT'$.  It is only for larger values of $n$ that we would expect to see the finer behavior of $\dk{125}{\cT(n)}$, and computing with $n \leq 7$ only gives two ``interesting'' levels.
\end{example}

\begin{example} \label{ex:p2m=0}
Our conjecture does not predict the period in the edge case that $m(r,2) =0$; this case appears more subtle.  For example, $\alpha(9,2) = 1/8 $ and hence $m(9,2)=0$, while Table \ref{table:p2d19} shows $\delta_{19,9}(\cT(n),0) = \dk{9}{\cT(n)} - 19 \cdot 2^{2n-3}$ for the two towers with ramification invariant $19$ in Example~\ref{ex:p2highpowers}.  It looks like the tower $\cT$ has period one and $\lambda =0$, while $\cT'$ has period two with $\lambda =1/2$.  
\end{example}

\begin{table}[ht]
\centering
\begin{tabular}{c|c c c c c c c c }
Level: & 1 & 2 & 3 & 4 & 5 & 6 & 7 \\ 
\toprule
$\delta_{19,9}(\cT(n),0)$ &  $-1/2$ & 8 & 8 & 8 & 8 & 8 & 8\\
$\delta_{19,9}(\cT'(n),0)$ & $-1/2$ & 5 & 6 & 6 & 7 & 7 & 8 \\
\bottomrule
\end{tabular}
\caption{``Constant Terms'' for $\cT$ and $\cT'$, $(p,d) = (2,19)$} \label{table:p2d19}
\end{table}

We have systematically tested Conjecture~\ref{conj:basicpowers} against a collection of at least $221$ basic $\Z_2$-towers where we analyzed at least $5$ levels (we analyzed seven levels for $55$ of them).   
For each ramification invariant $d$, we picked one tower $\cT_0$ where we had computed seven levels and used it to estimate $\lambda_{d,r}$ by computing\footnote{Recall we predict that $\lambda_{d,r}$ depends only on $d$ and $r$, and not on the specific tower.}
\[
\frac{\left( \dk{r}{\cT_0(7)} - \alpha(r,p) d p^{14} \right) - \left( \dk{r}{\cT_0(7-m(r,p))} - \alpha(r,p) d p^{2(7-m(r,p))} \right)}{m(r,p)}.
\]
If Conjecture~\ref{conj:basicpowers} held for $n \geq 7 - m(r,p)$, this ratio would equal $\lambda_{d,r}$. 
Furthermore, if Conjecture~\ref{conj:basicpowers} holds and we have the correct $\lambda_{d,r}$, for $n$ large enough $\delta_{d,r}(\cT(n),\lambda_{d,r})$ would equal $c(n)$.

Tables~\ref{table:discrepanciesp2r2}, \ref{table:discrepanciesp2r3}, and \ref{table:discrepanciesp2r4} show the discrepancies we have found in our database for $d<24$ and $r=2,3,4$.  Table~\ref{table:discrepanciesp2r2} shows the number of towers under consideration with each ramification invariant.   Note that the smallest possible discrepancy is $m(r,2)+1$, so is $3$ when $r=2$, $2$ when $r=3$, and $4$ when $r=4$.  These tables support Conjecture~\ref{conj:basicpowers} as the discrepancies appear to only occur for relatively small $n$ ($7$ is the largest potential discrepancy we would see using our data).  When $r=5$ (and $m(r,2)=1$) we only see discrepancies at levels two and three.

\begin{table}[ht]
\centering
\begin{tabular}{c|c c c c c c c c c c c c c c}
\toprule
$d$ & 3 & 5 & 7 & 9 & 11 & 13  & 15 & 17 & 19 & 21 & 23  \\ 
Discrepancies: & $\emptyset$  & $\{ 4, 5 \}$ & $\{ 3 \}$ & $\{ 3, 4, 5 \}$ & $\emptyset$  & $\{ 3, 5 
\}$ & $\{ 3, 4, 5 \}$ & $\{ 3 \}$ & $\emptyset$  & $\{ 3 \}$ & $\emptyset$  \\
Towers: & 2 & 4 & 8 & 14 & 24 & 31 & 40 & 36 & 42 & 6 & 6\\
\bottomrule
\end{tabular}
\caption{Some Observed Discrepancies for Basic $\Z_2$-Towers, $r=2$}
\label{table:discrepanciesp2r2}
\end{table}
\begin{table}[ht]
\centering
\begin{tabular}{c|c c c c c c c c c c c c c c}
\toprule
$d$ & 3 & 5 & 7 & 9 & 11 & 13  & 15 & 17 & 19 & 21 & 23  \\ 
Discrepancies: & $\emptyset$  & $\{ 2 \}$ & $\emptyset$  & $\{ 2 \}$ & $\{ 2, 4 \}$ & $\{ 2 \}$ & $\{ 2
\}$ & $\{ 2, 3, 4 \}$ & $\{ 2, 3, 6 \}$ & $\{ 2, 3 \}$ & $\{ 2, 4, 5 \}$
 \\
\bottomrule
\end{tabular}
\caption{Some Observed Discrepancies for Basic $\Z_2$-Towers, $r=3$}
\label{table:discrepanciesp2r3}
\end{table}
\begin{table}[ht]
\centering
\begin{tabular}{c|c c c c c c c c c c c c c c}
\toprule
$d$ & 3 & 5 & 7 & 9 & 11 & 13  & 15 & 17 & 19 & 21 & 23  \\ 
Discrepancies:  & $\emptyset$  & $\{ 5 \}$ & $\emptyset$  & $\{ 4, 5 \}$ & $\{ 5 \}$ & $\{ 4, 5 \}$ & 
$\{ 5, 6 \}$ & $\{ 4, 5 \}$ & $\{ 4, 5 \}$ & $\{ 4, 5 \}$ & $\{ 4, 5, 7 \}$  \\
\bottomrule
\end{tabular}
\caption{Some Observed Discrepancies for Basic $\Z_2$-Towers, $r=4$}
\label{table:discrepanciesp2r4}
\end{table}

\begin{remark}
In all of the examples we have computed, $| \delta_{d,r}(\cT(n) , \lambda_{d,r}) - \delta_{d,r}(\cT(n+1),\lambda_{d,r})| \leq 1 $ when $r=2$ (respectively $\leq 3$ when $r=3$ and $\leq 2$ when $r=4$).  
\end{remark}



\subsection{Characteristic Three}  Now let $p=3$.  The evidence for Conjecture~\ref{conj:basicpowers} is a bit weaker in characteristic three as our computations are limited to at most 5 levels.

\begin{example}
We begin by considering two basic $\Z_3$-towers with ramification invariant $5$.  Tables~\ref{table:p3d5a} and \ref{table:p3d5b} show the genus and the dimension of the kernel of the first ten powers of the Cartier operator for the first five levels of these two towers.  Table~\ref{table:constants35} shows $5 \alpha(r,3) $ and $m(r,3)$.  These examples support Conjecture~\ref{conj:basicpowers}. 

\begin{table}[ht]
\centering
\begin{tabular}{c|c c c c c c }
n= & 1 & 2 & 3 & 4 & 5 \\ 
\toprule
$g(\cT(n))$ & 4 & 46 & 442 & 4060 & 36784 \\ 
$\dk{1}{\cT(n)}$ & 2 & 19 & 154 & 1369 & 12304 \\ 
$\dk{2}{\cT(n)}$ & 4 & 26 & 230 & 2052 & 18456 \\ 
$\dk{3}{\cT(n)}$ & 4 & 31 & 275 & 2461 & 22145 \\ 
$\dk{4}{\cT(n)}$ & 4 & 35 & 305 & 2735 & 24605 \\ 
$\dk{5}{\cT(n)}$ & 4 & 39 & 326 & 2930 & 26365 \\ 
$\dk{6}{\cT(n)}$ & 4 & 42 & 344 & 3076 & 27680 \\ 
$\dk{7}{\cT(n)}$ & 4 & 45 & 362 & 3197 & 28712 \\ 
$\dk{8}{\cT(n)}$ & 4 & 46 & 368 & 3281 & 29525 \\ 
$\dk{9}{\cT(n)}$ & 4 & 46 & 374 & 3358 & 30197 \\ 
$\dk{10}{\cT(n)}$ & 4 & 46 & 380 & 3422 & 30756 \\ 
\bottomrule
\end{tabular}
\caption{$\cT: Fy-y=[x^{5}] +[2x^{2}]$  with $(p,d) = (3,5)$} \label{table:p3d5a}
\end{table}

\begin{table}[ht]
\centering
\begin{tabular}{c|c c c c c c }
n= & 1 & 2 & 3 & 4 & 5 \\ 
\toprule
$g(\cT(n))$ & 4 & 46 & 442 & 4060 & 36784 \\ 
$\dk{1}{\cT(n)}$ & 2 & 18 & 153 & 1368 & 12303 \\ 
$\dk{2}{\cT(n)}$ & 4 & 26 & 230 & 2052 & 18456 \\ 
$\dk{3}{\cT(n)}$ & 4 & 31 & 275 & 2461 & 22145 \\ 
$\dk{4}{\cT(n)}$ & 4 & 35 & 305 & 2735 & 24605 \\ 
$\dk{5}{\cT(n)}$ & 4 & 39 & 326 & 2930 & 26365 \\ 
$\dk{6}{\cT(n)}$ & 4 & 42 & 344 & 3076 & 27680 \\ 
$\dk{7}{\cT(n)}$ & 4 & 45 & 360 & 3195 & 28710 \\ 
$\dk{8}{\cT(n)}$ & 4 & 46 & 368 & 3281 & 29525 \\ 
$\dk{9}{\cT(n)}$ & 4 & 46 & 374 & 3358 & 30197 \\ 
$\dk{10}{\cT(n)}$ & 4 & 46 & 380 & 3422 & 30756 \\ 
\bottomrule
\end{tabular}
\caption{$\cT' : Fy-y=[x^{5}] +[2x^{4}] +[2x]$ with $(p,d) = (3,5)$} \label{table:p3d5b}
\end{table}

\begin{table}[ht]
    \centering
    \begin{tabular}{c|c c c c c c c c c c}
     r    & 1 & 2 & 3 & 4 & 5 & 6 & 7 & 8 & 9 & 10 \\
         \toprule
     $5 \alpha(r,3)$ & 5/24 & 5/16 & 3/8 & 5/12 & 25/56 & 15/32 & 35/72 & 1/2 & 45/88 & 25/48 \\
     $m(r,3)$ & 1 & 2 & 2 & 1 & 3 & 4 & 1 &2 & 5 & 2\\
    \bottomrule
    \end{tabular}
    \caption{Constants for $(p,d) = (3,5)$}
    \label{table:constants35}
\end{table}

 For $1 < n \leq 5$, observe that
\[
\dk{2}{\cT(n)} = \dk{2}{\cT'(n)} = \begin{cases}
\frac{5}{16} ( 3^{2n} -9) + \frac{n-1}{2} + 4 & n \text{ odd} \\
\frac{5}{16} ( 3^{2n} -1) + \frac{n}{2}  & n \text{ even} 
\end{cases}
\]
while $m(2,3) = 2$ as expected.  Similarly, for $1 < n \leq 5$ we see
\[
\dk{3}{\cT(n)} = \dk{3}{\cT'(n)} = \begin{cases}
\frac{3}{8} ( 3^{2n} -1) + 2 & n \text{ odd} \\
\frac{3}{8} ( 3^{2n} -1) + 1  & n \text{ even.} 
\end{cases}
\]
Furthermore, for $1 < n \leq 5$ we have
\[
\dk{4}{\cT(n)} = \dk{4}{\cT'(n)} = \frac{5}{12} (3^{2n}-9) +5.
\]
We expect $\dk{5}{\cT(n)}$ to depend on $n$ modulo $3$, and one might optimistically conjecture that for $n \geq 1$
\[
\dk{5}{\cT(n)} = \dk{5}{\cT'(n)} =  \begin{cases}
\frac{25}{56} ( 3^{2n} - 9) + \frac{n-1}{3} + 4 & n \equiv 1 \pmod{3} \\
\frac{25}{56} ( 3^{2n} - 25) + \frac{n-2}{3} + 14 & n \equiv 2 \pmod{3} \\
\frac{25}{56} ( 3^{2n} - 1) + \frac{n}{3}  & n \equiv 0 \pmod{3}. \\
\end{cases}
\]
This is consistent with our data but is weaker evidence for Conjecture~\ref{conj:basicpowers}, as there is only one multiple of three for which we have computed with $\cT(n)$.  As we have chosen the constant term of the $n \equiv 0 \pmod{3}$ case so that $\dk{5}{\cT(3)}$ is correct, that case is somewhat vacuous.

It is somewhat of a coincidence that $\dk{r}{\cT(n)} = \dk{r}{\cT'(n)}$ for $r = 2,3,4,5$.  These are not equal when $r=1$, or when $r=7$ where we find
\[
\dk{7}{\cT(n)} = \frac{35}{72} (3^{2n} -9^2) + 47 \quad\text{and}\quad \dk{7}{\cT'(n)} = \frac{35}{72} (3^{2n} -9^2) + 45 
\]
for $3 \leq n \leq 5$.  
 When $r=10$ we see a similar formula with the correct leading term, $\lambda=1/2$, and period $2$.  When $r=8$, observe that for $n=4,5$
  \[
 \dk{8}{\cT(n)} = \dk{8}{\cT'(n)} =  3^{2n}/2 + 1/2.
 \]
This suggests the $\dk{8}{\cT(n)}$ may have period one, while the predicted period is $m(8,3)=2$.  (This does not  contradict Conjecture \ref{conj:basicpowers}, as any function $c: \Z/m\Z\rightarrow \Q$
may be considered as a function on $\Z/m\ell \Z$ for each positive integer $\ell$.)
There are no obvious periodic formulas when $r=6$ or $r=9$, but our conjecture predicts these would depend on $n$ modulo $4$ or $5$, so with only five levels of the tower we can not expect to witness periodic behavior.  In these cases, the dimensions are quite close to $\alpha(r,3) \cdot d \cdot 3^{2n}$ as expected. 
\end{example}

We can systematically test Conjecture~\ref{conj:basicpowers} against the collection of basic $\Z_3$-towers described in Section~\ref{ss:evidencep3} where we had computed invariants for four or five levels.  For most of these towers, we computed with the first five powers of the Cartier operator.  (The unusual case that $m(r,3) =0$ does not occur for $r < 5$.)  For fixed $d$ and $r$, we used one of the towers $\cT$ where we had computed five levels (often $Fy - y = x^d$) to predict $\lambda_{d,r}$ by computing
\[
\frac{\left( \dk{r}{\cT(5)} - \alpha(r,p) d p^{10} \right) - \left( \dk{r}{\cT(5-m(r,p))} - \alpha(r,p) d p^{2(5-m(r,p))} \right)}{m(r,p)}.
\]
If Conjecture~\ref{conj:basicpowers} holds for $\cT(n)$ with $n \geq 5 - m(r,p) > N_d$, this ratio is precisely $\lambda_{d,r}$.

Using this prediction for $\lambda_{d,r}$, Table~\ref{table:discrepanciesp3r4} shows all the discrepancies (recall Definition~\ref{defn:discrepancy2}) with $r=4$ and $d<30$, and supports Conjecture~\ref{conj:basicpowers}.  Note that in this situation $m(r,3)=1$. 

\begin{table}[ht]
\centering
\begin{tabular}{c|c c c c c c c c c c c c c c}
\toprule
$d$ & 2 & 4 & 5 & 7 & 8 & 10  & 11 & 13 & 14 & 16  \\ 
Discrepancies: & $\emptyset$  &  $\emptyset$ &    \{ 2 \} &    \{ 2 \} &   $\emptyset$ &    \{ 2 \} &    \{ 2 \} &    \{ 2 \} &    \{ 2 \} &    \{ 2 \} \\
Towers:  & 4 & 25 & 13 & 40 & 25 & 25 & 25 & 36 & 25 & 36 \\
\hline
$d$ & 17 & 19 & 20 & 22 & 23 & 25 & 26 & 28 & 29 \\
 Discrepancies:  &    \{ 2, 4 \} &    \{ 2 \} &    \{ 2 \} &     \{ 2, 4 \} &    \{ 2 \} &    \{ 2 \} &    \{ 2 \} &    \{ 2 \} &    \{ 2, 3, 4 \} \\
 Tower: &36 &  36 & 36 & 47 & 37 & 47 & 46 & 48 & 47\\
\bottomrule
\end{tabular}
\caption{Observed Discrepancies for Basic $\Z_3$-Towers, $r=4$}
\label{table:discrepanciesp3r4}
\end{table}

Similarly, Tables~\ref{table:discrepanciesp3r2} and \ref{table:discrepanciesp3r3} support Conjecture~\ref{conj:basicpowers} in that the formulas in the conjecture depend on the parity of $n$.  Note that it is not possible to have a discrepancy at level 2 in this situation as $m(r,3)>1$.  

\begin{table}[ht]
\centering
\begin{tabular}{c|c c c c c c c c c c c c c c}
\toprule
$d$ & 2 & 4 & 5 & 7 & 8 & 10  & 11 & 13 & 14 & 16  \\ 
Discrepancies: & $\emptyset$ & $\emptyset$  & $\emptyset$  & $\emptyset$  & $\emptyset$  & $\emptyset$  & 
$\emptyset$ & \{ 3 \} & \{ 3 \} & $\emptyset$   \\
\hline
$d$ & 17 & 19 & 20 & 22 & 23 & 25 & 26 & 28 & 29 \\
 Discrepancies:  & \{ 3 \} & \{ 3 \} & \{ 3 \} 
& \{ 3 \} & \{ 3 \} & \{ 3 \} & \{ 3, 4 \} & \{ 3 \} & $\emptyset $\\
\bottomrule
\end{tabular}
\caption{Observed Discrepancies for Basic $\Z_3$-Towers, $r=2$}
\label{table:discrepanciesp3r2}
\end{table}

\begin{table}[ht]
\centering
\begin{tabular}{c|c c c c c c c c c c c c c c}
\toprule
$d$ & 2 & 4 & 5 & 7 & 8 & 10  & 11 & 13 & 14 & 16  \\ 
Discrepancies: & $\emptyset$  & $\emptyset$  & $\{ 3 \}$ & $\emptyset$  & $\{ 3 \}$ & $\{ 3 \}$ & 
$\emptyset$  & $\emptyset$  & $\{ 3 \}$ & $\{ 3 \}$   \\
\hline
$d$ & 17 & 19 & 20 & 22 & 23 & 25 & 26 & 28 & 29 \\
 Discrepancies:    & $\{ 3 \}$ & $\{ 3 \}$ & $\{ 3 \}$ & 
$\{ 3 \}$ & $\{ 3 \}$ & $\{ 3 \}$ & $\{ 3 \}$ & $\{ 3 \}$ & $\{ 3 \}$\\
\bottomrule
\end{tabular}
\caption{Observed Discrepancies for Basic $\Z_3$-Towers, $r=3$}
\label{table:discrepanciesp3r3}
\end{table}


%

Since $m(5,3)=3$, we only consider the asymptotic behavior of $\dk{5}{\cT(n)}$ as it is not feasible to spot patterns with period $3$ using only five levels.  Out of all of the towers $\cT$ we analyzed, the maximum value of 
\begin{equation}
\left| \frac{\dk{5}{\cT(n)}}{\alpha(r,3)\cdot d \cdot 3^{2n}} - 1\right| 
\end{equation}
is less than $.0021$ for $n=3$, less than $.0013$ for $n=4$, and less than $.00024$ for $n=5$.  This supports Conjecture~\ref{conj:stableasymptotic} and that we have the correct main term in Conjecture~\ref{conj:basicpowers}.

\subsection{Other Characteristics}  For $p>3$, our computations are necessarily more limited in scope.
Nevertheless, we record several examples with $p>3$ below.


\begin{example} \label{ex:p5powers}
When $p =5$, we compute that $m(5,5) =2$ and $m(r,5)>2$ for $r =2,3,4$.  Thus we focus first on the case that $r=5$, where there is a hope of seeing periodic behavior with just four levels of a $\Z_5$-tower.  By eyeballing the towers $\cT_d : Fy - y = x^d$ with $d \leq 12$, it looks like $\lambda_{d,5} =0$ for $d < 7$ and $\lambda_{d,5} = 1/2$ for $7 \leq d \leq 12$.  Table~\ref{table:p5powers} shows that $\delta_{d,5}(\cT(n),\lambda_{d,5})$ appears to be periodic with period two (as expected) for these towers.  This again supports Conjecture~\ref{conj:basicpowers}.
 \begin{table}[ht]
\centering
\begin{tabular}{c|c c c c  || c | c  c c c }
Level: & 1 & 2 & 3 & 4 & Level: & 1 & 2 & 3 & 4\\ 
\toprule
$\delta_{3,5}(\cT(n),0)$ & -21/26 & -5/26 & -21/26 & -5/26  & $\delta_{8,5}(\cT(n),1/2)$ & 53/78 & -20/39 & 53/78 & -20/39 \\ 

$\delta_{4,5}(\cT(n),0)$ & -16/39 & -10/39 & -16/39 & -10/39  & $\delta_{9,5}(\cT(n),1/2)$ & 14/13 & -15/26 & 14/13 & -15/26 \\ 
 
$\delta_{6,5}(\cT(n),0)$ & 5/13 & -5/13 & 5/13 & -5/13 & $\delta_{11,5}(\cT(n),1/2)$ & 73/39 & -55/78 & 73/39 & -55/78 \\ 

$\delta_{7,5}(\cT(n),1/2)$ & 11/39 & -35/78 & 11/39 & -35/78 & $\delta_{12,5}(\cT(n),1/2)$ & 59/26 & -10/13 & 59/26 & -10/13 \\ 
\bottomrule
\end{tabular}
\caption{$\cT_d: Fy-y=[x^d]$ with $3 \leq d \leq 12$, four levels}
\label{table:p5powers}
\end{table}

For $r \in \{2,3,4\}$, we can only meaningfully investigate the leading term.  We computed
\[
\left| \frac{ \dk{r}{\cT(n)}}{\alpha(r,5) d 5^{2n}} -1 \right|
\]
for these towers: with $n=3$ the maximum value was less than $.000273$ (resp. $.000266$, $.0021$) when $r=2$ (resp. $r=3$, $r=4$).  For $n=4$, the maximum value was less than $6.4 \cdot 10^{-5}$ (resp. $6.4 \cdot 10^{-5}$, $9.5 \cdot 10^{-5}$) when $r=2$ (resp. $r=3$, $r=4$).
 \end{example}
 
 \begin{example} \label{ex:plargepowers}
When $p>5$, we were only able to compute with two levels.  We computed
\[
\left| \frac{\dk{r}{\cT(2)}}{\alpha(r,p) d p^4} - 1 \right|
\]
for a variety of basic $\Z_p$-towers with ramification invariant $d$.  
\begin{itemize}
    \item  When $p=7$, we analyzed the second level of around a thousand $\Z_7$-towers.  The above quantity was always less than $.02$ for $r \in \{2,3,4,5\}$.
    
    \item  When $p=11$, we analyzed the second level of around $650$ $\Z_{11}$-towers.  The above quantity was always less than $.003$ for $r \in \{2,3,4,5\}$.
    
    \item  When $p=13$, we analyzed the second level of eleven $\Z_{13}$-towers.  The above quantity was always less than $.0042$ for $r \in \{2,3,4,5\}$.
\end{itemize}
Again, this supports the formula for the leading term in Conjecture~\ref{conj:basicpowers}.
 \end{example}

\section{Beyond Basic Towers} \label{sec:beyondbasic}

In Sections~\ref{sec:basicanumber} and \ref{sec:basicgeneral}, we focused on basic $\Z_p$-towers due to their simplicity.
In this section, we provide computational evidence that Conjectures~\ref{conj:stableasymptotic}, \ref{conj:stableanumber}, and \ref{conj:stablespecific} hold for other monodromy stables towers, and provide evidence that Philosophy~\ref{philosophy} holds for non-monodromy stable towers.  

\subsection{Monodromy Stable Towers with the Same Ramification as Basic Towers}  
So far, we have focused on basic $\Z_p$-towers as they have a particularly simple description using Artin-Schreier-Witt theory.  Now we consider more complicated $\Z_p$-towers that are totally ramified over a single point and have the same ramification as a basic $\Z_p$-tower.  To do this, we pick basic $\Z_p$-towers $\cT_{\basic} : Fy - y = \sum_{i=1}^d  [c_i x^i]$, and consider the related $\Z_p$-towers 
\[
\cT_{\modded} : Fy - y = \sum_{i=1}^d  [c_i x^i] + \sum_{j=1}^{d-1} d_j p [x^j] = \sum_{i=1}^d (c_i x^i,0,0,\ldots)  + \sum_{j=1}^{d-1}(0,d_j x^{jp},0,\ldots)
\]
where we let $d_j $ be $0$ or $1$ at random when $p \nmid j$ (and $d_j=0$ when $p \mid j$).
The first level of these modified towers agree with that of the basic tower, while higher levels do not.  However, by Fact~\ref{fact:localasw} we know they have the same ramification breaks above infinity
(and are unramified elsewhere).

We did this extensively in characteristic $p=3$, picking around $100$ basic towers with ramification invariants up to $19$ and considering ten modifications of each.  We computed $\dk{r}{\cT(n)}$ for the first four levels of all of these towers and $1 \leq r \leq 10$.  The modified towers always supported Conjectures~\ref{conj:stableasymptotic}, \ref{conj:stableanumber}, and \ref{conj:stablespecific}.  In fact, we almost always found that $\dk{r}{\cT_{\basic}(n)} = \dk{r}{\cT_{\modded}(n)}$.  There were only some scattered examples where they differed, and only for $r=8$.


\subsection{Towers Ramified Over Multiple Points}

We now consider monodromy stable towers of curves which are totally ramified over multiple points.  Because of the multiple points of ramification, we cannot use the program described in Section~\ref{sec:computing}.
As it is quite slow to compute examples without this program, we content ourselves with a couple of examples with $a$-numbers in characteristic $p=3$ and a general result in characteristic $p=2$.

\begin{example}
Let $p=3$, and consider the towers over $\PP^1_{\FF_p}$ defined by the Artin-Schreier-Witt equations
\[
\cT: Fy - y = [x^5] + [x^{-5}] \quad \text{and} \quad \cT' : Fy - y = [x^7] + [x^{-5}].
\]
Using {\sc Magma}'s built-in functionality
for computing $a$-numbers, we can compute the $a$-numbers of the first four levels of these towers.  These are shown in Table~\ref{table:multiplep=3}, along with data for the the basic towers $\cT_5$ and $\cT_7$ given by $Fy- y = [x^5]$ and $Fy - y = [x^7]$ respectively.  We were unable to investigate higher levels as {\sc Magma}'s built-in  functionality for computing with the Cartier operator is much less efficient than the (inapplicable) methods of Section~\ref{sec:computing}; for example, computing $a(\cT'(4))$ took around 38 hours.  
(For reference, Table~\ref{table:runningtimes} gives a systematic comparisons of the running of time of our algorithm and {\sc Magma}'s default methods when they both apply.)


\begin{table}[ht]
\centering
\begin{tabular}{c|c c c c c }
$n=$ & 1 & 2 & 3 & 4 \\ 
\toprule
$g(\cT(n))$ & 10 & 100 & 910 & 8200  \\ 
$a(\cT(n))$ & 4 & 36 & 306 & 2736 \\
$g(\cT'(n))$ &12 & 120 & 1092 & 9840  \\
$a(\cT'(n))$ & 6 & 44 & 368& 3284\\
\\
$g(\cT_5(n))$ & 4 & 46 & 442 & 4060 \\
$a(\cT_5(n))$ & 2 & 19 & 154 & 1369 \\
$g(\cT_7(n))$ & 6 & 66 & 624 & 5700  \\
$a(\cT_7(n))$ & 4 & 25 & 214 & 1915\\
\bottomrule
\end{tabular}
\caption{Invariants of $\cT$, $\cT', \cT_5,$ and $\cT_7$, characteristic $3$, levels $1$--$4$}
\label{table:multiplep=3}
\end{table}

We have that $\alpha(3) = 1/24$.   
For $2 \leq n \leq 4$ notice that 
\begin{align*}
a(\cT) &= 5 \alpha(3) ( 3^{2n}-9) + 5 \alpha(3) ( 3^{2n}-9) + 6 = \frac{5}{12} 3^{2n} + \frac{9}{4} \\
a(\cT') &= 5 \alpha(3) (3^{2n}-9) + 7\alpha(3) (3^{2n}-9) + 8 = \frac{1}{2} 3^{2n} + \frac{7}{2}.
\end{align*}
These support Conjecture~\ref{conj:stableasymptotic} and Conjecture~\ref{conj:stableanumber}.  
Note that the $a$-numbers for $\cT$ and $\cT'$ are almost a ``sum'' of the $a$-numbers of the basic towers:\footnote{Note there is an isomorphism of $\cT_5$ with the tower $Fy-y=[x^{-5}]$ lying over the automorphism $x\mapsto x^{-1}$ of $\PP^1$.}
we see that for $2 \leq n \leq 4$
\[
a(\cT(n)) = a(\cT_5(n)) + a(\cT_5(n))-2 \quad \text{and} \quad a(\cT'(n)) = a(\cT_5(n)) + a(\cT_7(n)) .
\]
This supports Philosophy~\ref{philosophy:sum}, as each point of ramification makes a contribution to the $a$-number.
\end{example}

\begin{example} \label{igusap3}
Consider the Igusa tower $\Ig$ in characteristic three as in Example~\ref{ex:igusaproperties}.
 There are two supersingular points of $\Ig(1) \simeq X_1(5)_k$ (this uses that $p=3$), so the tower given by  $\cT(n) := \Ig(n+1)$ is totally ramified over two points and unramified elsewhere.  The ramification invariant at level $n$ above each of the points is $9^{n-1}-1$.  We know the genus is
\[
g(\Ig(n)) = 3 \cdot 3^{2(n-1)} - 4 \cdot 3^{n-1} +1.
\]

\begin{table}[ht]
\centering
\begin{tabular}{c|c c c }
$n=$ & 1 & 2 & 3  \\ 
\toprule
$g(\Ig(n))$ & 0 & 16 & 208 \\
$a(\Ig(n))$ & 0 & 8 & 80 \\
$\dk{2}{\Ig(n)}$ & 0 & 12 & 120 \\
$\dk{3}{\Ig(n)}$ & 0 & 14 & 144 \\
\bottomrule
\end{tabular}
\caption{Invariants of $\Ig(n)$, characteristic $3$,  $3$ levels}
\label{table:igusap3}
\end{table}

Table~\ref{table:igusap3} shows invariants of $\Ig(n)$ for those small values of $n$ where we could compute it.\footnote{The analogous computation of the $a$-number for level 4 (genus $2080$) ran for over 1005 hours, using 23 GB of memory, without completing.}  In particular, notice that it appears $a(\Ig(n))=  3^{2(n-1)} - 1$.
Using Lemma~\ref{lem:breaks}, we see that $n$th break in the lower numbering filtration of $\cT$ is $12 \cdot 3^{n-1} - 4$ above each of the ramified points.  Then as $12 \alpha(3) = 1/2$, Conjecture~\ref{conj:stableanumber} predicts the $a$-number of $\Ig(n) = \cT(n-1) $ will be 
\[
(12 + 12) \alpha(3) 3^{2(n-1)} + c = 3^{2(n-1)}+ c \text{ for } n \gg 0.
\]
Our data for $a$-numbers therefore supports Conjecture~\ref{conj:stableasymptotic} and Conjecture~\ref{conj:stableanumber}.  Likewise, it appears that
\[
\dk{2}{\Ig(n)} = \frac{3}{2}\cdot 3^{2(n-1)} - \frac{3}{2} 
\]
in line with Conjecture~\ref{conj:stableasymptotic} and Conjecture~\ref{conj:stablespecific}.  Similarly for the third power, 
Conjecture~\ref{conj:stableasymptotic} predicts that $\dk{3}{\Ig(n)}$ is asymptotically $9/5 \cdot 3^{2(n-1)}$.  We compute that
\[
 \frac{\dk{3}{\Ig(2)}}{ 9/5 \cdot 3^2} \approx .864, \quad \text{and} \quad \frac{\dk{3}{\Ig(3)}}{ 9/5 \cdot 3^4} \approx .988.
\]
(Note that $\frac{9}{5}3^{2(n-1)}- \frac{9}{5} =\dk{3}{\Ig(n)}$ for $n=1,3$, but for $n=2$ this expression isn't an integer.  But taking its floor gives  $\dk{3}{\Ig(2)}$.)
Again, these support Conjecture~\ref{conj:stableasymptotic} for monodromy stable towers with multiple points of ramification and reflect Philosophy~\ref{philosophy:sum}.

To compute with this Igusa tower, we worked with the universal elliptic curve over $X_1(5)$
 \[
 E : y^2 + (1 + t)xy + ty = x^3 + tx^2.
 \]
 (Obtaining this equation is a relatively standard calculation, for example carried out in \cite[\S2.2]{silverberg}.)
 To obtain the function field of $\Ig(n)$ we adjoin the kernel of the Verschiebung $V^{n} : E^{(p^n)} \to E$ to $\F_p(t)$.  We can obtain a formula for $V$ from the multiplication by 3 map on $E$, and then iterate it (with appropriate Frobenius twists on coefficients) to obtain a polynomial with the coordinates of $\ker V^n$ as roots.  Given this description of the function field of $\Ig(n)$, {\sc Magma} can compute a basis for the regular differentials on $\Ig(n)$ and a matrix for the Cartier operator with respect to this basis.
\end{example}


We can also \emph{prove} similar behavior for $a$-numbers in towers with multiple points of ramification happens in characteristic two under a technical hypothesis - see Corollary~\ref{cor:anumberformula}.  Again, each point of ramification makes a contribution.

\begin{example} \label{ex:p2failtechnical}
Let $p=2$, and consider the $\Z_2$-towers $\cT: F y - y = [x^ 3] + [x^{-3}] + [(x-1)^{-3}]$ and $\cT' : Fy - y = [x^ 3] + [x] + [x^{-1}] + [(x-1)^{-5}]$.  These are towers over the projective line ramified over $0$, $1$, and $\infty$.  We have that $d_0(\cT(n)) = d_\infty(\cT(n)) = d_1(\cT(n)) = d_0(\cT'(n)) = 2^{2n-1}+1$, that $d_1(\cT'(n)) = \frac{5}{3} (2^{2n-1}+1)$, and that $d_\infty(\cT'(n)) =   (2^{2n-1}+1)/3$.
Table~\ref{table:3pointtower} shows data for the first four levels of these towers.  (As there are multiple points of ramification, we can only use {\sc magma}'s slower generic methods, so look at fewer levels.)  The $a$-numbers satisfy the formula \eqref{eq:anumberformula} even though the technical hypothesis 
\[
    \sum_{Q \in S} (d_Q(\cT(n))-1)/2 \geq 2 g(\cT(n-1))-2,
\]
is not satisfied for $n=4$ (and similarly for $\cT'$).    

\begin{table}[ht]
\centering
\begin{tabular}{c|c c c c }
$n=$ & 1 & 2 & 3 & 4  \\ 
\toprule
$g(\cT(n))$ & 5 & 24 & 98 & 390 \\
$a(\cT(n))$ & 3 & 6 & 24 & 96 \\
$g(\cT'(n))$ & 5 & 24 & 98 & 390 \\
$a(\cT'(n))$ & 2 & 7 & 25 & 97 \\
\bottomrule
\end{tabular}
\caption{Tower over $\PP^1$ ramified at three points,  $p=2$}
\label{table:3pointtower}
\end{table}
\end{example}

\subsection{Towers Over Other Bases}

We now briefly discuss $\Z_p$-towers whose base is not the projective line.
A simple way to obtain such towers is to start with a $\Z_p$-tower $\cT$
over the projective line, and---for any fixed $m$---forget the first $m$ levels of the tower to
obtain a $\Z_p$-tower over $\cT(m)$. It is worth pointing out that our conjectures are compatible with this procedure:

\begin{lem}
Suppose $\cT$ is a monodromy stable $\Z_p$-tower totally ramified above a set $S$, and for a fixed integer $m \geq 1$ let $\cT'$ be the $\Z_p$-tower $\ldots \to \cT(m+1) \to \cT(m)$. Then $\cT$ satisfies Conjecture~\ref{conj:stableasymptotic} (resp. Conjecture~\ref{conj:stableanumber} or Conjecture~\ref{conj:stablespecific}) if and only if $\cT'$ does.
\end{lem}

\begin{proof}
This is essentially \cite[Proposition 5.5]{kmmonodromy}, though we include a proof for the convenience of the reader.
Note that $\cT'$ is totally ramified ramified above a set $S'$ of points in $\cT(m)$, and for each $Q \in S$ there is a unique point $Q' \in S'$ lying above it.  As the Galois group of the tower $\cT'$ is a subgroup of the Galois group of the tower $\cT$, we can directly compare the lower numbering filtrations.  In particular, $d_Q(\cT(n+m)) = d_{Q'}(\cT'(n))$.  

Note that for a tower $T$ ramified over $Q$, having $s_Q(T(n)) = d p^{n-1} + c$ for $n\gg 0$ is equivalent to $s_Q(T(n+1)) - s_Q(T(n)) = d (p-1)p^{n-1}$ for $n\gg 0$.
Using Lemma~\ref{lem:breaks}, we compute 
\begin{align*}
    s_{Q'}(\cT'(n+1))) - s_{Q'}(\cT'(n)) & = ( d_{Q'}(\cT'(n+1)) - d_{Q'}(\cT'(n))) p^{-n} \\
    &=  ( d_{Q}(\cT(m+n+1)) - d_{Q}(\cT(m+n))) p^{-n} \\
    &= (s_{Q}(\cT(m+n+1)) - s_Q(\cT(m+n))) p^{m}
\end{align*}
Since $\cT$ is monodromy stable, there exists $d_Q \in \Q$ such that $s_Q(\cT(m+n+1)) - s_Q(\cT(m+n)) = d_Q (p-1) p^{n+m-1}$ for $n$ large enough.  Thus we see that $s_{Q'}(\cT'(n+1)) - s_{Q'}(\cT'(n)) = d_Q  p^{2m} (p-1) p^{n-1}$ for $n$ large enough and hence $\cT'$ is monodromy stable.  In particular, taking $d_{Q'} = d_Q p^{2m}$ we have
$s_Q(\cT'(n))= d_{Q'} p^{n-1} + c_{Q'}$ for $n \gg 0$.  

Now notice that
\[
\alpha(r,p) d_Q  p^{2(n+m)}= \alpha( r,p) d_{Q'} p^{2n};
\]
the left side is the contribution of $Q$ to the leading term of $\dk{r}{\cT(n+m)}$ predicted by Conjecture~\ref{conj:stableasymptotic}, and
the right side is the contribution of $Q'$ to $\dk{r}{\cT'(n)}$ predicted by Conjecture~\ref{conj:stableasymptotic}.  
Thus $\cT$ satisfies Conjecture~\ref{conj:stableasymptotic} if and only if $\cT'$ does. 

The implications for the other two conjectures follow from this and absorbing other terms involving $m$ into the unspecified constants. 
\end{proof}

Investigating examples that do {\em not} arise in the above manner
necessitates the use of
{\sc Magma}'s
native functionality for computing with function fields, 
rather than the program described in Section~\ref{sec:computing} (which only works for $\ZZ_p$-towers
over the projective line).  We must therefore limit ourselves to examples
with $p=2$ in levels $n\le 5$.

\begin{example} \label{ex:hyperellipticbasep2}
    Working in characteristic $p=2$ over $k=\F_2$, we consider the three $\Z_2$-towers over 
    hyperelliptic ($=$Artin--Schreier over $\PP^1_k$) curves given by
    \begin{align*}
        C_{1} &: y^2 - y = x-\frac{1}{x}-\frac{1}{x-1} \quad & & \cT_1 : Fz - z = [(x^2+x)y] \\
        C_{2} &: y^2 - y = x^3-\frac{1}{x} \quad & & \cT_2: Fz - z = [xy] \\
        C_{3} &: y^2 - y = x^5 \quad & & \cT_3: Fz - z = [y] 
    \end{align*}
    For $i=1,2,3$, the curve $C_i$ is a genus $2$ branched $\Z/2\Z$-cover of $\PP^1_k$,
    and $\cT_i$ is a $\Z_2$-tower, totally ramified over the unique point $Q_i$ on $C_i$ lying over $\infty$ on $\PP^1_k$,
    with (lower) ramification breaks 
    \begin{equation}
        d_{Q_i}(\cT_i(n))=5\cdot \frac{2^{2n - 1}+1}{3}\label{Tibreaks}
    \end{equation}    
        (indeed, in each case the
    function has an order $5$ pole at $Q_i$ and is regular elsewhere.)
    In particular,
    \begin{equation}
        g(\cT_i(n)) = \frac{5\cdot 2^{2n-1} - 1}{3} + 3\cdot 2^{n-1}\label{Tigenus}
    \end{equation}
    for $i=1,2,3$ and all $n\ge 1$.
    Note that $\cT_i$ is {\em not} a $\Z_2$-tower over $\PP^1_k$: 
    one can check (using {\sc Magma} or by hand\footnote{For example, if $\cT_3(1)\rightarrow \PP^1_k$ were Galois,
    there would be an automorphism $\sigma$ of $k(\cT_3(1))$ with $\sigma(y)=y+1$.  Then
    $w:=z+\sigma(z)$ would be an element of $k(\cT_3(1))$ satisfying $w^2+w+1=0$,
    which is impossible as $\F_2$ is algebraically closed in $k(\cT_3(1))$}) that the degree-4 cover $\cT_i(1) \rightarrow \PP^1_k$
    is not Galois for $i=2,3$, and is Galois with group $\Z/2\Z \times \Z/2\Z$ for $i=1$.
    The table below summarizes some basic data about the $C_i$.
    \begin{table}[ht]
\centering
    \begin{tabular}{c | c c c c}
    $C$ & $g(C)$ & $a^1(C)$ & $a^2(C)$ & \\ 
    \toprule
    $C_1$ & 2 & 0 & 0 & \\ 
    $C_2$ & 2 & 1 & 1 & \\ 
    $C_3$ & 2 & 1 & 2 & \\ 
    \end{tabular}
    \caption{Three genus $2$ hyperelliptic curves} \label{table:hyperellipticbasechar2}
\end{table}
    This represents all possible behaviors, as a result of Ekedahl \cite[Theorem 1.1]{Ekedahl}
    shows that if $V=0$ on $H^0(C,\Omega^1_{C/k})$ for a hyperelliptic curve $C$ over a perfect field
    $k$ of characteristic $p$, then $2g(C) \le p-1$ if $(g(C),p)\neq (1,2)$; in particular, there is 
    no genus $2$ hyperelliptic curve in characteristic $p=2$ with $a$-number $2$.


To investigate the behavior of $a^r(\cT_i(n))$ for $i=1,2,3$, we tabulate the differences
\[
   \delta_{5,r}(\cT_i(n),0)-c_r= a^r(\cT_i(n)) - (\alpha(r,2)\cdot 5\cdot 2^{2n} + c_r)
   \quad\text{where}\quad c_r:=\begin{cases} 1/3 & \text{if}\ r=3 \\ 2/3 & \text{otherwise} \end{cases}
\]
and $\alpha(r,2) = \frac{r}{6(r+3)}$ is as in Notation \ref{not:alpha}; the constant term $c_r$
was selected to render most of the table entries integral.
    \begin{table}[ht]
\centering
\resizebox{\linewidth}{!}{%
    \begin{tabular}{c|  c c c c c ||  c c c c c ||  c c c c c  }
     & \multicolumn{5}{c||}{$\delta_{5,r}(\cT_1(n),0)-c_r$} &  \multicolumn{5}{c||}{$\delta_{5,r}(\cT_2(n),0)-c_r$}   & \multicolumn{5}{c}{$\delta_{5,r}(\cT_3(n),0)-c_r$} \\
     \hline
         & $r=1$ & 2 & 3 & 4 & 5    & $r=1$ & 2 & 3 & 4 & 5   & $r=1$ & 2 & 3 & 4 & 5 \\ 
    \toprule
    $n=1$ & -1/2    &0   & 0 & -4/7 & -3/4 &-1/2  &  0 &   1 &10/7 &  5/4 & 3/2  &  4  &  4  & 24/7 & 13/4\\
    2 &         0    &0   & 1 & 5/7  &  1 &   0   & 1   & 1  &5/7   & 1 &   2   & 5   & 6  & 47/7   &  8\\
    3 &         0    &1   & 1 & -1/7  &  0 &   2  &  1  &  2 &20/7   & 4 &   2  &  3  &  5  & 48/7  & 10\\
    4 &         0    &0   & 1 & 3/7   & 0 &   2   & 2   & 2 &17/7    &6 &   2   & 5   & 5  & 45/7   & 10\\
    5 &         0    &1   & 1  &12/7  &  0 &   2  &  1   & 2& 12/7   & 6 &   2 &    4  &  5 & 47/7  & 10
 \end{tabular}
 }
 \caption{Differences $a^r(\cT_i(n)) - (\alpha(r,p)\cdot 5\cdot 2^{2n} + c_r)$}\label{Tideltas}
\end{table}

If $m(r,p)$ is as in Notation \ref{notation:mr}, we have $m(r,2)=1$ if $r=1,3,5$ while $m(2,2)=2$ and $m(4,2)=3$.
The facts that the $r=1,3,5$ columns in Table \ref{Tideltas} appear to stabilize to constant functions,
and the $r=2$ columns appear to be stabilizing to periodic functions with period $2$ 
(with possibly nonzero linear term $\lambda\cdot n$ when $i=2,3$),
support Conjectures \ref{conj:stableasymptotic}, \ref{conj:stableanumber}, and \ref{conj:stablespecific}, and suggest that a more precise variant of Conjecture 3.8 along the lines
of Conjecture \ref{conj:basicpowers} should be true in this context as well.
The $r=4$ columns appear to be less structured, but they are nonetheless consistent with 
Conjectures \ref{conj:stableasymptotic} and \ref{conj:stablespecific},
and a possible analogue of Conjecture \ref{conj:basicpowers}, which would
predict a period of $m(4,2)=3$ that large relative to the modest number ($n=5$) of levels we have been able to compute.
\end{example}

\begin{remark}
    We note that the $a$-number formula \eqref{eq:anumberformula} appears to hold for $\cT_1(n)$
    in all levels $n$, despite the fact that the technical hypothesis 
    \eqref{eq:hypothesis} does {\em not} hold: indeed, 
    writing $d_n:=d_{Q_i}(\cT_i(n))$ and $g_n:=g(\cT_i(n))$ (noting that these numbers are independent of $i$ by \eqref{Tibreaks}--\eqref{Tigenus}), we compute
    \[
        \frac{d_{n+1} - 1}{2} - (2g_n - 2) =
        \frac{5\cdot 2^{2n} + 1}{3} - \frac{5\cdot 2^{2n}-2}{3} -3\cdot 2^n + 2 = 
        3(1-2^n)
    \]
    which is negative for $n\ge 1$.  Note that as $\cT_1(0)$ is {\em ordinary},
    one knows {\em a-priori} that the $a$-number of $\cT_1(1)$ is given by \eqref{eq:anumberformula}
    due to Remark \ref{Voloch}. 
    For $i=2,3$, it appears that 
    \[
        a(\cT_i(n)) = a(\cT_1(n)) + 2
    \]
    for all $n$ (respectively all $n\ge 3$) when $i=3$ (respectively $i=2$);
    in particular, the formula \eqref{eq:anumberformula} appears to be off by 2 for these two towers when $n\ge 3$.
\end{remark}



\begin{example} \label{ex:hyperellipticbase}
Let $C$ be the hyperelliptic curve over $\F_3$ given by the equation $y^2 = x^5 +x^2 +  1 $.    The function $f = x y$ has a pole of order $7$ at the unique point $Q$ at infinity on $C$.  We obtain a $\Z_3$-tower $\cT$ with $\cT(0) =C$  from the Artin-Schreier-Witt equation $Fz - z = [f]$.  
({\sc Magma} computes that $ \Aut(\cT(1)) = S_3$, so $\cT(1)$ cannot be a Galois $\Z/3^m\Z$-cover of $\PP^1$ for any $m \geq 2$ and $C$ is the only curve for which $\cT(1)$ is a $\Z/3 \Z$-cover.  This shows the tower $\cT$ is not related to a $\Z_p$-tower over the projective line by adding an extra layer or modifying the base).
We can compute $d_Q(\cT(n))$ using Fact~\ref{fact:localasw} and check the tower is monodromy stable.  Table~\ref{table:hyperellipticbase} shows some invariants of $\cT(n)$ for small $n$.

\begin{table}[ht]
\centering
\begin{tabular}{c|c c c c c c c }
Level &  $\dk{1}{\cT(n)}$ &$\dk{2}{\cT(n)}$ & $\dk{3}{\cT(n)}$& $\dk{4}{\cT(n)}$ &$\dk{5}{\cT(n)}$ \\ 
\toprule
1   & 3 &5 & 7 & 8 & 9  \\ 
2   &  24 & 38 & 47 & 51 & 53  \\ 
3   & 213 & 321 & 386 & 430 & 464  \\ 
4   & 1914 & 2874 & 3450 & 3832 & 4107 \\ 
\bottomrule
\end{tabular}
\caption{The tower $Fz - z = [xy]$ over $C: y^2 = x^5 +x^2 +1$} \label{table:hyperellipticbase}
\end{table}
This supports Conjecture~\ref{conj:stableanumber} and Conjecture~\ref{conj:stablespecific}, and exhibits very similar behavior to basic towers with ramification invariant $7$ like the one in Example~\ref{example:p3d7}.  In fact, the $a$-numbers look exactly the same while $\dk{r}{\cT(n)}$ for $r>1$ exhibits slight variation.  For example, note that $7 \alpha(4,3) = 7/12$ and that for $n=3$ and $n=4$ we have
\[
\dk{4}{\cT(n)} = \frac{7}{12} \left( 3^{2n} - 9 \right) + 10.
\]
\end{example}

\subsection{Towers with Periodically Stable Monodromy} \label{ss:periodic}
 We next compute examples with towers $\cT$ which are not monodromy stable.  While the behavior of $\dk{r}{\cT(n)}$ in these examples does not exactly match that of monodromy stable towers as predicted by Conjectures~\ref{conj:stableanumber} and \ref{conj:stablespecific}, it nevertheless appears to be quite structured, in line with Philosophy~\ref{philosophy}.
 We consider some selected towers with periodically stable monodromy, though we do not attempt to explore this situation exhaustively.

\begin{example} \label{ex:periodicmonodromy}
Let $p=2$ and $d$ be odd. 
We consider the tower
\begin{equation}
\cT_d: Fy - y = [x^d] + \sum_{i=1}^\infty p^{2i-1} [x^{(d+2) 2^i -2}] = [x^d] + p [x^{2d+3 }] + p^3 [x^{8d+15}]+\cdots
\end{equation}
This is ramified only over infinity.  
Using Fact~\ref{fact:localasw} and Lemma~\ref{lem:breaks}, we see that $u(\cT_d(n)) = s_n +1$ and $s(\cT_d(n)) = s_n$, where $s_{n+1} = (d+2) 2^n -2$ if $n$ is even, and $s_{n+1}=(d+2) 2^n -1$ if $n$ is odd.
(Note that for $n$ odd, the value of $u(\cT_d(n))$ comes from the term $s_{n-1} p^{n-1-(n-2)} = p s_{n-1}$ appearing in the maximum in Fact~\ref{fact:localasw}.)

A straightforward computation with the Riemann--Hurwitz genus formula yields
\[
g(\cT_d(n)) 
=\frac{d+2}{2}\left(\frac{2^{2n}-1}{3}\right)- \left(\frac{5}{2}+\frac{(-1)^{n-1}}{6}\right)2^{n-1} +\frac{7}{6}
= \frac{d+1}{2}\left(\frac{2^{2n}-1}{3}\right)+ 2^{n-1}\left\lfloor\frac{2^n-7}{3}\right\rfloor +1,
\]
the second expression being visibly an integer.
Using Corollary~\ref{cor:anumberformula}, we will be able to prove that
\begin{equation}
      a(\cT_d(n)) = \frac{d+2}{4}\left(\frac{2^{2n-1}+1}{3}\right)+\frac{1}{4}\frac{(-2)^n-1}{3}+\frac{(-1)^{\frac{d+1}{2}}-1}{4},
      \label{eq:anummonodromyper}
\end{equation}
valid for all $n$ and $d$.  Note that to fit into the framework of our previous conjectures where the $a$-number is a quadratic polynomial in $2^n$, we would need the polynomial to depend on whether $n$ is even or odd because of the presence of the $(-2)^n$ term.  For monodromy stable towers, Conjecture~\ref{conj:stableanumber} predicts the $a$-number is given by a single formula and does not exhibit periodic behavior.  But as the monodromy in this tower has period two, it is not surprising that the formula \eqref{eq:anummonodromyper} for the $a$-number  depends on the parity of $n$; indeed, this is already the case for the genus.

Based on these formulae, and following the lead of Conjecture \ref{conj:stableasymptotic}, we are led to guess the asymptotic
formula
\[
    \lim_{n\rightarrow \infty} \frac{\dk{r}{\cT_d(n)}}{\alpha(r,2)(d+2)2^{2n}}
    = \lim_{n\rightarrow \infty} \frac{\dk{r}{\cT_d(n)}}{\frac{r}{(r+3)}\left(\frac{d+2}{6}\right)2^{2n}} = 1
\]
with $\alpha(r,p)$ from Notation \ref{not:alpha}.
This visibly holds for $r=1$.  For $2\le r \le 8$, odd $d$ with $3\le d \le 35$ and all $2 < n\le 7$,
we computed that
\begin{equation}
    \left|\frac{\dk{r}{\cT_d(n)}}{\frac{r}{(r+3)}\left(\frac{d+2}{6}\right)2^{2n}} - 1\right| < 2^{-n}.
    \label{eq:asympdeltapermonodromy}
\end{equation}
This is consistent with the existence of a secondary term of the form $c(n)\cdot 2^n$ for some periodic function $c:\ZZ/m\ZZ\rightarrow \QQ$.

In light of the above, it is tempting to believe that Conjecture \ref{conj:stablespecific}
may hold for arbitrary monodromy periodic towers.  Direct evidence for this
is somewhat elusive due to the presence of the secondary term of order $2^n$ and the fact that our computations are limited to small values of $n$ and $p$.  Furthermore, we have only computed with the single tower $\cT_d$ for each value of $d$
rather than looking at multiple towers with the same limiting ramification breaks.
The behavior is clearest for $r=5$, where our computations support the exact formula
\[
   \dk{5}{\cT_d(n)} =  \frac{d+2}{4}\left(\frac{5\cdot2^{2n-2}+1}{3}\right)+\frac{1}{4}\frac{(-2)^{n-1}-1}{3}+\frac{(-1)^{\frac{d+1}{2}}-1}{4},
\]
valid for all odd $d$ with $3\le d \le 35$ and all $3\le n\le 7$, with the sole exception of $(d,n)=(3,3)$.
In fact, further computation shows that this formula holds as well for all odd $d$ in the range $3\le d \le 161$
and all $3\le n\le 6$, again with the sole exception $(d,n)=(3,3)$.

Unfortunately, we were unable to find similar exact formulae for other values of $r$ that are uniform in
$d$ and $n$ for $n$ sufficiently large.
Nonetheless, our limited computations 
support that something like Conjecture \ref{conj:stablespecific} should hold
for any monodromy periodic tower, though the periodic ``coefficient'' functions occurring
therein would appear to have a rather complicated and subtle dependence on $d_Q$ and the function $c_Q(n)$
giving the (periodic) upper ramification breaks as in Definition \ref{defn:types}.




\end{example}

\subsection{Towers with Faster Genus Growth} \label{ss:fastermonodromy}

Next we collect data and make observations on a few towers with much faster genus growth.
In order to be in a situation where we can have some hope of identifying patterns, we limit
ourselves to cases where the genus growth is sufficiently regular.  The rapid growth of
the genus in these examples limits our computations to characteristic $p=2$ and levels $n\le 6$.

\begin{example}
    Let $d$ be a positive odd integer.
    In characteristic $p=2$, we consider the towers
    \begin{equation*}
        \cT_d : Fy - y = \sum_{n\ge 0} 2^n [x^{s_n}]
    \end{equation*}
    with $s_n:=(d+2)\cdot2^{2n+1}-1$.  Note that $s_\infty(\cT_d(n)) = s_{n-1}$ for $n\ge 1$.   
    Using the Riemann--Hurwitz formula in the form of Lemma~\ref{lem:genuslower}, we have
    \begin{equation*}
        g(\cT_d(n)) = \frac{(d+2)}{7}(2^{3n}-1)-2^n+1.
    \end{equation*}
    We computed $\dk{r}{\cT_d(n)}$ for $r\le 10$ and $n\le 6$, with $d\in \{1,3,5,7,9\}$.
    Inspired by Conjecture~\ref{conj:stableasymptotic} and Corollary~\ref{cor:ratio}, we
    first analyzed the ratio 
    \begin{equation}
        \frac{\dk{r}{\cT_d(n)}}{\frac{(d+2)}{7}\cdot 2^{3n}}\label{eq:fastmonanum}
    \end{equation}
    for each value of $r$ and $d$, and all $n\le 6$.
    For each fixed $r$, this ratio appears to stabilize as $n$ increases, to a value that appears not to depend on $d$.
    Working in level $n=6$ and truncating the resulting limiting values to 6 decimal places, we 
    then used continued fraction expansions to find best possible rational approximations with comparatively small denominator,
    which leads to the following guesses for the main term of $\dk{r}{\cT_d(n)}$.
    
    \begin{definition}
        For $r\le 10$,  define $$\nu_{d,r}(n):=\left\lfloor\frac{d+2}{7} \alpha(r) 2^{3n}\right\rfloor,$$
        where $\alpha(r)$ is given in Table~\ref{table:fastalpha}.
        \begin{table}[ht]
    \centering
    \begin{tabular}{c|c c c c c c c c c c }
        $r$ &  1 & 2 & 3 & 4 & 5 & 6 & 7 & 8 & 9 & 10\\
        \toprule
        $\alpha(r)$ & $\frac{3}{8}$  & $\frac{37}{64}$    &  $\frac{91}{128}$  & $\frac{63}{80}$  &  $\frac{75}{89}$  & $\frac{78}{89}$  & $\frac{64}{71}$  & $\frac{169}{184}$ & $\frac{41}{44}$ & $\frac{81}{86}$    \\
\bottomrule
\end{tabular}
\caption{Leading constants for $\cT_d$} \label{table:fastalpha}
\end{table}    
    \end{definition}
    
    Thanks to Corollary \ref{cor:anumberformula} below, we have the {\em exact formula}
        \[
            a(\cT_d(n)) = \frac{3}{8} \cdot \frac{d+2}{7}\cdot 2^{3n} + \frac{d-5}{14} = \frac{d+2}{2}\left(\frac{3\cdot 2^{3n -2}+1}{7}\right)-\frac{1}{2}
        \]
        for all $n$ and $d$, which provides a proof that the ratio \eqref{eq:fastmonanum} tends to $3/8$ as $n\rightarrow \infty$.
        In a similar spirit, we compute that
        \begin{equation}
            \dk{2}{\cT_d(n+1)} - 8 \dk{2}{\cT_d(n)} = 3 \frac{5-d}{2}+7\left\lfloor\frac{d-1}{8}\right\rfloor\label{eq:dk2}
        \end{equation}
        for all $d \in \{1,3,5,7,9\}$ and $2 \leq n \leq 5$, and
        \begin{equation}
            \dk{3}{\cT_d(n+1)} - 8 \dk{3}{\cT_d(n)} =\begin{cases} 7 & 3 \leq n \leq 5 \\ 3 & n=2\end{cases} \label{eq:dk3}
        \end{equation}
        for all $d \in \{1,3,5,7,9\}$.  These lead to the conjectural exact formul\ae
        \[
            \dk{2}{\cT_d(n)} = \frac{37}{64}\cdot \frac{d+2}{7} \cdot 2^{3n} +3\frac{d-5}{14} -\left\lfloor\frac{d-1}{8}\right\rfloor= \frac{d+2}{8}\left(\frac{37\cdot 2^{3(n -1)}+5}{7}\right)-1 +\frac{d\bmod 8 - 2}{8}
        \]
        for $n\ge 2$ (where $d\bmod 8$ denotes the least nonnegative residue of $d$ modulo $8$) and
        \[
            \dk{3}{\cT_d(n)} = \frac{91}{128}\cdot \frac{d+2}{7} \cdot 2^{3n} -1 = 13(d+2)2^{3n-7} - 1
        \]
        for $n\ge 3$.  These conjectural formul{\ae} support the given values of $\alpha(r)$ for $r=2,3$;
        note that these values of $\alpha$ arise naturally from \eqref{eq:dk2}--\eqref{eq:dk3} and the computed values of $a^r(\cT_d(n))$ for $n=2$.
         Thus, it is perhaps unsurprising that the values of $\alpha$
        do not appear to satisfy a simple formula.
        
        \begin{remark}
            The special formulae
            \begin{equation*}
                 \dk{2}{\cT_d(1)} = \frac{3d+(-1)^{(d-1)/2}}{4}+1 \quad \text{and} \quad \dk{3}{\cT_d(2)} = 6(d+2)+\frac{d+1}{2} 
            \end{equation*}
            hold for all (positive, odd) values of $d$ less than $30$ in levels 1 and 2, respectively.
            However, we could not discern any pattern for the values of $\dk{3}{\cT_d(1)}$ in level 1.
        \end{remark}

        For  $r > 3$, the differences $\dk{r}{\cT_d(n+1)} - 8 \dk{r}{\cT_d(n)}$
        are comparatively small, but do not seem to obey any obvious pattern.
        Nonetheless, for all values of $r$, $d$, and $n$ for which we computed $\dk{r}{\cT_d(n)}$,
        this integer is remarkably close to $\nu_{d,r}(n)$, and we tabulate the differences 
        $\dk{r}{\cT_d(n)} - \nu_{d,r}(n)$ for $4 \le r \le 10$, $d\in \{1,3,5,7,9,11\}$, and $1\le n\le 6$ in Table~\ref{table:deviations}.
        
    \begin{table}[!htb]
    \caption{$\dk{r}{\cT_d(n)} - \nu_{d,r}(n)$} \label{table:deviations}
    \begin{subtable}{.5\linewidth}
      \caption{d=1}
      \centering
     \centering
    \begin{tabular}{c|c c c c c c c c}
        \diag{.1em}{1ex}{$n$}{$r$} &  $4$ &$5$ & $6$& $7$ &$8$ &$9$ & $10$ \\ 
        \toprule
        1 &  0 &  0 & -1 &-1 & -1 &  -1 & -1\\
        2 & 0  & 0  & 0  & 0 & -1 & -1  &-1\\
        3 & 0  &-1  &-2  &-2 & -3 & -3  &-2\\
        4 & 0  &-3  &-3  &-3 & -1 & -3  &-6\\
        5 & -1  &-3  &-4  &-3 & -4 & -7  &-4\\
        6 &  0  &-4 & -3  &-3 & -7 &-12  &-14\\
    \bottomrule
    \end{tabular}
    \end{subtable}%
    \begin{subtable}{.5\linewidth}
      \centering
        \caption{d=3}
     \begin{tabular}{c|c c c c c c c c}
        \diag{.1em}{1ex}{$n$}{$r$} &  $4$ &$5$ & $6$& $7$ &$8$ &$9$ & $10$ \\ 
        \toprule
        1&  0 &  0 & -1 & -1 & -1 & -1 & -1\\
        2&  0 &  2 &  2 &  1 &  1 &  0 & -1\\
        3& -1 &  0 &  1 & -1 & -2 & -2 & -1\\
        4& -1 &  0 &  0 &  1 &  4 & -1 & -5\\
        5& -1 & -2 &  1 & -1 &  1 & -4 &  3\\
        6&  0 & -3 &  2 &  2 & -2 & -7 &-12\\
    \bottomrule
    \end{tabular}
    \end{subtable}%
    \vspace{1cm}
    \begin{subtable}{.5\linewidth}
      \centering
        \caption{$d=5$}
        \begin{tabular}{c|c c c c c c c c}
        \diag{.1em}{1ex}{$n$}{$r$} &  $4$ &$5$ & $6$& $7$ &$8$ &$9$ & $10$ \\ 
        \toprule
      1 & 0 & 0 & -1 & -1 & -1 & -1 & -1  \\ 
    2 & 0 & 1 & 1 & 2 & 2 & 1 & 0  \\ 
    3 & 0 & -1 & 1 & -1 & -2 & -2 & 0  \\ 
    4 & 1 & -2 & 0 & 1 & 4 & 2 & -3  \\ 
    5 & 1 & -3 & 0 & -1 & -1 & -2 & 6  \\ 
    6 & 2 & -3 & 4 & 4 & -6 & -7 & -16  \\ 
    \bottomrule
    \end{tabular}
    \end{subtable}%
 \begin{subtable}{.5\linewidth}
      \centering
        \caption{$d=7$}
        \begin{tabular}{c|c c c c c c c c}
        \diag{.1em}{1ex}{$n$}{$r$} &  $4$ &$5$ & $6$& $7$ &$8$ &$9$ & $10$ \\ 
        \toprule
     1 & 0 & 0 & -1 & -1 & -1 & -1 & -1  \\ 
2 & 1 & 2 & 4 & 4 & 3 & 2 & 1  \\ 
3 & 1 & 2 & 4 & 0 & -1 & -1 & 1  \\ 
4 & 2 & 0 & 1 & 2 & 11 & 5 & -2  \\ 
5 & 3 & -2 & 3 & 3 & 8 & 2 & 16  \\ 
6 & 5 & -1 & 9 & 11 & 7 & -3 & -5  \\ 
    \bottomrule
    \end{tabular}
    \end{subtable}%
   \vspace{1cm}
    \begin{subtable}{.5\linewidth}
      \centering
        \caption{$d=9$}
        \begin{tabular}{c|c c c c c c c c}
        \diag{.1em}{1ex}{$n$}{$r$} &  $4$ &$5$ & $6$& $7$ &$8$ &$9$ & $10$ \\ 
        \toprule
1 & 1 & 0 & -1 & -1 & -1 & -1 & -1  \\ 
2 & 1 & 2 & 1 & 2 & 2 & 2 & 2  \\ 
3 & 1 & 0 & 1 & 2 & 0 & 0 & 2  \\ 
4 & 3 & 1 & 2 & 4 & 5 & 3 & 0  \\ 
5 & 2 & -1 & 4 & 6 & 5 & 2 & 2  \\ 
6 & 4 & 1 & 11 & 12 & 6 & -2 & -18  \\ 
    \bottomrule
    \end{tabular}
    \end{subtable}%
 \begin{subtable}{.5\linewidth}
      \centering
        \caption{$d=11$}
        \begin{tabular}{c|c c c c c c c c}
        \diag{.1em}{1ex}{$n$}{$r$} &  $4$ &$5$ & $6$& $7$ &$8$ &$9$ & $10$ \\ 
        \toprule
1 & 1 & 0 & -1 & -1 & -1 & -1 & -1  \\ 
2 & 1 & 2 & 4 & 4 & 3 & 3 & 3  \\ 
3 & 2 & 2 & 3 & 2 & 0 & 0 & 2  \\ 
4 & 4 & 2 & 4 & 6 & 12 & 8 & 2  \\ 
5 & 4 & 0 & 7 & 5 & 8 & 7 & 17  \\ 
6 & 7 & 2 & 18 & 17 & 10 & 4 & -13  \\ 
    \bottomrule
    \end{tabular}
    \end{subtable}
\end{table}





\end{example}

\subsection{Non-\texorpdfstring{$\Z_p$}{Zp}-towers} 
In contrast, we now give an example of a sequence of Artin-Schreier covers in characteristic three which are not part of a $\Z_3$-tower but have the same ramification as we would see in a basic $\Z_p$-tower. 

\begin{example} \label{ex:nottower}
Take $p=3$ and $k= \F_p$.  Let $C_0 = \PP^1_{k}$ with function field $k(x)$, and for $1 \leq i \leq 4$ construct the curve $C_i$ 
by adjoining a root of $y_i^3 - y_i = f_i$ to $k(C_{i-1})$, where $f_1 = x^7$, $f_2 = x^{14} y_1$, $f_3 = x^{42}y_2$, and $f_4 = x^{126} y_3$.  
It is straightforward to verify that $C_i \to C_{i-1}$ is totally ramified over the point above infinity, with ramification invariants $7, 49, 427, 3829$.  This is the same sequence of layer-by-layer ramification breaks and genera as for basic $\Z_3$-towers with ramification invariant $7$, which we looked at in Example~\ref{example:p3d7}.  However, these curves do not form the first 4 layers of a $\Z_3$-tower---one can check (using {\sc Magma}, or directly) that $k(C_2)$ is not normal over $k(x)$.
Similarly define $C'_i$ for $1\leq i \leq 4$ using $f'_1 = x^7$, $f'_2 = x^{14}y_1 + x^2$, $f'_3 = (x^{42} + x^{20}) y_2 + x y_1$, and $f'_4 = x^{126} y_3 + x^5 y_2 + x^2$, where we have added some lower-order terms to the Artin-Schreier equations for $C_i$.
Table~\ref{table:nottower} records invariants of these sequences of curves.  
\begin{table}[ht]
\centering
\begin{tabular}{c|c c c c  || c | c c c c}
& 1 & 2 & 3 & 4 & & 1 & 2 & 3 & 4 \\ 
\toprule
$a(C_n)$ & 4  & 27  & 231  & 2057  & $a(C'_n)$ & 4 &27 &231 & 2025\\
$\dk{2}{C_n}$ & 5 & 39 & 364 & 3329 & $\dk{2}{C_n}$ & 5 & 39 & 353 & 3079 \\
$\dk{3}{C_n}$ & 6 & 49 & 442 & 4113 & $\dk{3}{C_n}$ & 6 & 49 & 429 & 3806\\
$\dk{4}{C_n}$ & 6 & 54 & 490 & 4550 & $\dk{4}{C_n}$ & 6 & 54& 479 & 4305 \\
\bottomrule
\end{tabular}
\caption{Invariants of $C_n$ and $C'_n$ for $1 \leq n \leq 4$}
\label{table:nottower}
\end{table}

They are much less structured that what we have seen for $\Z_p$-towers.  For example, based on the first three curves it would be reasonable to guess that
\[
a(C_n) = a(C'_n) = \frac{5}{16} 3^{2n} + \frac{1}{12} 3^n + \frac{15}{16}
\]
as this holds for $n =1,2,3$ and the proposed formula has relatively small denominators.  This would predict that the $a$-number of $C_4$ and $C'_4$ would be $2058$.  While $a(C_4)$ is close, $a(C'_4)$ is considerably different.  Higher powers of the Cartier operator display similar irregularity.  Furthermore, $\frac{a(C_4)}{g(C_4)}$ and $\frac{a(C'_4)}{g(C_4)}$ are around $.361$ and $.355$ respectively, while this ratio is much closer to $\frac{1}{3}$ for basic $\Z_3$-towers having similar ramification.  
This suggests that it is crucial that a sequence of Artin-Schreier covers actually form a $\Z_p$-tower in order to 
obtain exact (or even just asymptotic) formulae for $\dk{r}{C_n}$ when $n\gg 0$.

It is also curious that this ratio is not particularly close to $\frac{1}{2}(1-p^{-1})(1-p^{-2}) = 8/27  \approx 0.296$, the na\"{i}ve guess articulated below equation \eqref{naiveguess}.  Possibly that guess is wrong.  On the other hand, these are examples in which $n=4$ and it would be reasonable to expect noise on the order of $3^{-4} \approx 0.0123$.  Furthermore, we chose the $\Z/3\Z$-covers so they could be described by Artin-Schreier equations with relatively few terms, which is not the generic behavior of random covers.  
\end{example}


\section{Computing with the Cartier Operator in Towers} \label{sec:computing}

Let $k$ be a finite field of characteristic $p$ and $\cT$ a $\Z_p$-tower over $k$.  Suppose
the base of the tower is the projective line (i.e.  $\cT(0) = \PP^1_k$) and that $\cT$ is totally ramified over infinity and unramified elsewhere.  
In this section, we describe how to compute efficiently with the Cartier operator on the space of regular differentials on $\cT(n)$.  The key difficulty is that $g(\cT(n))$ grows at least exponentially in $n$ (see Remark~\ref{remark:lowergrowth}), so the dimension of the space of regular differentials on $\cT(n)$ quickly becomes intractably large as $n$ increases.  In order to compute with enough levels of $\cT$ to have any hope of systematically investigating Philosophy~\ref{philosophy}, we must be as efficient as possible.

The {\sc Magma} computer algebra system has extensive, robust algorithms for function fields; in particular, it has the ability  
to compute with Witt vectors and Artin--Schreier--Witt extensions in characteristic $p$,
and to compute a matrix representation of the Cartier operator on the space of regular differentials on the 
smooth projective curve associated to any function field over $k$.
Unfortunately, these algorithms are not efficient enough to compute beyond the first few levels in a $\Z_p$-tower,
which severely limits the computational support they can provide for our conjectures.
For example, Table~\ref{table:runningtimes} shows the time needed to compute the $a$-number of the first few levels of the basic $\Z_3$-tower $\cT_{\timing} : Fy - y = [x^7] + [x^5]$ in characteristic $3$ using {\sc Magma}'s default methods and using our more efficient methods.\footnote{These were run on a virtual server at the University of Canterbury equivalent to an Intel Core Processor (Skylake) CPU at 2600 MHz with 67,036 MB of RAM.}  (These computations also use substantial amounts of memory.  For example, our method used several gigabytes of memory for the fifth level.)








 \begin{table}[hb]
    \centering
    \begin{tabular}{c|c c}
        Level & {\sc Magma} & Our Method \\ 
        \toprule
         2& $.08$ seconds & $.01$ seconds \\
         3& $12$ seconds & $.16$ seconds \\
         4& $48$ hours & $25$ seconds \\
         5&  & $7$ hours\\
\bottomrule
\end{tabular}
\caption{Approximate Running Times to Compute $a(\cT(n))$ for the $\Z_3$-tower $\cT_{\timing} : Fy -y = [x^7] + [x^5]$} \label{table:runningtimes}
\end{table}


We have implemented our algorithm in {\sc Magma} \cite{bcgit} in order to build on {\sc Magma}'s support of efficient computations with polynomials and efficient linear algebra over finite fields.
The key improvement
is incorporating theoretic information about the space of regular differentials in a $\Z_p$-tower.  An overview of our algorithm is as follows.

\begin{enumerate}
    \item  Perform computations with Witt vectors to turn a description of a $\Z_p$-tower using Artin-Schreier-Witt theory into a description of the tower as a sequence of $\Z/p\Z$-covers.  This uses techniques and functions developed by Lu\'{i}s Finotti for performing computations with Witt vectors \cite{finotti14,finottigithub} which are substantially faster than the native Witt vector algorithms of {\sc Magma}. 
    
    \item  Rewrite the sequence of Artin-Schreier $\Z/p\Z$-covers describing the tower in a standard form (discussed in Section~\ref{ss:standardform}).
    
    \item  Using results of \cite{madden} for $\Z/p^n\Z$-covers of $\PP^1_k$ in the above standard form, 
    write down an explicit basis of regular differentials on $\cT(n)$; see Section~\ref{ss:maden}.
    
    \item  Recursively compute the images under the Cartier operator of a subset of this basis which suffices to compute the image of any differential using semilinearity. 
    
    \item  Build a matrix representing the Cartier operator and compute the kernel of its $r$th power to obtain $\dk{r}{\cT(n)}$.
    
\end{enumerate}

In the remainder of this section we explain these steps in detail.

\begin{remark}
We will not present a careful asymptotic analysis of the running time of the algorithm because it is still exponential in $n$ like {\sc Magma}'s default functionality.  Roughly speaking, the algorithm is polynomial time in the genus of $\cT(n)$; this is very bad as $g(\cT(n))$ grows exponentially in $n$.  Therefore this can only be practical for small $n$.  (For a basic $\Z_p$-tower with ramification invariant $d$, Lemma~\ref{lem:basicinvariants} says that $g(\cT(n)) = \Theta(d p^{2n})$). 
The advantage of our algorithm over the default methods available in {\sc Magma} is that it is much faster in practice, allowing us to compute further with basic towers and provide much stronger evidence for our conjectures.
\end{remark}

\begin{remark}
For $p>2$, the bottleneck in our algorithm is usually the computations with the Cartier operator in Step 4.  Writing down the tower in standard form and carrying out the linear algebra are substantially faster than carrying out the computations with the Cartier operator for the fifth level.

When $p=2$, other parts of the computation are the bottleneck.  For example, we are able to compute with the seventh level of many $\Z_2$-towers (and sometimes the eighth level for very simple towers like the one in Example~\ref{ex:p2longperiod}.  In these examples, writing a basic $\Z_2$-tower in standard form (Step 2) is often the bottleneck, or occasionally the computations with Witt vectors (Step 1).  For towers with faster growth like in Section~\ref{ss:fastermonodromy} in characteristic two, the linear algebra computations (Step 5) are actually the bottleneck.
\end{remark}

\begin{remark}
When {\sc Magma} does a similar computation, most of the time is spent finding a basis of the regular differentials on $\cT(n)$.  After {\sc Magma} has pre-computed a basis for the regular differentials on $\cT(n)$ (for example, in the course of computing the genus), {\sc Magma} can 
produce a matrix representing the Cartier operator and compute $\dk{r}{\cT(n)}$ quickly relative to the time spent finding the basis.
\end{remark}

\begin{remark}
Our program assumes that the base of the $\Z_p$-tower is the projective line and that the tower is totally ramified over infinity and unramified elsewhere.  The first assumption is essential to present the tower in standard form and use the results of  \cite{madden} to write down a basis of regular differentials.  If the tower has a different base, it may not be possible to write it in standard form (see Section~\ref{ss:standardform}).  

The second assumption, that there is only one point of ramification, is not essential but makes many of the computations simpler.  In particular, it allows representing differentials as polynomials instead of rational functions.  As the bulk of computation time is spent performing computations with these polynomials, this is an important simplification.
\end{remark}

\subsection{Computations with Witt Vectors}  The polynomials defining addition (and multiplication) for the length-$n$ Witt vectors  become increasingly complicated as $n$ increases.  Such computations are necessary to convert
the Artin-Schreier-Witt description of a $\Z_p$-tower as in \S\ref{sec:towers}
into explicit Artin-Schreier equations for each layer of the tower as a $\Z/p\Z$-cover of the previous layer.
{\sc Magma} has the functionality to compute with Witt vectors, but we use the methods of Lu\'{i}s Finotti \cite{finotti14,finottigithub} which are more efficient.  The calculations with Witt vectors are rarely the limiting factor in our computations---for example, these computations (using either method) for the fourth level of the $\Z_3$-tower $\cT_{\timing}$
appearing in Table~\ref{table:runningtimes}
take well under a second while the overall computation is substantially longer.  It is only for a $\Z_2$-tower like $\cT' : Fy - y = [x^3]$ in characteristic two where the ramification invariant is very small and the Artin-Schreier-Witt equation is very simple that the computations with Witt vectors take substantial time compared to the other parts of the computation.  In particular, the Witt vector computations needed to analyze the eighth level of $\cT'$ using Finotti's algorithms took about six hours (and then around 40 more minutes to compute the a-number) while {\sc Magma}'s native functionality did not finish the Witt vector computations within $24$ hours.

\subsection{Standard Form for Artin-Schreier-Witt Towers} \label{ss:standardform}
Let $\cT$ be a $\Z_p$-tower over a finite field $k$ of characteristic $p$ ramified over $S \subset \cT(0)$.  For $Q \in S$, let $Q(n)$ be the unique point of $\cT(n)$ over $Q$.  The function field of $\cT(n)$ is an Artin-Schreier extension of $\cT(n-1)$, given by adjoining a root of the Artin-Schreier equation 
\[
y_n^p - y_n = f_n
\]
for some $f_n \in k(\cT(n-1))$.  This representation is of course not unique: making the change of variable $y_n' = y_n + z$ replaces $f_n$ by $f_n' = f_n + z^p - z$ but gives an isomorphic extension of fields.

\begin{defn} \label{defn:standardform}
We say the functions $(f_1,f_2,\ldots,)$ present the tower $\cT$ in {\em standard form} provided that for all positive integers $n$, we have $d_Q(\cT(n)) = \ord_{Q(n-1)}(f_n)$ for all $Q \in S$ and $f_n$ is regular away from the points of $\cT(n-1)$ over $S$.  We can analogously talk about a single $\Z/p\Z$-cover being presented in standard form.
\end{defn}

\begin{remark}
 Over the projective line, the theory of partial fractions allows one to write every $\Z/p\Z$-cover of the projective line in standard form.
 \end{remark}
 
If $d_Q(\cT(n)) \neq \ord_{Q(n-1)}(f_n)$, it is always possible to locally modify the presentation of $\cT(n) \to \cT(n-1)$ so that $d_Q(\cT(n)) = \ord_{Q_n}(f'_n)$ by making a change of variable $y_n' = y_n + z$ where $z$ has the appropriate local behavior at $Q(n-1)$.  But it is not always possible to do so at each ramified point while keeping $f'_n$ regular away from the points of $\cT(n-1)$ above $S$; see \cite[\S7]{shabat}, especially the second example after Proposition 49.
Therefore the following result of Madden about towers over the projective line is initially surprising.  

\begin{fact}[{{\hspace{1sp}\cite[Theorem 2]{madden}}}] \label{fact:standardform}
Every $\Z_p$-tower over the projective line can be presented in standard form.
\end{fact}

While it is essential for our computations that we work with a tower in standard form, the description of the tower as a sequence of Artin-Schreier extensions produced by Artin-Schreier-Witt theory need not be in standard form. Thus a key step in our computations is to explicitly rewrite the given Artin-Schreier-Witt description of a tower in standard form.

Let $\cT$ be a $\Z_p$-tower with base the projective line (with function field $k(x)$) that is totally ramified over infinity and unramified elsewhere.  Let $P(n)$ be the unique point of $\cT(n)$ above infinity.  
Suppose the first $n-1$ levels of the tower are written in standard form using Artin-Schreier equations $y_i^p - y_i = f_i$.  We will describe how to rewrite the $n$th level in standard form.

Let the $n$th level be given by an Artin-Schreier equation 
\[
y_n^p - y_n = f
\]
where $f$ is a polynomial in $x$ and $y_1,\ldots, y_{n-1}$.  
Note that $\ord_{P(i)}(y_i) = \ord_{P(i-1)}(f_i) = d_\infty(\cT(i))$ for $i < n$.  
The $n$th level is not in standard form precisely if $\ord_{P(n-1)}(f)$ is a multiple of $p$, which necessarily must be larger than $d(\cT(n))$ \cite[Proposition 3.7.8]{stichtenoth}.  We may effectively write it in standard form if we can always produce a function $z$ on $\cT(n-1)$ with a pole of order $\ord_{P(n-1)}(f)/p $ at $P_{n-1}$ that is regular elsewhere:
 changing variables by an appropriate multiple of $z$, which replaces $f$ by $f + (cz)^p - cz$, will cancel out the leading term, after which we repeat this process.
 The proof of \cite[Theorem 2]{madden} shows there exists  non-negative integers $\nu$ and $a_1,\ldots, a_{n-1}$ such that $z= x^\nu y_1^{a_1} \ldots, y_{n-1}^{a_{n-1}}$ has the desired valuation (see in particular \cite[Lemma 3]{madden} and the subsequent decomposition of $L(\mathfrak{a}^{-1})$).
 
\begin{remark}
While conceptually easy, this is still computationally non-trivial as $f$ can have an enormous number of terms and require many iterations of the above procedure before ending up in standard form.  For example, putting the first five levels of the $\Z_3$-tower $\cT_{\timing}: Fy - y = [x^7] + [x^5]$ in standard form took a bit over three minutes. 
\end{remark}

\subsection{A Basis for Regular Differentials} \label{ss:maden} As before, let $\cT$ be a $\Z_p$-tower over a finite field $k$ of characteristic $p$ whose base is the projective line, totally ramified over infinity and unramified elsewhere. 
We identify $k(x)$ with the function field of $\cT(0)$, and present the tower in standard form as a sequence of Artin-Schreier extensions given by $y_n^p-y_n = f_n$.  Using the presentation of the tower in standard form, Madden's work \cite{madden} gives us an explicit basis for the space of regular differentials on $\cT(n)$.

\begin{definition}
Let $ \sS_n$ be the set of $n+1$-tuples of integers $(\nu,a_1,\ldots,a_n)$ such that:
\begin{enumerate}
    \item $0 \leq a_i <p$ for all $i$;

   \item  $\displaystyle 0 \leq p^n \nu \leq \left(\sum_{j=1}^n p^{n-j} d_\infty(\cT(j)) (p-1-a_j) \right) - p^n -1$.
\end{enumerate}
For $s = (\nu,a_1,\ldots,a_n) \in \sS_n$, define the differential
\[
\omega_s \colonequals x^\nu y_1^{a_1} \ldots y_n^{a_n} dx.
\]
\end{definition}
As $\cT$ is presented in standard form, \cite[Lemma 5]{madden} gives:

\begin{fact}
 The set $\{\omega_s : s \in \sS_n \}$ is a basis for $H^0(\cT(n),\Omega^1_{\cT(n)})$.
\end{fact}

\subsection{A Matrix for the Cartier Operator} 
We continue with the notation of Section~\ref{ss:maden}.  To represent the Cartier operator  on $H^0(\cT(n),\Omega^1_{\cT(n)})$
as a matrix, it suffices to compute
$V_{\cT(n)}(\omega_s)$ for $s \in \sS_n$.  As the Cartier operator is $p^{-1}$-semilinear, we have
\begin{align*}
V_{\cT(n)} \left( x^\nu y_1^{a_1} \ldots y_n^{a_n} dx \right) & = V_{\cT(n)} \left( x^\nu y_1^{a_1} \ldots y_{n-1}^{a_{n-1}} (y_n^p - f_n)^{a_n} dx \right) \\
&= \sum_{i=0}^{a_n} \binom{a_n}{i} y_n^i V_{\cT(n)} \left( x^\nu y_1^{a_1} \ldots y_{n-1}^{a_{n-1}} (- f_n)^{a_n-i} dx \right).
\end{align*}
Notice that  $x^\nu y_1^{a_1} \ldots y_{n-1}^{a_{n-1}} (- f_n)^{a_n-i} dx$ doesn't depend on $y_n$, and is a rational differential one-form on $\cT(n-1)$.  Thus we may compute $V_{\cT(n)}(\omega_s)$ by applying the Cartier operator to (several) differentials on $\cT(n-1)$.  This gives a recursive method to compute with $V_{\cT(n)}$, ultimately reducing to computing with the Cartier operator on the base curve, the projective line.

This is the most computationally expensive step of the algorithm.  The genus $g(\cT(n))$ grows (at least) exponentially with $n$, and we must evaluate the Cartier operator on $g(\cT(n))$ differentials.  Furthermore, the function $f_n$ can have an enormous number of terms; its order at the point above infinity is $-d_\infty(\cT(n))$, which is also growing (at least) exponentially in $n$.

\begin{remark} \label{rmk:xdfast}
The basic tower $Fy - y = [x^d]$ is always substantially faster to compute with precisely because the polynomials $f_n$ presenting the tower in standard form are (somewhat) simpler.  For example, building a matrix representing the Cartier operator on the fifth level of the $\Z_3$-tower given by $Fy - y = [x^4]$ takes less than fifteen minutes, while doing the same for $Fy -y = [x^4] + [x^2]$ takes more than three and a half hours.
\end{remark}

To implement the step described above, we first pre-compute
\begin{equation} \label{eq:precompute}
V_{\cT(m)} \left( x^\nu y_1^{a_1} \ldots y_m^{a_m} dx \right)
\end{equation}
with $0 \leq \nu <p$ and $0\leq a_i <p$ for $1 \leq i \leq m$ for $m=1,2,\ldots ,n$.  The computation at level $m$ makes use of the pre-computations at level $m-1$.  Note that the semilinearity of the Cartier operator allows us to compute $V_{\cT(m)}(\omega_s)$ as a $k(x)$-linear combination of these special values quickly.  
\begin{remark}
To give some context, for the basic $\Z_3$-tower $\cT_{\timing}: Fy - y = [x^7]+[x^5]$, the pre-computations up to level $5$ took
about six hours.  Once they are completed, it takes less than four minutes to use them to build a matrix for the Cartier operator on the 5-th level of the tower.  As $g(\cT_{\timing}(5)) = 51546$, this matrix has over $2.6$ {\em billion} entries!
\end{remark}

\begin{remark}
As all of the $p^{m+1}$ pre-computations of \eqref{eq:precompute} with $ 0\leq \nu <p$ and $0 \leq a_i <p$ for level $m$ can be performed independently using the results from level $m-1$, this step would be amenable to parallelization.  
\end{remark}

\subsection{Linear Algebra Over Finite Fields}
Linear algebra over small finite fields is very efficient in {\sc Magma}.  For a basic $\Z_3$-tower like $\cT_{\timing}: Fy - y = [x^7]+[x^5]$, whose fifth level has genus $51546$, computing the dimension of the kernel of the $51546$ by $51546$ matrix representing the Cartier operator on the space of regular differentials takes about a minute and a half.  Some of the matrices we consider, such as those in Section~\ref{ss:fastermonodromy}, are of course even larger, and consume many gigabytes of memory in storage.

\section{Theoretical Evidence} \label{sec:theoretical}

In this section, we study the interaction of the trace map on differential forms with the Cartier operator in Artin-Schreier extensions.  We use this to provide theoretical evidence for our conjectures about the $a$-number in characteristic two, in particular proving Conjecture~\ref{conj:basicanumber} for basic $\Z_2$-towers and more generally proving Conjecture~\ref{conj:stableanumber} for $\Z_2$-towers. 

The following fact is standard and will be used repeatedly (see for example \cite[Lemma 3.7]{bc20}).

\begin{fact} \label{fact:ordersofvanishing}
Let $\pi : Y \to X$ be a $\Z/p\Z$-cover of curves over a perfect field $k$ of characteristic $p$, 
with Artin-Schreier equation $y^p - y = \psi$.  If the defining equation is in standard
form at a branch point $Q$ with ramification invariant $d_Q$ ({i.e.}~$\ord_Q(\psi) = -d_Q$)
then a meromorphic differential $\displaystyle \omega = \sum_{i=0}^{p-1} \omega_i y^i$ on $Y$
is regular above $Q$ if and only if
$$\ord_Q(\omega_i) \geq - \left \lceil \frac{(p-1-i) d_Q}{p} \right \rceil \quad \text{ for } 0 \leq i \leq p-1. $$
\end{fact}

\subsection{Vanishing Trace}  Let $\pi:Y\rightarrow X$ be a $\Z/p\Z$-cover
with branch locus $S$, and for $Q \in S$ let $d_Q$ be the ramification invariant above $Q$.

Associated to the finite map $\pi$ is a canonical $\Oscr_X$-linear trace morphism 
$\pi_* \Omega^1_{Y/k}\rightarrow \Omega^1_{X/k}$ which is dual, via Grothendieck--Serre duality,
to the usual pullback morphism $\Oscr_X\rightarrow \pi_*\Oscr_Y$ on functions.
We will write $\pi_*$ for the induced trace map on global differential forms.  Note that the Cartier operator
is induced by the trace morphism $F_*$ attached to absolute Frobenius; since Frobenius
commutes with arbitrary ring maps, it in particular commutes with $\pi$
and it follows that the trace map commutes with the Cartier operator.
We will use the following formula repeatedly
\begin{equation}
    \pi^* \pi_* = \sum_{g\in \Z/p\Z} g^*.\label{eq:tracesum}
\end{equation}

In characteristic two, we can be very explicit about the kernel of the trace map:

\begin{lem} \label{lem:tracekernel}
If $p=2$, the kernel of $\pi_*$ on $H^0(Y,\Omega^1_{Y})$ is isomorphic to $H^0(X,\Omega^1_X(\sum_{Q \in S} \lceil d/2 \rceil [Q]))$.
\end{lem}

\begin{proof}
For $Q \in S$, locally express the cover as an Artin-Schreier extension $y^2 - y = f$ in standard form at $Q$
with $g^*y=y+1$ for the nontrivial element $g\in \Z/2\Z$.
A general meromorphic differential on $Y$ may be written as $\eta=\omega_0 + y \omega_1$ with $\omega_0, \omega_1$ meromorphic differentials on $X$.  By \eqref{eq:tracesum}, $\pi_*\eta=0$ forces $\omega_1=0$, and $\eta$ is regular at $Q$ if and only if $\ord_Q(\omega_0) \geq - \lceil d/2 \rceil $ (see Fact~\ref{fact:ordersofvanishing}).
\end{proof}

We next analyze the trace of differentials killed by the Cartier operator.

\begin{theorem} \label{thm:tracevanishing}
 If
 $\eta \in H^0(Y,\Omega^1_Y)$ is killed by $V_Y$, then for every branch point $Q \in S$ we have
     $\ord_Q(\pi_*(\eta)) \ge d_Q -\lceil d_Q/p \rceil$
    with strict inequality  when $d_Q \equiv \lfloor d_Q/p \rfloor \bmod p$.
\end{theorem}

\begin{proof}
    We may work locally at $Q$, where the extension is given by an Artin-Schreier equation $y^p-y = \psi$ with $d \colonequals d_Q = -\ord_Q(\psi) $.  We write $d=pq+r$ with $0 < r < p$, and decompose
    \[
        \eta = \sum_{i=0}^{p-1} \omega_i y^i
    \]
    with the $\omega_i$ differentials on $X$.
    If $\eta$ is regular above $Q$, Fact~\ref{fact:ordersofvanishing} implies $\ord_Q(\omega_i) \ge - \lceil (p-1-i)d/p\rceil$. 
    Furthermore, substituting $y=y^p-\psi$ in the expression for $\eta$ above, 
    and using the fact that $V$ is additive and $p^{-1}$-linear, 
    we compute (as in \cite[Lemma 4.1]{bc20}) that
    \[
        V_Y(\eta) = \sum_{i=0}^{p-1} \left( \sum_{j=i}^{p-1} \binom{j}{i} V_X(\omega_j(-\psi)^{j-i})\right) y^i.
    \]
    The assumption that $V_Y(\eta)=0$ implies that $V_X(\omega_{p-1})=0$ and that for $0\le i\le p-2$
    \begin{equation}
         V_X(\omega_{p-1}(-\psi)^{p-1-i}) = - \sum_{j=i}^{p-2}\binom{j}{i} V_X(\omega_j(-\psi)^{j-i}).
         \label{decomp}
    \end{equation}
    It is straightforward to check that the order of vanishing of the right side of \eqref{decomp} at $Q$
    is at least $\ord_Q(V_X(\omega_{p-2}(-\psi)^{p-2-i}))$, from which we deduce (replacing
    $i$ with $p-1-i$) that
    \begin{equation}
        \ord_Q(V_X(\omega_{p-1}(-\psi)^i)) \ge - \left\lceil\frac{(i-1)d + \lceil d/p \rceil}{p}  \right\rceil = -(i-1)q - \left\lceil \frac{(i-1)r + q +1}{p} \right\rceil.
        \label{eq:Vpoles}
    \end{equation}
    
    Let $u$ be a uniformizer at $Q$.  We may write $-\psi = cu^{-d}v$ with $c\in k^{\times}$ and $v$ a $1$-unit in the local ring at $Q$.  Working in the complete local ring at $Q$, as $p\nmid d$ Hensel's lemma implies there exists a $1$-unit $w$
    with $w^{-d}=v$, whence $-\psi = cu^{-d}w^{-d} = c z^{-d}$ with $z \colonequals uw$ a uniformizer at $Q$.
    Let $r_i$ be the least nonnegative residue of $ir$ modulo $p$, so that 
    $id = i(pq+r) = p(iq + \lfloor ri/p \rfloor) + r_i$; then
    \[
        V_X(\omega_{p-1}(-\psi)^i) = V_X(\omega_{p-1}c^iz^{-di}) = c^{i/p}z^{-iq-\lfloor ri/p \rfloor} V_X(\omega_{p-1}z^{-r_i}).
    \]
    We obtain
    \begin{equation}
            \ord_Q(V_X(\omega_{p-1}(-\psi)^i)) = -iq - \lfloor ri/p \rfloor + \ord_Q(V_X(\omega_{p-1}z^{-r_i})).
    \end{equation}
    Combining this with \eqref{eq:Vpoles} gives
    \[
        \ord_Q(V(\omega_{p-1}z^{-r_i})) \ge q +\left\lfloor \frac{ri}{p} \right\rfloor- \left\lceil \frac{(i-1)r + q +1}{p} \right\rceil.
    \]
    As $z$ is a uniformizer, the set $\{z^i\}_{0\le i< p}$ is a $p$-basis for $\widehat{\mathscr{O}}_{X,Q}\simeq k[\![z]\!]$, so we may write
    \begin{equation}
        \omega_{p-1} = (f_1^p z + f_2^p z^2 + \cdots + f_{p-1}^p z^{p-1} + f_p^p z^p) \frac{dz}{z}
    \end{equation}
    where $f_i$ are local functions.  Since $V_X(\omega_{p-1})=0$, we have $f_p=0$, and we compute
    \[
        V(\omega_{p-1} z^{-r_i}) = f_{r_i} \frac{dz}{z}.
    \]
    Therefore we conclude that for $1\le i \le p-1$
    \begin{equation}
        \ord_Q(f_{r_i}) - 1 \ge q +\left\lfloor \frac{ri}{p} \right\rfloor- \left\lceil \frac{(i-1)r + q +1}{p} \right\rceil.
    \end{equation}
    Let $s_i$ be the least nonnegative residue of $(i-1)r + q + 1$ modulo $p$, so that 
    \[
       p \left\lceil \frac{(i-1)r + q +1}{p} \right\rceil  = (i-1)r+q+1 - s_i + \begin{cases}  p & s_i\neq 0 \\ 0 & s_i =0.\end{cases}
    \]
    We find 
    \begin{align*}
        \ord_Q(f_{r_i}^p z^{r_i} \frac{dz}{z}) &\ge pq +p\left(\left\lfloor \frac{ri}{p} \right\rfloor- \left\lceil \frac{(i-1)r + q+1}{p} \right\rceil+1\right)+ ri - p\left\lfloor\frac{ri}{p}\right\rfloor - 1 \\
        & = pq + ri +p - 1 - p\left\lceil \frac{(i-1)r + q+1}{p} \right\rceil \\
        & = pq + ri +p - 1  - (i-1)r - (q+1) + s_i -  \begin{cases}  p & s_i\neq 0 \\ 0 & s_i =0\end{cases}\\
        &= (pq+r) - (q+1) + p-1+s_i -\begin{cases}  p & s_i\neq 0 \\ 0 & s_i =0\end{cases} \\ 
        &= d - \lceil d/p \rceil  + \begin{cases}  s_i - 1 & s_i\neq 0 \\ p-1 & s_i =0.\end{cases}
    \end{align*}
    If $q\not\equiv r \bmod p$, we claim there exists $i$ with $1\le i \le p-1$
    and $s_i = 1$.  Indeed, $i = (1 - q r^{-1} \bmod p)$ does the trick.  
    On the other hand, if $q\equiv r\bmod p$, then $(i-1)r + q \equiv i r \bmod p$,
    which is never $0 \bmod p$, so that $s_i\neq 1$ for all $i$ with $1\le i \le p-1$ in this case.
    We conclude that 
    \[
        \ord_Q(\omega_{p-1}) \ge d - \left\lceil \frac{d}{p} \right \rceil,   
    \]
    and that the inequality is strict if $q\equiv r \bmod p$, or what is the same, if $\lfloor d/p \rfloor \equiv d\bmod p$.
\end{proof}

\begin{cor}\label{cor:traceeq0}
   Suppose that  
    $\displaystyle \sum_{Q\in S} (d_Q -\lceil d_Q/p\rceil) \ge 2g(X)-2$,
    with strict inequality when $d_Q\not\equiv \lfloor d_Q/p\rfloor \bmod p$ for all $Q \in S$.
    If $\eta\in H^0(Y,\Omega^1_Y)$ is killed by $V_Y$ then $\pi_*(\eta)=0$.
\end{cor}

\begin{proof}
As the differential $\pi_*(\eta)$ is regular when $\eta$ is, the corollary follows immediately from
the fact that for an effective divisor $D$, one has 
$H^0(X,\Omega^1_X(-D))=0$ whenever $\deg(D) > 2g-2$.
\end{proof}

Finally for a $\Z_p$-tower $\cT$  totally ramified over a non-empty set $S$ we investigate the hypothesis
\begin{equation}
    \sum_{Q\in S} (d_Q(\cT(n+1)) -\lceil d_Q(\cT(n+1))/p\rceil) > 2g(\cT(n))-2
    \label{eq:ramhyp}
\end{equation}
 For convenience, we define
 \[
 \Delta_n \colonequals   \sum_{Q\in S} (d_Q(\cT(n+1)) -\lceil d_Q(\cT(n+1))/p\rceil) - (2g(\cT(n))-2).
 \]

\begin{lemma} \label{lem:technicalramhypothesis}
Suppose that exists an integer $N$ such that 
\begin{equation}
\sum_{Q \in S} (s_Q(\cT(j+1)) - 2s_Q(\cT(j)) \ge 2 g(\cT(0))-2 + \# S\quad\text{for all}\ j>N.
\label{techhyp}
\end{equation}
If \eqref{techhyp} is an equality for all $j>N$, assume moreover that 
$\Delta_N > 0$. 
Then $\cT$ satisfies \eqref{eq:ramhyp} for $n \gg 0$.
\end{lemma}

\begin{proof}
As we only deal with one tower in this proof and its corollary, to simplify notation we will let $d_{Q,n}$ (resp. $s_{Q,n}$, $g_n$) denote $d_Q(\cT(n))$ (resp. $s_Q(\cT(n))$ and $g(\cT(n))$).  
From Lemma \ref{lem:breaks}(\ref{break3}) we get for $Q \in S$ that 
\[
 d_{Q,n+1} - \lceil d_{Q,n+1}/p \rceil  = (s_{Q,n+1}  - s_{Q,n}) \varphi(p^n) + d_{Q,n} - \lceil d_{Q,n}/p \rceil, 
\]
and by induction that 
\[
  d_{Q,n+1} - \lceil d_{Q,n+1}/p \rceil = \sum_{j=N}^n (s_{Q,j+1} - s_{Q,j} )\varphi(p^j) + d_{Q,N} - \lceil d_{Q,N}/p \rceil.
\]
From Lemma \ref{lem:genuslower} we  obtain 
\[
(2g_n- 2) = p^n(2 g_0 -2) + \# S (p^n-1) + \sum_{Q \in S} \sum_{j=1}^n \varphi(p^j) s_{Q,j} .
\]
Therefore we conclude that for $n > N$
\begin{equation}
       \Delta_n=
   \Delta_N + \sum_{j=N+1}^n \sum_{Q\in S} (s_{Q,j+1} - 2 s_{Q,j}) \varphi(p^j)
   - (p^n-p^N)(2g_0 - 2 + \#S).
 \end{equation}
    If $c$ is any constant with 
    $\sum_{Q \in S} (s_{Q,j+1}- 2s_{Q,j}) \ge c$ for all $j>N$ then we obtain
   \[
   \Delta_n \geq  \Delta_N + (p^n-p^N) ( c - (2g_0 -2 + \# S)).
   \]
    Our hypotheses ensure that we may take $c \ge 2g_0 -2 + \# S$, and in the case of equality, that $\Delta_N > 0$,
    so it follows that $\Delta_n > 0$ for all $n$ sufficiently large; {\em i.e.}~\eqref{eq:ramhyp} is satisfied for $n \gg 0$.
\end{proof}

\begin{corollary} \label{cor:ramhyp}
If $p>2$ then \eqref{eq:ramhyp} is satisfied for $n \gg 0$.

If $p=2$, suppose that $\cT$ is monodromy stable with $s_Q(\cT(n)) = d_Q p^{n-1} + c_Q$ for $n \gg 0$ for each $Q \in S$.  Then \eqref{eq:ramhyp} is satisfied for $n \gg 0$ provided
\[
\sum_{Q \in S} (-c_Q) > 2g_0 -2 + \# S. 
\]
\end{corollary}

\begin{proof}
If $p>2$, then $s_{Q,j+1} - 2 s_{Q,j} \geq s_{Q,j}$ as $s_{Q,j+1} \geq p s_{Q,j}$.  Hence $\sum_{Q \in S} (s_{Q,n+1}  - 2 s_{Q,n})$ is larger than $2 g_0-2 + \# S$ for $n$ sufficiently large. 

If $p=2$ and $\cT$ is monodromy stable, we compute that $s_{Q,j+1} - 2 s_{Q,j} = - c_Q$ (note that we must have $c_Q \leq 0$ as $s_{Q,j+1} \geq 2 s_{Q,j}$).   The claim follows from Lemma~\ref{lem:technicalramhypothesis}.
\end{proof}

In particular, notice that \eqref{eq:ramhyp} holds for basic towers over the projective line in characteristic two as the genus of the base curve is $0$, the tower is ramified only over infinity, and $c_\infty = 0$.  
 It is {\em also} satisfied for any $\Z_2$-tower over $\PP^1$ with $\#S=2$,
since $\Delta_0 > 0$ and \eqref{techhyp} holds automatically as the right side is 0 and the left side is nonnegative.

\begin{remark}
Consider a sequence of positive integers $\{s_n\}$ such that $p \nmid s_0$, $s_{n+1} \geq p s_n$, and whenever $p$ divides $s_{n+1}$ we have $s_{n+1} = p s_n$.  Then using Fact~\ref{fact:localasw} we can construct a local Artin-Schreier-Witt extension such that the breaks in the upper ramification filtration are $s_n$.  
This shows there is a large variety of potential ramification behavior in $\Z_p$-towers.  In light of this, Lemma~\ref{lem:technicalramhypothesis} shows that \emph{not} satisfying condition \eqref{eq:ramhyp} for $n \gg 0$ is a very restrictive hypothesis on the ramification of a tower.  

For example, consider monodromy stable towers over a fixed base with fixed branch locus $S$.  Writing $s_Q(\cT(n)) = d_Q p^{n-1} + c_Q$ for $Q \in S$, if we fix each $d_Q \in \Q$ there are finitely many choices of $\{c_Q\}_{Q \in S}$ for which the tower does {\em not} satisfy \eqref{eq:ramhyp} for $n \gg 0$.  Using Corollary~\ref{cor:ramhyp}, this is because we must have $c_Q \leq 0$, and for fixed $d_Q$ the requirement that $d_Q p^{n-1} + c_Q \in \Z$ for $n \geq 1$ gives a bound on the denominator of $c_Q$.  
\end{remark}

\subsection{\texorpdfstring{$a$}{a}-numbers in Characteristic Two} \label{ss:anumberproof}

\begin{notation}
Let $C$ be a curve over a perfect field $k$ of characteristic $p$.  Given an effective divisor $D$ on $C$,  we let $a^r(\Omega^1_C(D))$ denote the dimension of the kernel of $V_C^r$ on $H^0(C,\Omega^1_C(D))$.  We use $a(\Omega^1_C(D))$ as a shorthand for $a^1(\Omega^1_C(D))$. 
\end{notation}

We now specialize to working over a field of characteristic $p=2$, where we can compute the $a$-number in a cover using the base curve.

\begin{prop} \label{prop:anumberbase}
Suppose $\pi : Y \to X$ is a $\Z/2 \Z$-cover totally ramified over $S \subset X$.  For $Q \in S$, let $d_Q$ be the ramification invariant above $Q$.  If 
\begin{equation} \label{eq:hypothesis}
    \sum_{Q \in S} (d_Q-1)/2 \geq 2 g(X)-2,
\end{equation}
with strict inequality when $d_Q \equiv 1 \pmod{4}$ for all $Q \in S$, 
then
\[
a(Y) = a \left(\Omega_X^1 \left( \sum_{Q \in S} \frac{d_Q+1}{2} [Q] \right) \right).
\]
\end{prop}

\begin{proof}
The $a$-number of $Y$ is the dimension of the kernel of the Cartier operator on $H^0(\Omega^1_Y)$.  By Corollary~\ref{cor:traceeq0}, this is a subspace of the kernel of the trace map, and so by Lemma~\ref{lem:tracekernel}
\[
a(Y) = a\left(\Omega_X^1\left(\sum_{Q \in S} \lceil d_Q/2 \rceil [Q] \right) \right) = a \left(\Omega_X^1 \left( \sum_{Q \in S} \frac{d_Q+1}{2} [Q] \right) \right). \qedhere
\]
\end{proof}

\begin{cor} \label{cor:anumberformula}
With the notation and hypothesis of Proposition~\ref{prop:anumberbase}, we have
\begin{equation} \label{eq:anumberformula}
    a(Y) = \sum_{\substack{Q \in S \\ d_Q \equiv 1 \, \smallmod{4}}} \frac{d_Q-1}{4} + \sum_{\substack{Q \in S \\ d_Q \equiv 3 \, \smallmod{4}}} \frac{d_Q +1}{4}.
\end{equation}
\end{cor}

\begin{proof}
By definition, the {\em Tango number} of $X$ is
\[
   \operatorname{n}(X):=\max\left\{\sum_{x\in X(\overline{k})} \left\lfloor \frac{\ord_x(df)}{p}\right\rfloor \ :\ f\in \overline{k}(X)-\overline{k}(X)^p\right\}.
\]
In \cite{tango72}, Tango proves that whenever $D$ is a divisor on $X$ with $\deg D > \operatorname{n}(X)$,
the pullback map along absolute Frobenius $F_X^*: H^1(X,\Oscr_X(-D))\rightarrow H^1(X,\Oscr_X(-pD))$
is {\em injective}.  Applying Grothendieck--Serre duality, the Cartier operator
 $V_X: H^0(X,\Omega^1_X(pD))\twoheadrightarrow H^0(X,\Omega^1_X(D))$ is then {\em surjective}
for such $D$.  When $\deg(D)>0$, the Riemann--Roch formula thereby yields an exact formula
for the dimension of the kernel of $V_X$ on $H^0(X,\Omega^1_X(pD))$, which
may be parlayed into a formula for the dimension of the kernel of $V_X$
on $H^0(X,\Omega^1_X(D'))$ for any $D'$ of sufficiently large degree;
see \cite[Corollary 6.13]{bc20} for the precise statement.
As $p=2$, we have $\lfloor (2g(X) -2 )/p \rfloor = g(X)-1 \geq \operatorname{n}(X)$
thanks to \cite[Lemma 10]{tango72}, and the hypothesis \eqref{eq:hypothesis}
ensures that the divisor  $D':=\sum_{Q \in S} \frac{d_Q+1}{2} [Q]$ has large enough 
degree to apply \cite[Corollary 6.13]{bc20}, whereby Proposition \ref{prop:anumberbase}
yields the stated exact formula for $a(Y)$.
\end{proof}

\begin{remark}\label{Voloch}
This formula was already known to hold when $X$ is ordinary \cite[Theorem 2]{VolochChar2} without needing the hypothesis of equation \eqref{eq:hypothesis}.
\end{remark}

For a $\Z_2$-tower $\cT$ we may apply Corollary~\ref{cor:anumberformula} to compute $a(\cT(n))$ 
for $n \gg 0$ in terms of the ramification of the tower, assuming a mild technical hypothesis on the ramification (recall Lemma~\ref{lem:technicalramhypothesis} and Corollary~\ref{cor:ramhyp}).  This is exactly as we would expect based on Philosophy~\ref{philosophy}.  Since the ramification may be quite poorly behaved (see Remark~\ref{remark:lowergrowth}), while the a-number in $\cT$ is ``regular'' in the sense that it depends on the ramification breaks of the tower, the resulting formula (like the general Riemann--Hurwitz formula
of Lemma \ref{lem:genuslower}) may not be especially simple.  Of course, for towers 
whose ramification breaks behave in a regular manner, the $a$-number---like the genus---will admit a simple formula.

\begin{cor} \label{cor:anumberbasic}
Let $\cT$ be a basic $\Z_2$-tower with ramification invariant $d$.  Then for $n >1$
\begin{equation}
    a(\cT(n)) = \begin{cases}
     \frac{d}{24} 2^{2n} + \frac{d+3}{12} & d \equiv 1 \pmod{4}\\
     \frac{d}{24} 2^{2n} + \frac{d-3}{12} & d \equiv 3 \pmod{4}
    \end{cases}
\end{equation}
which proves Conjecture~\ref{conj:basicanumber} when $p=2$.  More concisely,
\[
a(\cT(n)) = \frac{d}{24}( 2^{2n} - 4) + a(\cT(1)) -\frac{1}{2}= \frac{d}{6} (2^{2(n-1)}-1) + a(\cT(1)) -\frac{1}{2}.
\]
\end{cor}
 
\begin{proof}
By Corollary~\ref{cor:ramhyp}, basic towers satisfy the hypothesis \eqref{eq:ramhyp}.  Then combine Corollary~\ref{cor:anumberformula} with Lemma~\ref{lem:basicinvariants}, and note that $a(\cT(1)) = (d-1)/4$ if $d \equiv 1 \pmod{4}$ and $a(\cT(1)) = (d+1)/4$ if $d \equiv 3 \pmod{4}$.
\end{proof}

\begin{cor} \label{cor:anumberstable}
Let $\cT$ be a monodromy stable $\Z_2$-tower totally ramified over $S \subset \cT(0)$, so for $Q \in S$ we have $s_Q(\cT(n)) = c_Q + d_Q p^{n-1}$ for $n \gg 0$.  Suppose that $\sum_{Q \in S} (-c_Q)  > 2 g(\cT(0))-2 + \# S$.
Then there exists $a,c \in \bQ$ such that $a(\cT(n)) = a 2^{2n} + c$ for $n \gg 0$,
and we may take $\displaystyle a = \sum_{Q \in S} \frac{d_Q}{24}$.
\end{cor}

This proves Conjecture~\ref{conj:stableanumber} when $p=2$ and Conjecture~\ref{conj:stableasymptotic} when $p=2$ and $r=1$ under the additional technical assumption that $\sum_{Q \in S} (-c_Q)  > 2 g(\cT(0))-2 + \# S$.

\begin{proof}
Combine Lemma~\ref{lem:dsinstable} with Corollary~\ref{cor:anumberformula}, and note that the hypotheses in the latter are automatically satisfied for $n \gg 0$.
\end{proof}

\begin{example} \label{ex:igusap2}
We apply this to the Igusa tower $\Ig$ in characteristic two, working over $k= \overline{\F}_2$.  We rigidify as in Example~\ref{ex:igusaproperties}  by adding an additional $\Gamma_1(5)$-level structure, and obtain 
a $\Z_2$-tower
\[
\ldots \to \Ig(3) \to \Ig(2) \to  \Ig(1) \simeq \PP^1_k
\]
totally ramified over the unique supersingular point of $\Ig(1)$ and unramified elsewhere, with $g(\Ig(n)) = 2^{2n-2}-2^{n}+1$ and $d(\Ig(n))  = 2^{2(n-1)}-1$.  As there is a single point of ramification, the technical hypothesis \eqref{eq:ramhyp} holds (Corollary~\ref{cor:ramhyp}), so applying Corollary~\ref{cor:anumberformula} we obtain $a(\Ig(n)) = 2^{2n-4}$ for $n >1$.
\end{example}

\begin{remark}
Example~\ref{ex:p2failtechnical} and Example~\ref{ex:hyperellipticbasep2} look at examples of monodromy stable $\Z_2$-towers which don't satisfy the technical inequality in Proposition~\ref{prop:anumberbase}.  They still appear to satisfy the conclusions of Corollary~\ref{cor:anumberstable}, although not always the precise formulas given by Corollary~\ref{cor:anumberformula}.
\end{remark}

\subsection{Powers of the Cartier Operator} \label{ss:otherproof}  The previous techniques do not suffice to compute $\dk{r}{\cT(n)}$ for a $\Z_2$-tower when $r>1$.    We can currently only prove the following limited lemma.
 
\begin{lemma} \label{lem:p2higher}
Let $\cT$ be a $\Z_2$-tower with $\cT(0)=\PP_k^1$ and branch locus $S$.
\begin{enumerate}
    \item \label{p2higher:1} Writing $D = \sum_{Q \in S} \frac{d_Q(\cT(1))+1}{2} [Q]$, for any $r \geq 1$ we have
\[
\dk{r}{\cT(1)} = a^r ( \Omega^1_{\PP^1} (D) ) = \deg (D) - \sum_{Q \in S} \left \lceil \frac{d_Q(\cT(1))+1}{2^{r+1}} \right \rceil.
\]

\item  Suppose $\cT$ furthermore satisfies $d_Q(\cT(2)) = 3 d_Q(\cT(1)) $ for all $Q \in S$ and that 
\begin{equation} \label{eq:p2higherhyp}
     \sum_{Q \in S} \frac{d_Q(\cT(1)) -3}{2} > -4.
\end{equation}
If $Q(1)$ is the unique point of $\cT(1)$ over $Q \in S$ and $D' = \sum_{Q \in S} \frac{3d_Q(\cT(1)) +1}{2} [Q(1)]$
then $$\dk{2}{\cT(2)} = a^2 ( \Omega^1_{\cT(1)}( D' )) =  \sum_{Q \in S} \left( \left \lfloor \frac{3d_Q +1}{4} \right \rfloor + \left \lfloor \frac{7d_Q+7}{16} \right \rfloor \right) .$$
\end{enumerate}
\end{lemma}

The hypothesis $d_Q(\cT(2)) = 3 d_Q(\cT(1))$ says that $d_Q(\cT(2))$ is as small as possible (see Remark~\ref{remark:lowergrowth}) and is the behavior seen in basic $\Z_2$-towers.  The inequality \eqref{eq:p2higherhyp} is satisfied unless there are a large number of $Q \in S$ with $d_Q =1$.  The expression involving floor functions avoids a large number of case-by-case formulas depending on $d_Q$ modulo $16$.  

\begin{proof}  We may assume that $k$ is algebraically closed as $a^r(\cT(n))$ is independent of extension of scalars.  
As $\cT(0) = \PP^1_k$, we may represent the extension of function fields corresponding to $\cT(1) \to \cT(0)$ as an Artin-Schreier extension $y_1^2 - y_1 = f_1$ in standard form (recall Definition~\ref{defn:standardform}).  Then every meromorphic differential on $\cT(1)$ may be written $\omega = \omega_0 + y_1 \omega_1$ with $\omega_0, \omega_1$ meromorphic differentials on $\PP^1_k$; if $\omega$ is regular then  $\omega_1$ is a differential on $\PP^1_k$ without poles by Fact \ref{fact:ordersofvanishing}.  That is,  $\omega_1=0$ and we conclude that $H^0(\cT(1),\Omega^1_{\cT(1)}) = H^0(\PP^1_k, \Omega^1_{\PP^1} ( D))$.  
The formula then follows from the usual, explicit description of $H^0(\PP^1_k, \Omega^1_{\PP^1} ( D))$ using partial fractions
and a straightforward calculation with the Cartier operator.

For the second assertion, write the extension of functions fields corresponding to $\cT(2) \to \cT(1)$  as $y_2^2 - y_2 = f_2$ with $f_2$ a function on $\cT(1)$ and note that the hypothesis of Corollary~\ref{cor:traceeq0} holds for $\cT(1) \to \cT(0)$ by inspection.  It also holds for $\cT(2) \to \cT(1)$ using hypothesis \eqref{eq:p2higherhyp} as
\[
\sum_{Q \in S} d_Q(\cT(2)) - \lceil d_Q(\cT(2))/2 \rceil = \sum_{Q \in S} \frac{3 d_Q(\cT(1)) -1}{2} \quad \text{and} \quad 2g(\cT(1))-2 = -4 + \sum_{Q \in S} (d_Q(\cT(1)) + 1).
\]
By Fact~\ref{fact:standardform}, we may assume the functions $(f_1,f_2)$ present $\cT(2)\rightarrow \cT(0)$ in standard form,
or what is the same that  $\ord_{Q(1)} f_2 = -3 d_Q(\cT(1))$ for all $Q \in S$. For $\omega = \omega_0 + y_2 \omega_1 \in H^0(\Omega^1_{\cT(2)})$ in the kernel of $V_{\cT(2)}^2$, we know that $\pi_*(V_{\cT(2)}(\omega))=0$ by Corollary~\ref{cor:traceeq0}.  We compute that 
\[
V_{\cT(2)} (\omega_0 + y_2 \omega_1) = V_{\cT(2)} (\omega_0 +  (y_2^2 + f_2) \omega_1) = V_{\cT(1)} (\omega_0 + f_2 \omega_1) + y_2 V_{\cT(1)}(\omega_1)
\]
and thus $V_{\cT(1)}(\omega_1) =0$.  Again using Corollary~\ref{cor:traceeq0}, we conclude that $\pi_*(\omega_1) =0$; in other words, $\omega_1$ is the pullback of an element (also denoted $\omega_1$) of $H^0(\Omega^1_{\cT(0)} (D))$.  We also obtain
\begin{equation} \label{eq:v2simp}
V_{\cT(2)}^2(\omega) = V^2_{\cT(1)}(\omega_0 + f_2 \omega_1) =0.
\end{equation}

Suppose $\omega_1 \neq 0$.  We know that $\ord_{Q(1)}(\omega_1)$ is even as $V_{\cT(1)}(\omega_1) =0$ (consider the local expansion at $Q(1)$).  A small calculation shows that $\ord_{Q(1)}(\omega_1) \equiv d_Q(\cT(1)) +1 \pmod{4}$ as $\omega_1$ is the pullback of a differential on $\cT(0)$.  As $\ord_{Q(1)}(f_2) = - 3d_Q(\cT(1))$, we conclude that $\ord_{Q(1)} (f_2 \omega_1) \equiv -3d_{Q}(\cT(1)) + d_Q(\cT(1)) + 1 \equiv 3 \pmod{4}$ as $d_Q(\cT(1))$ is odd.  By considering the local expansion at $Q(1)$, we conclude that $\ord_{Q(1)} ( V_{\cT(1)}^2(f_2 \omega_1)) = (-3d_Q(\cT(1)) -3 + \ord_{Q(1)} (\omega_1))/4$.  By Fact~\ref{fact:ordersofvanishing}, we know that $\ord_{Q(1)}(\omega_0) \geq - \frac{3 d_Q(\cT(1)) +1}{2}$ and hence using \eqref{eq:v2simp} we conclude that
\[
-3 d_Q(\cT(1))+ \ord_{Q(1)} (\omega_1)  \geq - \frac{3 d_Q(\cT(1)) +1}{2}.
\]
Summing over $Q \in S$, we conclude that 
\[
\deg(\omega_1) \geq \sum_{Q \in S} \frac{3 d_Q(\cT(1)) -1}{2}.
\]
But as this is larger than $2 g(\cT(1)) -2  =  -4 + \sum_{Q \in S} (d_Q(\cT(1)) +1)$ by \eqref{eq:p2higherhyp}, there are no non-zero differentials of this degree.  Thus $\omega_1 =0$ and $\omega$ is the pullback of a global section of $\Omega^1_{\cT(1)}(D')$ by Fact~\ref{fact:ordersofvanishing}. We therefore conclude that
\[
\dk{2}{\cT(2)} = a^2 ( \Omega^1_{\cT(1)}( D' )).
\]

It remains to compute $ a^2 ( \Omega^1_{\cT(1)}( D' ))$.  Set 
$$
D'':=\sum_{\substack{Q \in S \\ d_Q \equiv 1 \, \smallmod{4}}} \frac{3d_Q+1}{4}[Q] + \sum_{\substack{Q \in S \\ d_Q \equiv 3 \, \smallmod{4}}} \frac{3d_Q -1}{4}[Q]\quad \text{and}\quad R:=\sum_{\substack{Q \in S \\ d_Q \equiv 3 \, \smallmod{4}}}[Q]
$$
so that $D'=2D'' + R$.  Observe that $\deg D'' > g(\cT(1)) -1$ as
\[
    \sum_{\substack{Q \in S \\ d_Q \equiv 1 \, \smallmod{4}}} \frac{3d_Q+1}{4} + 
    \sum_{\substack{Q \in S \\ d_Q \equiv 3 \, \smallmod{4}}} \frac{3d_Q -1}{4}
    >  - 2 + \sum_{Q \in S}  \frac{d_Q + 1}{2}.
\]
Thus by Tango's theorem \cite[Theorem 15]{tango72} we conclude
\[
  V_{\cT(1)} : H^0(\Omega^1_{\cT(1)}(D')) \rightarrow H^0(\Omega^1_{\cT(1)}(D''+ R))
\]
is surjective.  As we know the dimension of the domain and codomain, the kernel of this map has dimension $\deg(D'')$ and we conclude that
\begin{equation} \label{eq:a2}
     a^2 ( \Omega^1_{\cT(1)}( D' ))  = \deg(D'')  + a^1 ( \Omega^1_{\cT(1)}( D''+R )).
\end{equation}
Thus we are reduced to studying the kernel of the Cartier operator on $\cT(1)$.

Consider a rational differential $\omega=\omega_0 + \omega_1 y_1$ on $\cT(1)$ with $\omega_0,\omega_1$ rational on $\PP^1$.  If $V_{\cT(1)}(\omega)=0$ then $V_{\cT(0)}(\omega_0 + f_1\omega_1) + y_1 V_{\cT(0)}(\omega_1)=0$, and we get that
\[
V_{\cT(0)}(\omega_1)= 0 \quad \text{and} \quad V_{\cT(0)}(\omega_0) = V_{\cT(0)}(f_1 \omega_1) .
\]
Note the condition $\ord_{Q(1)}\omega \ge -n$ is equivalent to $\ord_Q \omega_1 \ge -\lceil n/2 \rceil$ and $\ord_Q \omega_0 \ge -\lceil (n+d_Q)/2\rceil$. 
If $\omega \in H^0(\Omega^1_{\cT(1)}(D''+R))$ then we claim $\omega_1=0$.  
As $V_{\cT(0)}(\omega_1)=0$ we have that $\ord_{Q}( \omega_1) =2 m$ is even for $Q \in S$, and hence $\ord_Q(V_{\cT(0)}(f_1\omega_1)) = -\frac{d_Q+1}{2} + m$.  On the other hand, from the definition of $D''+R$ we deduce that $\ord_Q(\omega_0) \geq - \lceil 7 d_Q/8 \rceil$.  (Throughout we implicitly verify various simplifications of floor and ceiling functions by checking them for all congruences classes of $d_Q$ modulo the denominator.)  Therefore we see that for $Q \in S$
\[
 \ord_Q V_{\cT(0)} (\omega_0) \ge - \left \lceil \frac{1}{2} \left\lceil  \frac{7d_Q}{8}\right \rceil\right \rceil .
\]
The requirement that $V(\omega_0) = V(f_1\omega_1)$ then forces
\[
    m \ge \frac{d_Q+1}{2} - \left \lceil \frac{1}{2} \left\lceil  \frac{7d_Q}{8}\right \rceil\right \rceil \geq 0;
\]
note that to check the last inequality it suffices to check it for $d_Q < 16$.  Therefore $\omega_1$ is regular on $\cT(0) = \PP^1$ and hence $\omega_1 = 0$ as claimed.
This implies that
\[
a^1 ( \Omega^1_{\cT(1)}( D''+R )) = a^1\left(\Omega^1_{\PP^1}(\sum_{Q \in S} \lceil 7 d_Q/8\rceil [Q]) \right)
= \sum_{Q \in S} \lfloor \frac{1}{2}\lceil 7d_Q/8\rceil\rfloor.
\]
Combining this with equation \eqref{eq:a2} gives that 
\[
    a^2(\cT(2)) = \deg(D'') + \sum_{Q \in S} \lfloor \frac{1}{2}\lceil 7d_Q/8\rceil\rfloor
    = \sum_{Q \in S} \left( \left \lfloor \frac{3d_Q +1}{4} \right \rfloor + \left \lfloor \frac{7d_Q+7}{16} \right \rfloor \right)
\]
where again we verify the simplifications of the floor functions by checking on congruence classes of $d$ modulo $16$.  
\end{proof}

This is the limit of what can be shown using just ramification information for the tower.  In Example ~\ref{ex:p2d21}, we saw basic $\Z_2$-towers $\cT$ and $\cT'$ over the projective line  with identical ramification (which satisfy the hypotheses and conclusions of Lemma~\ref{lem:p2higher}) such that $\dk{2}{\cT(3)} \neq \dk{2}{\cT'(3)}$ and $\dk{3}{\cT(2)} \neq \dk{3}{\cT'(2)}$.

\newcommand{\etalchar}[1]{$^{#1}$}
\providecommand{\bysame}{\leavevmode\hbox to3em{\hrulefill}\thinspace}
\providecommand{\MR}{\relax\ifhmode\unskip\space\fi MR }
\providecommand{\MRhref}[2]{%
  \href{http://www.ams.org/mathscinet-getitem?mr=#1}{#2}
}
\providecommand{\href}[2]{#2}

\end{document}